\documentclass[a4paper, 9pt]{article}

\usepackage{latexsym,amsmath,amssymb,amsfonts,amsthm,amstext,verbatim,graphics,color}
\usepackage[all]{xy}
\usepackage[mathscr]{euscript}
\usepackage{yfonts}
\usepackage[dvipdfmx]{hyperref}

\numberwithin{equation}{section}
\newtheorem{thm}{Theorem}[section]
\newtheorem{prop}[thm]{Proposition}
\newtheorem{lem}[thm]{Lemma}
\newtheorem{cor}[thm]{Corollary}
\newtheorem{claim}{Claim}{\bf}{\it}

\newtheorem{fthm}{Theorem}{\bf}{\it}
{\bf}{\it}
\newtheorem{fcor}[fthm]{Corollary}{\bf}{\it}
{\bf}{\it}
{\bf}{\it}

\theoremstyle{definition}
\newtheorem{defn}[thm]{Definition}

{\bf}{\rm}

\theoremstyle{remark}

\newtheorem{rem}[thm]{Remark}
{\bf}{\it}

\newtheorem{definition and corollary}[thm]{Definition and Corollary}

{\it}{\rm}

\newcommand{\A}{{\mathbb A}}
\newcommand{\Aff}{A\!f\!\!f}
\newcommand{\af}{\mathrm{af}}
\newcommand{\al}{\alpha}
\newcommand{\bB}{{\mathbf B}}

\newcommand{\bP}{{\mathbb P}}
\newcommand{\bi}{{\mathbf i}}
\newcommand{\C}{{\mathbb C}}
\newcommand{\cO}{{\mathcal O}}

\newcommand{\bK}{{\mathbb K}}

\newcommand{\Hom}{\mbox{\rm Hom}}

\newcommand{\ch}{\mathrm{ch}}

\newcommand{\Fr}{\mathsf{Fr}}

\newcommand{\tI}{\mathtt I}
\newcommand{\tJ}{\mathtt J}

\newcommand{\la}{\lambda}

\newcommand{\Spec}{\mbox{\rm Spec}}

\newcommand{\gch}{\mathrm{gch}}
\newcommand{\F}{\mathbb F}
\newcommand{\g}{\mathfrak{g}}
\newcommand{\gb}{\mathfrak{b}}

\newcommand{\h}{\mathfrak{h}}

\newcommand{\wth}{\widetilde{\mathfrak{h}}}

\newcommand{\gI}{\mathfrak{I}}
\newcommand{\bI}{\mathbf{I}}
\newcommand{\gm}{\mathfrak{m}}
\newcommand{\gn}{\mathfrak{n}}

\newcommand{\tg}{\widetilde{\mathfrak{g}}}

\newcommand{\bv}{\mathbf{v}}
\newcommand{\bw}{\mathbf{w}}
\newcommand{\bu}{\mathbf{u}}
\newcommand{\bO}{\mathbb{O}}

\renewcommand{\P}{\mathbb{P}}
\newcommand{\bQ}{\mathbf{Q}}

\newcommand{\sB}{\mathscr{B}}
\newcommand{\sGB}{\mathscr{GB}}
\newcommand{\sQ}{\mathscr{Q}}
\newcommand{\bX}{\mathbb{X}}

\newcommand{\bW}{\mathbb{W}}

\newcommand{\gX}{\mathfrak{X}}
\newcommand{\gY}{\mathfrak{Y}}
\newcommand{\Q}{\mathbb{Q}}
\newcommand{\R}{\mathbb{R}}
\newcommand{\ra}{\mathrm{rat}}

\newcommand{\Z}{\mathbb{Z}}
\newcommand{\sZ}{\mathscr{Z}}
\newcommand{\si}{\frac{\infty}{2}}
\newcommand{\SL}{\mathop{SL}}

\newcommand{\Ga}{\mathbb G_a}
\newcommand{\Gm}{\mathbb G_m}

\title{Frobenius splitting of Schubert varieties of semi-infinite flag manifolds}
\author{Syu \textsc{Kato}\footnote{Department of Mathematics, Kyoto University, Oiwake Kita-Shirakawa Sakyo Kyoto 606-8502 JAPAN \tt{E-mail:syuchan@math.kyoto-u.ac.jp}}\footnote{Originally published in {\it Forum of Mathematics, Pi} {\bf 9} e5 (2021) 56pp.}}

\begin{document}
\maketitle
 
\begin{abstract}
We exhibit basic algebro-geometric results on the formal model of semi-infinite flag varieties and its Schubert varieties over an algebraically closed field $\bK$ of characteristic $\neq 2$ from scratch. We show that the formal model of a semi-infinite flag variety admits a unique nice (ind)scheme structure, its projective coordinate ring has a $\Z$-model, and it admits a Frobenius splitting compatible with the boundaries and opposite cells in positive characteristic. This establishes the normality of the Schubert varieties of the quasi-map space with a fixed degree (instead of their limits proved in [K, Math. Ann. {\bf 371} no.2 (2018)]) when $\mathsf{char} \, \bK =0$ or $\gg 0$, and the higher cohomology vanishing of their nef line bundles in arbitrary characteristic $\neq 2$. Some particular cases of these results play crucial roles in our proof [K, arXiv:1805.01718] of a conjecture by Lam-Li-Mihalcea-Shimozono [J. Algebra {\bf 513} (2018)] that describes an isomorphism between affine and quantum $K$-groups of a flag manifold.
\end{abstract} 

\section*{Introduction}
The semi-infinite flag varieties are variants of affine flag varieties that encode the modular representation theory of a semi-simple Lie algebra, representation theory of a quantum group at roots of unity, and representation theory of an affine Lie algebra at the critical level. They originate from the ideas of Lusztig \cite{Lus80} and Drinfeld, put forward by Feigin-Frenkel \cite{FF}, and subsequently polished by the works of Braverman, Finkelberg, and their collaborators \cite{FM99,FFKM,ABBGM,Bra06,BFFR,BF14a,BF14b,BF14c}. They (mainly) employed the ind-model of semi-infinite flag varieties and achieved spectacular success on the geometric Langlands correspondence \cite{ABBGM,BFGM}, on the quantum $K$-groups of flag manifolds \cite{BF14a}, and on their (conjectural) relation to the finite $\mathcal W$-algebras \cite{BFFR}.

In \cite{KNS17}, we have initiated the study of the formal model of a semi-infinite flag variety (over $\C$) that follows the classical description of flag varieties \cite{Kum02,Mat88,Lit98,KL02} more closely than the above. We refer this formal model of a semi-infinite flag variety as a ``semi-infinite flag manifold" since we hope to justify that it is ``smooth" in a sense. However, the analysis in \cite{KNS17} has two defects: the relation with the ind-models of semi-infinite flag varieties is unclear, and the treatment there is rather ad hoc (it is just an indscheme whose set of $\C$-valued points have the desired property, and lacks a characterization as a functor cf. \cite{BL94, Fal03}). The former defect produces difficulty in the discussion of deep properties on the identification between the equivariant $K$-group of a semi-infinite flag manifold and the equivariant quantum $K$-group of a flag manifold \cite{Kat18c}, that is in turn inspired by the works of Givental and Lee \cite{Giv94, GL03}. The goal of this paper is to study semi-infinite flag manifolds in characteristic $\neq 2$ from scratch, and resolve the above defects. In particular, we verify that the scheme in \cite{KNS17} is the universal one among all the indschemes with similar set-theoretic properties, and provide new proofs of the normality of Zastava spaces \cite{BF14a} and of the semi-infinite flag manifolds \cite{KNS17}.

It is possible to regard our works (\cite{Kat18,Kat18b,KNS17,Kat18c,Kat19a}) as a part of catch-up of Peterson's original construction \cite{Pet97} of his isomorphism \cite{LS10} between the quantum cohomology of a flag manifold and the cohomology of an affine Grassmannian in the $K$-theoretic setting. From this view-point, this paper provides some varieties considered in \cite[Lecture 11]{Pet97} with their appropriate compactifications. Hence, though there are still some missing pieces to complete the original program along the lines in \cite{Pet97}, this paper provides a step to fully examine his ideas.

\medskip

To explain our results more precisely, we introduce more notation: Let $\g$ denote a simple Lie algebra (given in terms of root data and the Chevalley generators) over an algebraically closed field $\bK$ of characteristic $\neq 2$. Let $G$ denote the connected simply-connected algebraic group over $\bK$ such that $\g = \mathrm{Lie} \, G$. Let $H \subset G$ be a Cartan subgroup and let $N$ be an unipotent radical of $G$ that is normalized by $H$. We set $B:= HN$ and $\sB := G / B$ (the flag manifold of $G$). Let $\bI^+ \subset G ( \bK [\![ z ]\!] )$ denote the Iwahori subgroup that contains $B$, and let $\bI^- \subset G ( \bK [z^{-1}])$ be its opposite Iwahori subgroup. Let $\tg$ denote the untwisted affine Kac-Moody algebra associated to $\g$. Let $W$ and $W_\af$ be the finite Weyl group and the affine Weyl group of $\g$, respectively. The coroot lattice $Q^{\vee}$ of $\g$ yields a natural subgroup $\{ t_{\beta} \}_{\beta \in Q^{\vee}} \subset W_\af$. Let $w_0 \in W$ be the longest element.

Our first main result is as follows:

\begin{fthm}[$\doteq$ Theorem \ref{bQ-int} and Proposition \ref{bQcoarse}]\label{fbQ}
There is an indscheme $\bQ_G ^{\ra}$ with the following properties:
\begin{enumerate}
\item The indscheme $\bQ_G ^{\ra}$ is expressed as the union of infinite type integral schemes flat over $\Z$;
\item If we set $( \bQ_G ^{\ra} )_\bK := \bQ_G ^{\ra} \otimes _\Z \bK$, then we have
$$( \bQ_G ^{\ra} )_\bK ( \bK ) \cong G ( \bK (\!(z)\!)) / ( H ( \bK ) N ( \bK (\!(z)\!)) )$$
that intertwines the natural $G ( \bK (\!(z)\!) ) \ltimes \Gm ( \bK )$-actions on the both sides, where $\Gm$ is the loop rotation;
\item The functor
$$\Aff^{op}_{\bK} \ni R \mapsto G ( R (\!(z)\!)) / ( H ( R ) N ( R (\!(z)\!)) )\in \mathrm{Sets}.$$
is coarsely ind-representable by $( \bQ_G ^{\ra} )_\bK$ $($see \S \ref{cmp} for the convention$)$. 
\end{enumerate}
\end{fthm}

One can equip $( \bQ_G ^{\ra} )_\bK ( \bK )$ with an indscheme structure using the arc scheme of the basic affine space $\overline{G/N}$. Such an indscheme {\it cannot} coincide with ours (in general) by the appearance of the nontrivial nilradicals \cite{Mus01,FM17,FM18}. In fact, such an indscheme defines a radicial thickening of $( \bQ_G ^{\ra} )_\bK$.

The set of $H \times \Gm$-fixed points of $( \bQ_G ^{\ra} )_\bK$ is in bijection with $W_\af$. Let $p_w \in ( \bQ_G ^{\ra} )_\bK$ be the point corresponding to $w \in W_\af$. We set $\bO ( w ) := \bI^+ p_w$ and $\bO^- ( w ) := \bI^- p_w$ for each $w \in W_\af$.

Theorem \ref{fbQ} has some applications to the theory of quasi-map spaces from $\P^1$ to $\sB$ presented in \cite{FM99,FFKM,BF14a,BF14b,BF14c} as follows:

\begin{fthm}\label{fmain}
In the above settings, it holds:
\begin{enumerate}
\item $($Corollary {\rm\ref{FQ-uniq}} and Theorem {\rm\ref{QQ-isom}}$)$ If $\mathsf{char} \, \bK > 0$, then the scheme $( \bQ_G^{\ra} )_{\bK}$ admits an $\bI^{\pm}$-canonical Frobenius splitting that is compatible with $\overline{\bO ( w )}$'s and $\overline{\bO ^- ( v )}$'s $(w, v \in W_\af)$; 
\item $($Lemma {\rm\ref{Q'red}} and Theorem {\rm\ref{cut-out-by-div}}$)$ For each $w, v \in W_\af$, the intersection $\sQ ( v, w ) := \overline{\bO ( w )} \cap \overline{\bO ^- ( v )}$ is reduced. It is irreducible when $v = w_0 t_{\beta}$ for some $\beta \in Q^{\vee}$;
\item $($Lemma {\rm\ref{Q-split}} and Corollary {\rm\ref{ggnormal}}$)$ For each $w, v \in W_\af$, the scheme $\sQ ( v, w )$ is weakly normal. It is $($irreducible and$)$ normal when $\mathsf{char} \, \bK = 0$ or $\mathsf{char} \, \bK \gg 0$;
\item $($Lemma {\rm\ref{Q-mod}}$)$ For each $\beta \in Q^{\vee}_+$, the set of $\bK$-valued points of the scheme $\sQ ( w_0 t_{\beta}, e )$ is in bijection with the set of $(\bK$-valued$)$ Drinfeld-Pl\"ucker data. In particular, $\sQ ( w_0 t_{\beta}, e )$ is isomorphic to the quasi-map space in {\rm\cite{FM99}} when $\bK = \C$.
\end{enumerate}
\end{fthm}

Theorem \ref{fmain} is a key result at the deepest part (correspondence between natural bases) in our proof (\cite{Kat18c}) of a conjecture of Lam-Li-Mihalcea-Shimozono \cite{LLMS17} about the comparison between the equivariant $K$-group of the affine Grassmannian of $G$ and the equivariant small quantum $K$-group of $\sB$. In \cite{Kat18c}, we also prove that $\sQ ( w_0 t_{\beta}, w )$ admits only rational singularities (and hence it is Cohen-Macauley) when $\bK = \C$ on the basis of Theorem \ref{fmain}. We remark that Theorem \ref{fmain} 3) is proved in \cite{BF14a,BF14b} when $v = w_0 t_{\beta}$, $w = e$, and $\bK = \C$.

Our proof of Theorem \ref{fmain} 1) is not at all standard, and in fact it forms the core of the technical contributions in this paper. To appreciate its contents, let us recall that there are two standard ways to construct a Frobenius splitting of $\sB$ (\cite{BK05}): One is to consider the Bott-Samelson-Demazure-Hansen resolution of $\sB$, that reduces the assertion to the case of a point, which is a Schubert variety with trivial Frobenius splitting. Another is to analyze the space of global sections of (some power of) the canonical bundle of $\sB$.

However, the both of the above proof strategies do not work for $\bQ_G^{\ra}$. The first one fails since any Schubert variety of $\bQ_G^{\ra}$ is infinite-dimensional, and carries rich internal structure by itself. The second one fails since the canonical bundle of $\bQ_G^{\ra}$ simply does not make sense, at least naively. These require us some new ideas in order to prove Theorem \ref{fmain} 1). Our idea here is as follows: An interpretation of the filtrations in \cite{KL17} reduces the existence of a Frobenius splitting of $\bQ_G^\ra$ into a property of the Frobenius splitting of the corresponding thick affine flag manifold \cite[Corollary B]{Kat18b}. This property can be seen as a special case of some homological property in the representation theory of affine Lie algebras (\cite{CI15,CK18}), but it is proved only for characteristic zero. Thus, we use Kashiwara's theory of global basis (\cite{Kas94,Kas02}) to transfer such a homological property into the positive characteristic setting (Proposition \ref{proj-cover}).

In the rest of introduction, we assume $\bK = \C$ for the sake of simplicity. Let $P$ be the weight lattice of $H$, and let $P_+ \subset P$ denote its subset corresponding to dominant weights. For each $\la \in P$, we have an equivariant line bundle $\cO_{\bQ_G^{\ra}} ( \la )$ on $\bQ_G^{\ra}$, whose restriction to $\sQ ( v, w )$ is denoted by $\cO_{\sQ ( v, w )} ( \la )$. Associated to $\la \in P_{+}$, we have a level zero extremal weight module $\bX ( \la )$ of $U ( \tg )$ in the sense of Kashiwara \cite{Kas02}. We know that $\bX ( \la )$ is equipped with two kinds of Demazure modules, and a distinguished basis (the global basis).

\begin{fcor}[$\doteq$ Theorem \ref{coh}]\label{fcv}
Let $w, v \in W_\af$. For each $\la \in P_{+}$, we have
$$H^{>0} ( \sQ ( v, w ), \cO_{\sQ ( v, w )} ( \la )) = \{ 0 \}.$$
The space $H^0 ( \sQ ( v, w ), \cO_{\sQ ( v, w )} ( \la ))^{\vee}$ is the intersection of two Demazure modules of $\bX ( \la )$ spanned by a subset of the global basis of $\bX ( \la )$ if $\la$ is strictly dominant. If we have $w',v' \in W_\af$ such that $\sQ ( v',w' ) \subset \sQ ( v, w )$, then the restriction map
$$H^0 ( \sQ ( v, w ), \cO_{\sQ ( v, w )} ( \la )) \longrightarrow \!\!\!\!\! \rightarrow H^0 ( \sQ ( v', w' ), \cO_{\sQ ( v', w' )} ( \la ))$$
is surjective.
\end{fcor}

Note that Corollary \ref{fcv} adds new vanishing region to \cite[Theorem 3.1 1)]{BF14b}. 
We also provide parabolic versions of Theorem \ref{fbQ}, Theorem \ref{fmain}, and Corollary \ref{fcv}. We have a description of $H^0 ( \sQ ( v, w ), \cO_{\sQ ( v, w )} ( \la ))$ for general $\la \in P_+$ that is more complicated (Theorem \ref{SMT}).

Let $\sB_{2,\beta}$ be the space of genus zero stable maps with two marked points to $\sB$ with the class of its image $\beta \in Q^{\vee}_+ \subset Q^{\vee} \cong H_2 ( \sB, \Z )$. We have evaluation maps $\mathtt{e}_j : \sB_{2,\beta} \to \sB$ for $j = 1,2$. The following purely geometric result is a byproduct of our proof, that maybe of independent interest.

\begin{fcor}[$\doteq$ Corollary \ref{Bconn+}]\label{fconn}
Let $\beta \in Q^{\vee}_+$, and let $x, y \in \sB$. The space $( \mathtt{e}_1^{-1} ( x ) \cap \mathtt{e}_2^{-1} ( \overline{B y} ) )$ is connected if it is nonempty.
\end{fcor}

Note that Corollary \ref{fconn} is contained in \cite{BCMP} whenever $\{x\},\overline{By} \subset \sB$ are in general position.

\medskip

The plan of this paper is as follows: In section one, we collect basic material needed in the sequel. In section two, after recalling generalities on Frobenius splitting and representation theory of quantum loop algebras, we construct the indscheme $\bQ_G^{\ra}$ and equip it with a Frobenius splitting (Corollary \ref{FQ-uniq}). In section three, we first interpret $\bQ_G^{\ra}$ as an indscheme (coarsely) representing the coset $G ( \bK (\!(z)\!) ) / (H ( \bK ) N ( \bK (\!(z)\!) ))$ (Theorem \ref{fbQ}). Using this, we identify some Richardson varieties of $\bQ_G^{\ra}$ with quasi-map spaces (Theorem \ref{QQ-isom}) and present their cohomological properties (Theorem \ref{coh}), and hence prove (large parts of) Theorem \ref{fmain} and Corollary \ref{fcv}. Since our construction equips quasi-map spaces with Frobenius splittings (Lemma \ref{Q-split}), they are automatically weakly normal. Moreover, we explain how to connect characteristic zero and positive characteristic (\S \ref{subsec:tofrom}). In section four, we analyze the fibers of the graph space resolutions of quasi-map spaces and deduce that Richardson varieties of semi-infinite flag manifolds (over $\C$) are normal based on the weak normality proved in the previous section. This proves the remaining part of Theorem \ref{fmain}. Our analysis here contains an inductive proof that the fibers of the evaluation maps of the space of genus zero stable maps to flag varieties are connected (Corollary \ref{fconn}). In Appendix A, we give a proof of the normality of (the ind-pieces of) $\bQ_G^{\ra}$ (that works also in the positive characteristic setting) and present an analogue of the Kempf vanishing theorem \cite{Kem76} for $\bQ_G^\ra$. Appendix B exhibits the structure of global sections of nef line bundles of Richardson varieties of $\bQ_G^{\ra}$.

Note that Theorem \ref{fmain} equips a quasi-map space from $\P^1$ to $\sB$ with a Frobenius splitting compatible with the boundaries. However, the notion of boundary in quasi-map spaces depend on a configuration of points in $\P^1$ (we implicitly set them to $\{0,\infty\} \subset \P^1$ throughout this paper). This makes our analogues of open Richardson varieties not necessarily smooth contrary to the original case \cite{Ric92} (see also \cite[\S 8.4.1]{FM99}). The author hope to give more account of this, as well as the factorization structure (\cite[\S 6.3]{FM99}) from the view point presented in this paper, in his future works.

\section{Preliminaries}

We work over an algebraically closed field $\bK$ unless stated otherwise. A vector space is a $\bK$-vector space, and a graded vector space refers to a $\Z$-graded vector space whose graded pieces are finite-dimensional and its grading is bounded from the above or from the below. Tensor products are taken over $\bK$ unless specified otherwise.

Let $A$ be a PID. For a graded free $A$-module $M = \bigoplus _{m \in \Z} M_m$, we set $M^{\vee} := \bigoplus _{m \in \Z} \mathrm{Hom}_A ( M_m, A )$, where $\mathrm{Hom}_A ( M_m, A )$ is understood to be degree $-m$.

As a rule, we suppress $\emptyset$ and associated parenthesis from notation. This particularly applies to $\emptyset = \tJ \subset \tI$ frequently used to specify parabolic subgroups.

\subsection{Groups, root systems, and Weyl groups}\label{subsec:groups}
We refer to \cite{CG97, Kum02} for precise expositions of general material presented in this subsection.

Let $G$ be a connected, simply connected simple algebraic group of rank $r$ over an algebraically closed field $\bK$, and let $B$ and $H$ be a Borel subgroup and a maximal torus of $G$ such that $H \subset B$. We set $N$ $(= [B,B])$ to be the unipotent radical of $B$ and let $N^-$ be the opposite unipotent subgroup of $N$ with respect to $H$. We denote the Lie algebra of an algebraic group by the corresponding German small letter. We have a (finite) Weyl group $W := N_G ( H ) / H$. For an algebraic group $E$, we denote its set of $\bK [z]$-valued points by $E [z]$, its set of $\bK [\![z]\!]$-valued points by $E [\![z]\!]$, and its set of $\bK (\!(z)\!)$-valued points by $E (\!(z)\!)$ etc... Let $\mathbf I \subset G [\![z]\!]$ be the preimage of $B \subset G$ via the evaluation at $z = 0$ (the Iwahori subgroup of $G [\![z]\!]$). We set $\bI^- \subset G [z^{-1}]$ be the opposite Iwahori subgroup of $\bI$ in $G (\!(z)\!)$ with respect to $H$. By abuse of notation, we might consider $\bI$ and $G [\![z]\!]$ as pro-algebraic groups over $\bK$ whose $\bK$-valued points are given as these.

Let $P := \mathrm{Hom} _{gr} ( H, \Gm )$ be the weight lattice of $H$, let $\Delta \subset P$ be the set of roots, let $\Delta_+ \subset \Delta$ be the set of roots that yield root subspaces in $\gb$, and let $\Pi \subset \Delta _+$ be the set of simple roots. We set $\Delta_- := - \Delta_+$. Let $Q^{\vee}$ be the dual lattice of $P$ with a natural pairing $\left< \bullet, \bullet \right> : Q^{\vee} \times P \rightarrow \Z$. We define $\Pi^{\vee} \subset Q ^{\vee}$ to be the set of positive simple coroots, and let $Q_+^{\vee} \subset Q ^{\vee}$ be the set of non-negative integer span of $\Pi^{\vee}$. For $\beta, \gamma \in Q^{\vee}$, we define $\beta \ge \gamma$ if and only if $\beta - \gamma \in Q^{\vee}_+$. We set $P_+ := \{ \lambda \in P \mid \left< \alpha^{\vee}, \lambda \right> \ge 0, \hskip 2mm \forall \alpha^{\vee} \in \Pi^{\vee} \}$ and $P_{++} := \{ \lambda \in P \mid \left< \alpha^{\vee}, \lambda \right> > 0, \hskip 2mm \forall \alpha^{\vee} \in \Pi^{\vee} \}$. Let $\mathtt I := \{1,2,\ldots ,r\}$. We fix bijections $\mathtt I \cong \Pi \cong \Pi^{\vee}$ such that $i \in \tI$ corresponds to $\alpha_i \in \Pi$, its coroot $\alpha_i^{\vee} \in \Pi ^{\vee}$, and a simple reflection $s_i \in W$ corresponding to $\alpha_i$. Let $\{\varpi_i\}_{i \in \tI} \subset P_+$ be the set of fundamental weights (i.e. $\left< \al_i^{\vee}, \varpi_j \right> = \delta_{ij}$) and $\rho := \sum_{i \in \tI} \varpi_i = \frac{1}{2}\sum_{\al \in \Delta^+} \al \in P_+$.

For a subset $\tJ \subset \tI$, we define $P ( \tJ )$ as the standard parabolic subgroup of $G$ corresponding to $\tJ$. I.e. we have $\gb \subset \mathfrak p (\tJ) \subset \g$ and $\mathfrak p (\tJ)$ contains the root subspace corresponding to $- \alpha_i$ ($i \in \tI$) if and only if $i \in \tJ$. Let $H \subset L ( \tJ ) \subset P ( \tJ )$ be the standard Levi subgroup (that is isomorphic to the quotient of $P ( \tJ )$ by its unipotent radical). Then, the set of characters of $P ( \tJ )$ is identified with $P_{\tJ} := \sum_{i \in \tI \setminus \tJ} \Z \varpi_i$. We also set $P_{\tJ, +} := \sum_{i \in \tI \setminus \tJ} \Z_{\ge 0} \varpi_i = P_+ \cap P_{\tJ}$ and $P_{\tJ, ++} := \sum_{i \in \tI \setminus \tJ} \Z_{\ge 1} \varpi_i = P_{++} \cap P_{\tJ}$. We define $W_\tJ \subset W$ to be the reflection subgroup generated by $\{s_i\}_{i \in \tJ}$. It is the Weyl group of $[L ( \tJ ), L (\tJ) ]$ and $L ( \tJ )$. We define $\rho_\tJ$ to be the half-sum of positive roots whose root spaces are contained in the unipotent radical of $\mathfrak p ( \tJ )$.

Let $\Delta_{\af} := \Delta \times \Z \delta \cup \{m \delta\}_{m \neq 0}$ be the untwisted affine root system of $\Delta$ with its positive part $\Delta_+ \subset \Delta_{\af, +}$. We set $\alpha_0 := - \vartheta + \delta$, $\Pi_{\af} := \Pi \cup \{ \alpha_0 \}$, and $\mathtt I_{\af} := \mathtt I \cup \{ 0 \}$, where $\vartheta$ is the highest root of $\Delta_+$. We set $W _{\af} := W \ltimes Q^{\vee}$ and call it the affine Weyl group. It is a reflection group generated by $\{s_i \mid i \in \mathtt I_{\af} \}$, where $s_0$ is the reflection with respect to $\alpha_0$. We also have a reflection $s_{\alpha} \in W_\af$ corresponding to $\alpha \in \Delta \times \Z \delta \subsetneq \Delta^\af$. Let $\ell : W_\af \rightarrow \Z_{\ge 0}$ be the length function and let $w_0 \in W$ be the longest element in $W \subset W_\af$. Together with the normalization $t_{- \vartheta^{\vee}} := s_{\vartheta} s_0$ (for the coroot $\vartheta^{\vee}$ of $\vartheta$), we introduce the translation element $t_{\beta} \in W _{\af}$ for each $\beta \in Q^{\vee}$.

For each $i \in \tI_{\af}$, we have a connected algebraic group $\SL ( 2, i )$ that is isomorphic to $\SL ( 2 )$ equipped with an inclusion $\SL ( 2, i ) ( \bK ) \subset G (\!(z)\!)$ as groups corresponding to $\pm \al_i \in \tI_{\af}$. Let $\rho_{\pm \alpha_i} : \Gm \rightarrow \SL ( 2, i )$ denote the unipotent one-parameter subgroup corresponding to $\pm \alpha _i \in \Delta_{\af}$. We set $B_i := \SL ( 2, i ) \cap \mathbf I$, that is a Borel subgroup of $\SL ( 2, i )$. For each $i \in \tI$, we set $P_i := P ( \{ i \} )$. For each $i \in \tI_\af$, we set $\bI ( i ) := \SL ( 2, i ) \bI$. Each $\bI (i)$ can be regarded as a pro-algebraic group.

As a variation of \cite[Chapter V\!I]{Kum02}, we say an indscheme $\mathfrak X$ over $\bK$ admits a $G(\!(z)\!)$-action if it admits an action of $\bI$ and $\SL ( 2, i )$ ($i \in \tI_\af$) as (ind)schemes over $\bK$ that coincides on $B_i = ( \bI \cap \SL ( 2, i ) )$, and they generate a $G(\!(z)\!)$-action on the set of closed points of $\mathfrak X$ (the latter is a group action on a set). We consider the notion of $G(\!(z)\!)$-equivariant morphisms accordingly.

We set
$$Q^{\vee}_< := \{\beta \in Q^{\vee} \mid \left< \beta, \al_i \right> < 0, \forall i \in \tI \}.$$

Let $\le$ be the Bruhat order of $W_\af$. In other words, $w \le v$ holds if and only if a subexpression of a reduced decomposition of $v$ yields a reduced decomposition of $w$. We define the generic (semi-infinite) Bruhat order $\le_\si$ as:
\begin{equation}
w \le_\si v \Leftrightarrow w t_{\beta} \le v t_{\beta} \hskip 5mm \text{for every } \beta \in Q^{\vee} \text{ such that } \left< \beta, \al_i \right> \ll 0 \text{ for } i \in \tI. \label{si-ord}
\end{equation}
By \cite{Lus80}, this defines a preorder on $W_{\af}$. Here we remark that $w \le v$ if and only if $w \ge_\si v$ for $w, v \in W$. We also have
\begin{equation}
w \le_\si v \Leftrightarrow ww_0 \ge_\si vw_0 \hskip 5mm w,v\in W_\af.\label{ord-opp}
\end{equation}
For proofs and related results, we refer to \cite[\S 2.2]{KNS17} and \cite[Lecture 13]{Pet97}.

For each $u \in W$ and $\beta \in Q^{\vee}$, we set
$$\ell^\si ( u t_{\beta} ) := \ell ( u ) + \sum_{\al \in \Delta_+} \left< \beta, \al \right> = \ell ( u ) + 2 \left< \beta, \rho \right>.$$

\begin{thm}[Lusztig \cite{Lus80} cf. \cite{LS10} Proposition 4.4]\label{cover}
For each $w, v \in W_\af$ such that $w \le_\si v$, there exists $\al \in \Delta_+^\af$ such that $ w \le_\si s_{\al} v \le_\si v$ and $\ell^{\si} ( s_{\al} v ) = \ell^{\si} ( v ) + 1$. \hfill $\Box$
\end{thm}

For each $\la \in P_+$, we denote the corresponding Weyl module by $V ( \la )$ (see e.g. \cite[Proposition 1.22]{APW91} and \cite[Theorem 5]{Kas91}). By convention, $V ( \la )$ is a finite-dimensional indecomposable $G$-module with a cyclic $B$-eigenvector $\bv_\la^0$ (highest weight vector) with its $H$-weight $\la$ whose character obeys the Weyl character formula. For a semi-simple $H$-module $V$, we set
$$\ch \, V := \sum_{\la \in P} e^\la \cdot \dim _{\bK} \mathrm{Hom}_H ( \bK_\la, V ).$$
If $V$ is a $\Z$-graded $H$-module in addition, then we set
\begin{equation}
\gch \, V := \sum_{\la \in P, n \in \Z} q^n e^\la \cdot \dim _{\bK} \mathrm{Hom}_H ( \bK_\la, V_n ).\label{defgch}
\end{equation}

Let $\sB := G / B$ and call it the flag manifold of $G$. We have the Bruhat decomposition
\begin{equation}
\sB = \bigsqcup _{w \in W} \bO_{\sB} ( w )\label{Bdec}
\end{equation}
into $B$-orbits such that $\dim \, \bO_{\sB} ( w ) = \ell ( w _0 ) - \ell ( w )$ for each $w \in W \subset W_\af$. We set $\sB ( w ) := \overline{\mathbb O_{\sB} ( w )} \subset \sB$.

For each $\la \in P$, we have a line bundle $\cO _{\sB} ( \la )$ such that
$$H ^0 ( \sB, \cO_{\sB} ( \la ) ) \cong V ( \la )^*, \hskip 3mm \cO_{\sB} ( \la ) \otimes_{\cO_\sB} \cO _\sB ( - \mu ) \cong \cO_\sB ( \la - \mu ) \hskip 5mm \la, \mu \in P_+.$$

For each $w \in W$, let $p_w \in \bO_{\sB} ( w )$ be the unique $H$-fixed point. We normalize $p_w$ (and hence $\bO_{\sB} ( w )$) so that the restriction of $\cO_{\sB} ( \la )$ to $p_w$ is isomorphic to $\bK_{- w w_0 \la}$ for every $\la \in P_+$. (Here we warn that the convention differs from \cite{Kat18c}.)

\subsection{Representations of affine and current algebras}\label{rep-ac}

In the rest of this section, we work over $\bK = \C$, the field of complex numbers. Material in this subsection without a reference can be found in \cite{Kac,Kas91}. Every result in this subsection is transferred to an arbitrary field in \S \ref{rep-Z}.

Let $\tg$ denote the untwisted affine Kac-Moody algebra associated to $\g$. I.e. we have
$$\tg = \g \otimes \C [z, z^{-1}] \oplus \C K \oplus \C d,$$
where $K$ is central, $[d, X \otimes z^m] = m X \otimes z ^m$ for each $X \in \g$ and $m \in \Z$, and for each $X, Y \in \g$ and $f, g \in \C [z^{\pm 1}]$ it holds:
$$[X \otimes f , Y \otimes g ] = [X, Y] \otimes f g + ( X, Y )_{\g} \cdot K \cdot \mathrm{Res}_{z = 0} f \frac{\partial g}{\partial z},$$
where $(\bullet, \bullet)_{\g}$ denotes the $G$-invariant bilinear form such that $( \al^{\vee}, \al^{\vee})_{\g} = 2$ for a long simple root $\al$. Let $E_i, F_i$ ($i \in \tI_\af$) denote the Kac-Moody generators of $\tg$ corresponding to $\al_i$. We set $\wth := \h \oplus \C K \oplus \C d$. Let $\gI$ be the Lie subalgebra of $\tg$ generated by $E_i$ ($i \in \tI_\af$) and $\wth$, and $\gI^-$ be the Lie subalgebra of $\tg$ generated by $F_i$ ($i \in \tI_\af$) and $\wth$. For each $i \in \tI_\af$ and $n \ge 0$, we set $E_i^{(n)} := \frac{1}{n!} E_i^n$ and $F_i^{(n)} := \frac{1}{n!} F_i^n$.

We define
$$Q^{\af, \vee} :=  \Z d \oplus \bigoplus_{i \in \tI_\af} \Z \al^{\vee}_i \subset \wth, \hskip 10mm P^{\af} := \Z \delta \oplus \bigoplus_{i \in \tI_\af} \Z \Lambda_i \subset \wth^*,$$
and a pairing $Q^{\af, \vee} \times P^\af \rightarrow \Z$ such that
$$\left< \al_i^{\vee}, \Lambda_j \right> = \delta_{ij} \hskip 2mm (i,j \in \tI_\af), \hskip 4mm \left< \al^{\vee}_i, \delta \right> \equiv 0, \hskip 4mm \left< d, \Lambda_i \right> = \delta_{i0} \hskip 2mm (i \in \tI_\af), \hskip 4mm \left< d, \delta \right> = 1.$$
We have a projection map
$$P^\af \ni \Lambda = k \delta + \sum_{i \in \tI_\af} a_i \Lambda_i \mapsto \overline{\Lambda} = \sum_{i \in \tI} a_i \varpi_i \in P,$$
that has a unique splitting $P \subset P^\af$ whose image is orthogonal to $d,K \in \wth$. We set $P^{\af}_+ := \sum_{i \in \tI_\af} \Z_{\ge 0} \Lambda_i$. Each $\Lambda \in P^{\af}_+$ defines an irreducible integrable highest weight module $L ( \Lambda )$ of $\tg$ with its highest weight vector $\bv_{\Lambda}$. In addition, each $\la \in P_+$ defines a level zero extremal weight module $\bX ( \la )$ of $\tg$ by means of the specialization of the quantum parameter $\mathsf q = 1$ in \cite[Proposition 8.2.2]{Kas94} and \cite[\S 5.1]{Kas02}, that is integrable and $K$ acts by $0$. The module $\bX ( \la )$ carries a cyclic $\wth$-weight vector $\bv_{\la}$ such that:
$$
H \bv_{\la} = \la ( H ) \bv_{\la} \hskip 2mm (H \in \h), \hskip 2mm K \bv _{\la} = 0 = d \bv_{\la}, \hskip 2mm
E_i \bv_{\la} = 0 \hskip 2mm (i \in \tI), \hskip 2mm \text{and} \hskip 2mm F_0 \bv_{\la} = 0.
$$
(We can deduce that $\bX ( \la )$ is the maximal integrable $\tg$-module that possesses a cyclic vector with the above properties \cite[\S 8.1]{Kas94}.) Moreover, each $w = u t_{\beta} \in W_\af$ ($u \in W, \beta \in Q^{\vee}$) defines an element $\bv_{w \la} \in \bX ( \la )$ such that
$$H \bv_{w \la} = ( w \la ) ( H ) \bv_{w \la} \hskip 3mm (H \in \h), \hskip 3mm K \bv _{w \la} = 0, \hskip 3mm d \bv_{w \la} = - \left< \beta, \la \right> \bv_{w \la}$$
up to sign (see \cite[\S 8.1]{Kas94}). We call a vector in $\{\bv_{w\la}\}_{w \in W_\af}$ an extremal weight vector of $\bX ( \la )$.

We set $\g [z] := \g \otimes_\C \C [z]$ and regard it as a Lie subalgebra of $\tg$. We have $\gI \subset \g [z] + \C K + \C d$. The Lie algebra $\g [z]$ is graded, and its grading is the internal grading of $\tg$ given by $d$.

For each $\la \in P_+$, we set
$$\bW_w ( \la ) := U ( \gI ) \bv_{w\la} \subset \bX ( \la ).$$
These are the $\mathsf q = 1$ cases of the Demazure modules of $\bX ( \la )$, as well as the generalized global Weyl modules in the sense of \cite{FMO18}. We set $\bW ( \la ) := \bW_{w_0} ( \la )$. By construction, the both of $\bX ( \la )$ and $\bW _w ( \la )$ are semi-simple as $( H \times \Gm )$-module, where $\Gm$ acts on $z$ by $a : z^m \mapsto a^{m} z^m$ ($m \in \Z$).

\begin{thm}[LNSSS \cite{LNSSS2}, Chari-Ion \cite{CI15}, cf. \cite{Kat18} Theorem 1.6]\label{qwc}
For each $\lambda \in P_+$, the $\gI$-action on $\bW ( \la )$ prolongs to $\g [z]$, and it is isomorphic to the global Weyl module of $\g [z]$ in the sense of Chari-Pressley {\rm\cite{CP01}}. Moreover, $\bW ( \la )$ is a projective module in the category of $\g[z]$-modules whose restriction to $\g$ is a direct sum of modules in $\{ V ( \mu )\}_{\mu \le \la}$. \hfill $\Box$
\end{thm}

\begin{thm}[\cite{Kat18} Theorem 3.3 and Corollary 3.4]\label{WC-surj}
Let $\la, \mu \in P_+$ and $w \in W$. We have a unique $($up to scalar$)$ injective degree zero $\gI$-module map
$$\bW_w ( \la + \mu ) \hookrightarrow \bW_w ( \la ) \otimes \bW_w ( \mu ).$$
\end{thm}

\begin{proof}[Sketch of proof]
For each $\la, \mu \in P_+$, the projectivity of $\bW ( \la + \mu )$ in the sense of Theorem \ref{qwc} yields a unique graded $\g [z]$-module map
$$\bW ( \la + \mu ) \longrightarrow \bW ( \la ) \otimes \bW ( \mu )$$
of degree $0$. This map is injective by examining the specializations to local Weyl modules (for their definitions, see \cite[Theorem 1.4]{Kat18}, or Lemma \ref{Wk-free} and Remark \ref{Wk-remark}). By examining the $\gI$-cyclic vectors, it uniquely restricts to a map
$$\bW_w ( \la + \mu ) \longrightarrow \bW_w ( \la ) \otimes \bW_w ( \mu )$$
up to scalar. This map must be also injective as the ambient map is so.
\end{proof}

\subsection{Semi-infinite flag manifolds}
We work over $\C$ as in the previous subsection. Material in this section is reproved in the setting of characteristic $\neq 2$ in \S \ref{subsec:fsQ} and \S \ref{sec:mi} (cf. \S \ref{cmp}). We define the semi-infinite flag manifold as the reduced indscheme such that:
\begin{itemize}
\item We have a closed embedding
$$\bQ_G^{\ra} \subset \prod_{i \in \tI} \P ( V ( \varpi_i ) \otimes \C (\!(z)\!) ); \hskip 3mm \text{and}$$
\item We have an equality $\bQ_G^{\ra} ( \C ) = G (\!(z)\!) / \left( H ( \C ) \cdot N (\!(z)\!) \right)$.
\end{itemize}
This is a pure indscheme of ind-infinite type \cite{KNS17}. Note that the group $Q^{\vee} \subset H (\!(z)\!) / H ( \C )$ acts on $\bQ_G^{\ra}$ from the right. The indscheme $\bQ_G^{\ra}$ is equipped with a $G (\!(z)\!)$-equivariant line bundle $\cO _{\bQ_G^{\ra}} ( \la )$ for each $\la \in P$. Here we normalized so that $\Gamma ( \bQ_G^{\ra}, \cO_{\bQ_G^{\ra}} ( \la ) )$ is $B^- (\!(z)\!)$-cocyclic to a $H$-weight vector with its $H$-weight $- \la$. We warn that this convention is twisted by $-w_0$ from that of \cite{Kat18c}, and complies with \cite{KNS17}.

\begin{thm}[\cite{FM99,FFKM,KNS17,Lus80}]\label{si-Bruhat}
We have an $\bI$-orbit decomposition
$$\bQ_G^{\ra} = \bigsqcup_{w \in W_\af} \bO ( w )$$
with the following properties:
\begin{enumerate}
\item Each $\mathbb O ( w )$ is isomorphic to $\A^{\infty}$ and have a unique $(H \times \Gm)$-fixed point;
\item The right action of $\gamma \in Q^{\vee}$ on $\bQ_G^{\ra}$ yields the translation $\mathbb O ( w ) \mapsto \mathbb O ( w t_{\gamma})$;
\item We have $\mathbb O ( w ) \subset \overline{\mathbb O ( v )}$ if and only if $w \le_{\si} v$;
\item The relative dimension of $\bO ( u t_{\beta} )$ $(u \in W, \beta \in Q^{\vee})$ and $\bO ( e )$, counted as the difference of the cardinality of the maximal chain of intermediate $\bI$-orbits to a common smaller $\bI$-orbits, is $\ell^\si ( u t_{\beta} )$. \hfill $\Box$
\end{enumerate}
\end{thm}

For each $w \in W_\af$, let $\bQ_G ( w )$ denote the closure of $\bO ( w )$. We refer $\bQ_G ( w )$ as a Schubert variety of $\bQ_G^{\ra}$ (corresponding to $w \in W_\af$).

Let $S = \bigoplus _{\la \in P_{\tJ,+}} S ( \la )$ be a $P_{\tJ,+}$-graded commutative ring such that $S ( 0 ) = A$ is a PID, $S$ is torsion-free over $A$, and $S$ is generated by $\bigoplus_{i \in \tI \setminus \tJ} S ( \varpi_i )$. We define
\begin{eqnarray}
\mathrm{Proj} \, S = ( \Spec \, S \setminus E ) / H \subset \prod_{i \in \tI \setminus \tJ} \P_A ( S ( \varpi_i )^{\vee} )\label{mproj}
\end{eqnarray}
as the $P_{\tJ,+}$-graded proj over $\Spec \, A$, where $E$ is the locus that whole of $S ( \varpi_i )$ vanishes for some $i \in \tI \setminus \tJ$ (the irrelevant locus).

\begin{thm}[\cite{KNS17} Theorem 4.26 and Corollary 4.27]\label{QG-proj}
For each $w \in W_\af$, it holds:
$$\bQ_G ( w ) \cong \mathrm{Proj} \, \bigoplus _{\la \in P_+} \bW _{ww_0} ( \la )^{\vee},$$
where the multiplication of the ring $\bigoplus_\la \bW _{ww_0} ( \la )^{\vee}$ is given by Theorem {\rm\ref{WC-surj}}.
\end{thm}

\subsection{Quasi-map spaces and Zastava spaces}\label{sec:QM}

We work over $\C$ as in the previous subsection. Here we recall basics of quasi-map spaces from \cite{FM99,FFKM}.

We have $W$-equivariant isomorphisms $H^2 ( \sB, \Z ) \cong P$ and $H_2 ( \sB, \Z ) \cong Q ^{\vee}$. This identifies the (integral points of the) nef cone of $\sB$ with $P_+ \subset P$ and the effective cone of $\sB$ with $Q_+^{\vee}$. A quasi-map $( f, D )$ is a map $f : \P ^1 \rightarrow \sB$ together with a $\Pi^{\vee}$-colored effective divisor
$$D = \sum_{\alpha \in \Pi^{\vee}, x \in \P^1 (\C)} m_x (\alpha^{\vee}) \alpha^{\vee} \otimes [x] \in Q^{\vee} \otimes_\Z \mathrm{Div} \, \P^1 \hskip 3mm \text{with} \hskip 3mm m_x (\alpha^{\vee}) \in \Z_{\ge 0}.$$
For $i \in \mathtt I$, we set $D_i := \left< D, \varpi_i \right> \in \mathrm{Div} \, \P^1$. We call $D$ the defect of the quasi-map $(f, D)$. Here we define the (total) degree of the defect by
$$|D| := \sum_{\alpha \in \Pi^{\vee}, x \in \P^1 (\C)} m_x (\alpha^{\vee}) \alpha^{\vee} \in Q_+^{\vee}.$$

For each $\beta \in Q_+^{\vee}$, we set
$$\sQ ( \sB, \beta ) : = \{ f : \P ^1 \rightarrow X \mid \text{ quasi-map s.t. } f _* [ \P^1 ] + | D | = \beta \},$$
where $f_* [\P^1]$ is the class of the image of $\P^1$ multiplied by the degree of $\P^1 \to \mathrm{Im} \, f$. We denote $\sQ ( \sB, \beta )$ by $\sQ ( \beta )$ in case there is no danger of confusion.

\begin{defn}[Drinfeld-Pl\"ucker data]\label{Zas}
Consider a collection $\mathcal L = \{( \psi_{\la}, \mathcal L^{\la} ) \}_{\la \in P_+}$ of inclusions $\psi_{\la} : \mathcal L ^{\la} \hookrightarrow V ( \la ) \otimes _{\C} \mathcal O _{\P^1}$ of line bundles $\mathcal L ^{\la}$ over $\P^1$. The data $\mathcal L$ is called a Drinfeld-Pl\"ucker data (DP-data) if the canonical inclusion of $G$-modules
$$\eta_{\la, \mu} : V ( \la + \mu ) \hookrightarrow V ( \la ) \otimes V ( \mu )$$
induces an isomorphism
$$\eta_{\la, \mu} \otimes \mathrm{id} : \psi_{\la + \mu} ( \mathcal L ^{\la + \mu} ) \stackrel{\cong}{\longrightarrow} \psi _{\la} ( \mathcal L^{\la} ) \otimes_{\cO_{\P^1}} \psi_{\mu} ( \mathcal L^{\mu} )$$
for every $\la, \mu \in P_+$.
\end{defn}

\begin{thm}[Drinfeld, see Finkelberg-Mirkovi\'c \cite{FM99}]\label{Dr}
The variety $\sQ ( \beta )$ is isomorphic to the variety formed by isomorphism classes of the DP-data $\mathcal L = \{( \psi_{\la}, \mathcal L^{\la} ) \}_{\la \in P_+}$ such that $\deg \, \mathcal L ^{\la} = \left< w_0 \beta, \la \right>$. In addition, $\sQ ( \beta )$ is an irreducible variety of dimension $2 \left< \rho, \beta \right> + \ell ( w_0 )$.
\end{thm}

For each $w \in W$, let $\sZ ( \beta, w ) \subset \sQ ( \beta )$ be the locally closed subset consisting of quasi-maps that are defined at $z = 0$, and their values at $z = 0$ are contained in $\sB ( w ) \subset \sB$. We set $\sQ ( \beta, w ) := \overline{\sZ ( \beta, w )}$. (Hence, we have $\sQ ( \beta ) = \sQ ( \beta, e)$.)

\begin{thm}[Finkelberg-Mirkovi\'c \cite{FM99}]\label{aff-Zas-p}
Let $\bK$ be an algebraically closed field $($that is not necessarily characteristic zero$)$, and let $\sQ ( \beta )_{\bK}$ and $\sZ ( \beta, w_0 )_{\bK}$ be the spaces obtained by replacing the base field $\C$ with $\bK$ in Definition {\rm\ref{Zas}}. For each $\beta \in Q^{\vee}_+$, the space $\sZ ( \beta, w_0 )_{\bK}$ is an irreducible affine scheme equipped with an action of $( B \times \Gm )$ over $\bK$. In addition, this action has a unique fixed point.
\end{thm}

\begin{proof}[Remarks on proof]
Theorem \ref{aff-Zas-p} is proved in \cite{FM99} for $\bK = \C$ using \cite{MV07}, and is proved in the current setting in \cite{BFGM} using \cite{BG}. One can also replace the usage of \cite{MV07} with \cite[Corollary 5.3.8]{Zhu17} along the lines of \cite{FM99}.
\end{proof}

For each $\la \in P$ and $w \in W$, we have a $G$-equivariant line bundle $\cO _{\sQ ( \beta, w )} ( \la )$ (and its pro-object $\cO _{\sQ} ( \la )$) obtained by the (tensor products of the) pull-backs $\cO _{\sQ ( \beta, w )}( \varpi_i )$ of the $i$-th $\cO ( 1 )$ via the embedding
\begin{equation}
\sQ ( \beta, w ) \hookrightarrow \prod_{i \in \mathtt I} \P ( V ( \varpi_i ) \otimes_{\C} \C [z] _{\le - \left< w_0 \beta, \varpi_i \right>} )\label{Pemb}
\end{equation}
for each $\beta \in Q_+^{\vee}$.

We have embeddings $\sB \subset \sQ ( \beta ) \subset \bQ_G ( e )$ such that the line bundles $\cO ( \la )$ ($\la \in P$) correspond to each other by restrictions (\cite{BF14b,Kat18,KNS17}).

\section{Semi-infinite flag manifolds over $\Z [\frac{1}{2}]$}\label{sec:fsQ}

We keep the settings of the previous section. In this section, we sometimes work over a (commutative) ring or a non-algebraically closed field. For a ring $S$ or a scheme $\gX$, we may write $S_{A}$ and $\gX_{A}$ if it is defined over $A$. In addition, we may consider their scalar extensions $S_{B} := S_A \otimes _A B$ and $\gX_{B}$ for a ring map $A \to B$.

\subsection{Frobenius splittings}

Let $\Bbbk$ be a field, and let $p$ be a prime. We assume $\mathsf{char}\, \Bbbk = p$, $\Bbbk \subset \bK$, and the $p$-th power map is invertible on $\Bbbk$ (e.g. $\Bbbk = \F_p$ or $\overline{\F}_p$) throughout this subsection.

We follow the generality on Frobenius splittings in \cite{BK05}, that considers separated schemes of finite type. We sometimes use the assertions from \cite{BK05} without finite type assumption when the assertion is independent of that, whose typical disguises are properness, finite generation, and the Serre vanishing theorem.

\begin{defn}[Frobenius splitting of a ring]\label{FrR}
Let $R$ be a commutative ring over $\Bbbk$, and let $R^{(1)}$ denote the set $R$ equipped with the map
$$R \times R^{(1)} \ni (r,m) \mapsto r^p m \in R^{(1)}.$$
This equips $R^{(1)}$ an $R$-module structure over $\Bbbk$ (the $\Bbbk$-vector space structure on $R^{(1)}$ is also twisted by the $p$-th power operation), together with a map $\imath : R . 1 \rightarrow R^{(1)}$. An $R$-module map $\phi : R^{(1)} \to R$ is said to be a Frobenius splitting if $\phi \circ \imath$ is an identity.
\end{defn}

Note that $\imath$ in Definition \ref{FrR} must be an inclusion if we have a Frobenius splitting $\phi$. Since the $p$-th power map in $\Bbbk$ is invertible, we can twist the (scalar multiplication part of the) $\Bbbk$-vector space structure of $R$ ($\cong R^{(1)}$ as sets) to make $\imath$ into a $\Bbbk$-linear map, without making it into an $R$-linear map (when $R \neq \Bbbk$).

\begin{defn}[Frobenius splitting of a scheme]
Let $\mathfrak X$ be a separated scheme defined over $\Bbbk$. Let $\Fr$ be the (relative) Frobenius endomorphism of $\mathfrak X$ (that induces a $\Bbbk$-linear endomorphism). We have a natural inclusion $\imath : \cO_{\mathfrak X} \rightarrow \Fr_{*} \cO_{\mathfrak X}$. A Frobenius splitting of $\mathfrak X$ is a $\cO_{\mathfrak X}$-linear morphism $\phi : \Fr_{*} \cO_{\mathfrak X} \rightarrow \cO_{\mathfrak X}$ such that the composition $\phi \circ \imath$ is the identity.
\end{defn}

\begin{defn}[Compatible splitting]
Let $\gY \subset \gX$ be a closed immersion of separated schemes defined over $\Bbbk$. A Frobenius splitting $\phi$ of $\gX$ is said to be compatible with $\gY$ if $\phi (\mathsf{Fr}_* \mathcal I _{\gY} ) \subset \mathcal I_{\gY}$, where $\mathcal I _{\gY} := \ker ( \cO_{\gX} \to \cO_{\gY} )$. Compatible Frobenius splitting of a pair of a commutative ring and its quotient ring is defined through their spectrums.
\end{defn}

\begin{rem}
A Frobenius splitting of $\gX$ compatible with $\gY$ induces a Frobenius splitting of $\gY$ (see e.g. \cite[Remark 1.1.4 (ii)]{BK05}).
\end{rem}

\begin{thm}[\cite{BK05} Lemma 1.1.14 and Exercise 1.1.E]\label{F-rel}
Let $\mathfrak X$ be a separated scheme of finite type over $\Bbbk$ with semiample line bundles $\mathcal L_1,\ldots, \mathcal L_r$. If $\mathfrak X$ admits a Frobenius splitting, then the multi-section ring
$$\bigoplus_{n_1,\ldots,n_r \ge 0} \Gamma ( \mathfrak X, \mathcal L_1 ^{\otimes n_1} \otimes \cdots \otimes \mathcal L_r ^{\otimes n_r} )$$
admits a Frobenius splitting $\phi$. Moreover, a closed subscheme $\gY \subset \gX = \mathrm{Proj} \, R$ admits a compatible Frobenius splitting if and only if the homogeneous ideal $I _{\gY} \subset R$ that defines $\gY$ satisfies $\phi ( I_{\gY} ) \subset I_{\gY}$, i.e. the pair $(R, R / I _{\gY})$ admits a compatible Frobenius splitting $\phi$.  \hfill $\Box$
\end{thm}

\begin{defn}[Canonical splitting]
Let $\gX$ be a separated scheme over $\Bbbk$ equipped with a $B$-action. A Frobenius splitting $\phi$ is said to be $B$-canonical if it is $H$-fixed, and each $i \in \tI$ yields
\begin{equation}
\rho_{\al_i} ( z^p ) \phi ( \rho_{\al_i} ( - z ) f )  = \sum_{j = 0}^{p-1} \frac{z^j}{j!} \phi_{i, j} ( f ) \hskip 5mm z \in \Bbbk,\label{Bcaneq}
\end{equation}
where $\phi_{i, j} \in \Hom_{\cO_{\mathfrak X}} ( \Fr_{*} \cO_{\mathfrak X}, \cO_{\mathfrak X} )$. We similarly define the notion of $B^-$-canonical splitting (resp. $\bI$-canonical splitting and $\bI^-$-canonical splitting) by using $\{ \rho_{-\al_i} \}_{i \in \tI}$ (resp. $\{ \rho_{\al_i} \}_{i \in \tI_{\af}}$ and $\{ \rho_{-\al_i} \}_{i \in \tI_{\af}}$) instead. Canonical splittings of a commutative ring $S$ over $\Bbbk$ is defined through its spectrum.
\end{defn}

\begin{prop}[cf. \cite{BK05} Proposition 4.1.8]\label{B-can-inv}
Let $S = \bigoplus_{m \ge 0} S_m$ be a graded ring with $S_0 = \Bbbk$ such that
\begin{itemize}
\item $S$ is equipped with a degree preserving $\bI$-action;
\item Each $S_m$ is a graded $\Bbbk$-vector space compatible with the multiplication; 
\item We have an $\bI$-canonical Frobenius splitting $\phi : S^{(1)} \to S$.
\end{itemize}
Then, the induced map
$$\phi^{\vee} : S_{m}^{\vee} \longrightarrow S_{pm}^{\vee} \hskip 5mm m \in \Z_{\ge 0}$$
satisfies
$$\phi^{\vee} ( E_i^{(n)} \bv ) = E_i^{(pn)} \phi^{\vee} ( \bv ) \hskip 5mm \forall i \in \tI_\af, n \in \Z_{\ge 0}, \bv \in S_m^{\vee}.$$
Similar results hold for the $\bI^-$- and $B^{\pm}$-actions.
\end{prop}

\begin{rem}
In the opinion of the author, a merit of Proposition \ref{B-can-inv} over \cite[Proposition 4.1.8]{BK05} is that it becomes obvious that a projective variety $\mathfrak X$ with a $B$-action has at most one $B$-canonical splitting whenever the space of global sections of all ample line bundles are (or can be made) $B$-cocyclic compatible with multiplications (cf. \cite[Theorem 4.1.15]{BK05} and Corollary \ref{af-uniq}).
\end{rem}

\begin{proof}[Proof of Proposition \ref{B-can-inv}]
The condition that $S_m$ is a graded vector space implies $S_m \stackrel{\cong}{\longrightarrow}(S_m^{\vee})^{\vee}$ for each $m \in \Z_{\ge 0}$. By \cite[Proposition 4.1.8]{BK05}, each $\bw \in S_{pm} \subset S^{(1)}$ satisfies $\phi ( E_i^{(pn)} \bw ) = E_i^{(n)} \phi ( \bw )$ for $i \in \tI_\af$ and $n \ge 0$. Using the natural non-degenerate invariant pairing $\left< \bullet, \bullet \right>$ between $S_m^{\vee}$ and $S_m$, we compute the most LHS of
$$\left< \bv, \phi ( E^{(p)}_i \bw ) \right> = \left< \bv, E_i \phi ( \bw ) \right> = - \left< \phi^{\vee} ( E_i \bv ), \bw \right>$$
by using the invariance under the corresponding unipotent action as
\begin{align*}
\left< \bv, \phi ( E^{(p)}_i \bw ) \right> & = - \sum_{k_1 = 1}^{p} \left< E^{(k_1)}\phi^{\vee} ( \bv ), E^{(p-k_1)}_i \bw \right>\\
\cdots & = \sum_{m=1}^p \sum_{k_\bullet > 0, k_1 + k_2 +\cdots + k_m = p} (-1)^m \left< E^{(k_1)}_i E^{(k_2)}_i \cdots \phi^{\vee} ( \bv ), \bw \right>\\
& = - \left< E^{(p)} \phi^{\vee} ( \bv ), \bw \right>
\end{align*}
since we have $E^{(k_1)}_i E^{(k_2)}_i \cdots E^{(k_m)}_i \in p \Z E^{(p)}_i$ except for $k_1 = p, 0 = k_2 = \cdots$. This implies the case $n = 1$.

Similarly, we have
$$\left< \bv, \phi ( E^{(pn)}_i \bw ) \right> = \sum_{m=1}^n \sum_{k_\bullet > 0, k_1 + k_2 +\cdots + k_m = n} (-1)^m \left< E^{(pk_1)}_i E^{(pk_2)}_i \cdots \phi^{\vee} ( \bv ), \bw \right>.$$
Compared with
$$\left< \bv, E^{(n)}_i \phi ( \bw ) \right> = \sum_{m=1}^n \sum_{k_\bullet > 0, k_1 + k_2 +\cdots + k_m = n} (-1)^m \left<  \phi^{\vee} ( E^{(k_1)}_i E^{(k_2)}_i \cdots\bv ), \bw \right>$$
using induction on $n$, we conclude the result.
\end{proof}

\subsection{Representations of affine Lie algebras over $\Z$}\label{rep-Z}

In this section, we systematically use the global basis theory \cite{Kas91,Kas94,Kas02,Kas05,Lus92,GL93} by specializing the quantum parameter $\mathsf q$ to $1$. Therefore, we might refer these references without an explicit declaration that we specialize $\mathsf q$.

We consider the Kostant-Lusztig $\Z$-form $U ^+ _{\Z}$ (resp. $U^-_\Z$) of $U ( [\gI,\gI] )$ (resp. $U ( [\gI^-,\gI^-] )$) obtained as the specialization $\mathsf q = 1$ of the $\Z [\mathsf q,\mathsf q^{-1}]$-integral form of the quantized enveloping algebras \cite[\S 23.2]{Lus93}.

\begin{rem}
We remark that $U ^\pm _{\Z}$ are the same integral forms dealt in \cite{Gar78}, and also coincide with the integral forms obtained through the Drinfeld presentation (\cite[\S 2]{BCP99} and \cite[Lemma 2.5]{Nao17}).
\end{rem}

Note that $U^{\pm}_\Z$ are equipped with the $\Z$-bases $\bB ( \mp \infty )$ obtained by the specialization $\mathsf q = 1$ of the lower global basis \cite{Kas91} (see also \cite[\S 25]{Lus93}). In view of \cite{Lus92,Kas94}, we have an idempotent $\Z$-integral form
\begin{align*}
\dot{U}_\Z = \bigoplus_{\Lambda \in P^{\af}} U ^- _{\Z} U ^+ _{\Z} a_{\Lambda} & & \text{such that}\\
a_{\Lambda} a_{\Gamma} = \delta_{\Lambda, \Gamma} a_{\Lambda} & & \Lambda, \Gamma \in P^{\af} \hskip 3mm \text{and},\\
 E_i^{(m)} a_{\Lambda} = a_{\Lambda + m \al_i} E_i^{(m)}, \hskip 2mm F_i^{(m)} a_{\Lambda} & = a_{\Lambda - m \al_i} F_i^{(m)} & i \in \tI_\af, m \in \Z_{\ge 0}.
\end{align*}
We set $\dot{U}^{\ge 0}_\Z \subset \dot{U}_\Z$ to be the subalgebra generated by $\{F_i^{(m)}\}_{i \in \tI, m \in \Z_{\ge 0}}$, $\{a _{\Lambda}\}_{\Lambda \in P^\af}$, and $U_\Z^+$.

If a $\dot{U}_\Z$-module $M$ admits a decomposition
$$M = \bigoplus_{\Lambda \in P^\af} a_{\Lambda} M,$$
then we call this the $P^\af$-weight decomposition. If $\Lambda \in P^\af$ satisfies $a_{\Lambda} M \neq 0$, then we call $\Lambda$ a $P^\af$-weight of $M$. In case $M$ is defined over a field $\Bbbk$, we define the $P^\af$-character of $M$ as
$$\mathrm{gch} \, M := \sum _{\Lambda \in P^\af} e^{\Lambda} \dim_{\Bbbk} \, a_{\Lambda} M$$
whenever the RHS makes sense. For each $n \in \Z$, we set
$$M_n := \sum_{\Lambda \in P^\af, \left< d, \Lambda \right> = n} a_\Lambda M \subset M$$
and call it the $d$-degree $n$-part of $M$. Note that these are consistent with (\ref{defgch}) through the identification $q = e^{\delta}$.

For each $\la \in P$, we set
\begin{equation}
a_{\la}^0 M := \sum_{\Lambda \in P^\af, \la = \bar{\Lambda}} a_{\Lambda} M.\label{a0-a}
\end{equation}
We call the decomposition
$$M = \bigoplus_{\la \in P} a_{\la}^0 M,$$
the $P$-weight decomposition. We call a non-zero element of $a_{\Lambda} M$ (resp. $a^0 _\la M$) a $P^\af$-weight vector of $M$ (resp. a $P$-weight vector of $M$). We also call $\la \in P$ with $a_{\la}^0 M \neq \{ 0 \}$ a $P$-weight of $M$.

Similarly, we have the Kostant-Lusztig $\Z$-form $U ^{0, +} _{\Z}$ (resp. $U^{0,- }_\Z$) of $U ( \gn )$ (resp. $U ( \gn^- )$). We have $U^{0,+}_{\Z} \subset U^{+}_{\Z}$ and $U^{0,-}_{\Z} \subset U^{-}_{\Z}$. In view of the characterization of global bases (\cite{Kas91}), we find that $\bB^0 ( \mp \infty ) := \bB ( \mp \infty) \cap U^{0,\pm}_{\Z}$ define $\Z$-bases of $U^{0,\pm}$.

We set $\dot{U}^{0}_\Z \subset \dot{U}_\Z$ to be the subalgebra of $\dot{U}_\Z$ generated by $\{E_i^{(m)}, F_i^{(m)}\}_{i \in \tI, m \in \Z_{\ge 0}}$, $\{a _{\Lambda} \}_{\Lambda \in P^\af}$. For a field $\Bbbk$, a $\dot{U}_{\Bbbk}^{\ge 0}$-module $M$ with a $P^\af$-weight decomposition is said to be $\dot{U}^0_\Bbbk$-integrable if its $\{ E^{(m)}_i, F^{(m)}_i \}_{m \ge 0}$-action induces a $\SL ( 2, i )_\Bbbk$-action whose $( \SL ( 2, i ) \cap H )_\Bbbk$-eigenvalues are given by the $P$-weights for each $i \in \tI$.

Note that if a $U ( \tg_\C )$-module $V$ over $\C$ carries a cyclic $\wth_\C$-weight vector whose weight belongs to $P^\af$ and each of its $\wth_\C$-weight space is finite-dimensional, then we have a $\dot{U}_\Z$-lattice $V_{\Z}$ inside $V$. In such a case, the module $V_{\Z} \otimes_{\Z} \Bbbk$ admits $P^\af$- or $P$-weight decompositions.

We have the Chevalley involution of $\dot{U}_\Z$ defined as:
$$\theta ( E_i ^{(m)}) = F_i ^{(m)}, \theta ( F_i ^{(m)}) = E_i ^{(m)}, \hskip 2mm \text{and} \hskip 2mm \theta ( a _{\Lambda} ) = a_{- \Lambda} \hskip 4mm i \in \tI_\af, m \in \Z_{\ge 0}, \Lambda \in P^\af.$$

\begin{defn}[\cite{Kas05} Definition 2.4 and \S 2.8]
A $U ( \tg_{\C} )$-module $V$ over $\C$ with a cyclic $\wth_{\C}$-weight vector $\bv$ is said to be compatible with the negative global basis if we have
$$U ^{-}_\Z \bv = \bigoplus _{b \in \bB ( \infty )} \Z b \bv \subset V.$$

If $(V,\bv)$ is compatible with the negative global basis, then we set
$$\bB ^- ( V ) = \bB ^{-} ( V, \bv ):= \{ b\bv \mid b \in \bB ( \infty ) \text{ s.t. } b\bv \neq 0 \} \subset V$$
and refer them as the negative global basis of $V$.

Compatibility with the positive global basis of $V$ and the positive global basis $\bB ^{+} ( V ) = \bB ^{+} ( V, \bv )$ of $V$ is defined similarly.
\end{defn}

\begin{thm}[Kashiwara \cite{Kas91} Theorem 5]\label{Bla}
We have:
\begin{enumerate}
\item For each $\Lambda \in P^{\af}_+$, the $\tg_{\C}$-module $L ( \Lambda )_\C$ is compatible with the negative global basis;
\item For each $\la \in P_+$, we have
$$V ( \la )_\C = \bigoplus _{b \in \bB^0 ( \infty )} \C b \bv_{\la}^0.$$
\end{enumerate}
\end{thm}

We set $\bB ( \Lambda ) := \bB^- ( L ( \Lambda )_\C, \bv_{\Lambda} )$ for each $\Lambda \in P^\af_+$.

For each $\Lambda \in P^{\af}_+$ and $\la \in P_+$, we set
$$L ( \Lambda )_{\Z} := U _{\Z}^- \bv_{\Lambda} \subset L ( \Lambda )_\C \hskip 5mm \text{and} \hskip 5mm V ( \la )_{\Z} := ( U _{\Z}^{0,-} ) \bv_{\la}^0 \subset V ( \la )_\C.$$

Here $V ( \la )_{\Z}$ acquires the action of $\dot{U}^0_\Z$ thanks to the splitting $P \hookrightarrow P^\af$.

\begin{cor}\label{L-dualmult}
We have:
\begin{enumerate}
\item For each $\Lambda, \Gamma \in P^{\af}_+$, we have a natural inclusion $L ( \Lambda + \Gamma )_\Z \hookrightarrow L ( \Lambda )_\Z \otimes_\Z L ( \Gamma )_\Z$ of $\dot{U}_{\Z}$-modules, that is a direct summand as $\Z$-modules;
\item For each $\la,\mu \in P_+$, we have a natural inclusion $V ( \la + \mu )_\Z \hookrightarrow V ( \la )_\Z \otimes_\Z V ( \mu )_\Z$ of $\dot{U}_{\Z}^0$-modules, that is a direct summand as $\Z$-modules.
\end{enumerate}
\end{cor}

\begin{proof}
Since the two cases are completely parallel, we only prove the first case. The $\tg$-module $L ( \Lambda )_\C \otimes_\C L ( \Gamma )_\C$ decomposes into the direct sum of integrable highest weight modules (\cite[Proposition 9.10]{Kac}), with a direct summand $L ( \Lambda + \Gamma )_\C$. In view of \cite[Theorem 3]{Kas91}, it gives rise to the $\Z [\mathsf q]$-lattice of the quantized version of $L ( \Lambda ) \otimes L ( \Gamma )$ compatible with those of $L ( \Lambda + \Gamma )$ via the natural embedding. By setting $\mathsf q = 1$, we obtain a direct sum decomposition of $L ( \Lambda )_\Z \otimes_\Z L ( \Gamma )_\Z$ as $\Z$-modules with its direct summand $L ( \Lambda + \Gamma )_\Z$.
\end{proof}

\begin{thm}[Kashiwara \cite{Kas94} Proposition 8.2.2]\label{X-compat}
For each $\la \in P_+$, the $\tg_\C$-module $\bX ( \la )_\C$ is compatible with the negative/positive global basis $($for every extremal weight vector$)$. \hfill $\Box$
\end{thm}

For each $\la \in P_+$, we set
$$\bX ( \la ) _\Z := \dot{U}_{\Z} \bv_{\la} \subset \bX ( \la )_\C.$$

\begin{thm}[Kashiwara \cite{Kas05}]\label{b-compat}
Let $\la \in P_+$. There exists a $\Z$-basis $\bB ( \bX ( \la ) )$ of $\bX ( \la )_\Z$ that contains the negative/positive global basis of $\bX ( \la )_\Z$ constructed from every extremal weight vector of $\bX ( \la )_\C$.
\end{thm}

\begin{proof}
We set $\bB ( \bX ( \la ) )$ to be the specialization of the global basis of a quantum loop algebra module \cite[Proposition 8.2.2]{Kas94}. Then, it is compatible with the global basis generated from an extremal weight vectors by \cite[Theorem 3.3]{Kas05}.
\end{proof}

For each $\la \in P_+$ and $w \in W_\af$, we define
$$\bW _w ( \la )_{\Z} := U_{\Z}^+ \bv_{w \la} \subset \bX ( \la )_\C \hskip 3mm \text{and} \hskip 3mm \bW _w^- ( \la )_{\Z} := U_{\Z}^- \bv_{w \la} \subset \bX ( \la )_\C .$$
We set $\bW ( \la )_\Z := \bW _{w_0} ( \la )_{\Z}$ and $\bW ^-( \la )_\Z := \bW _{e} ^- ( \la )_{\Z}$.

\begin{lem}[Naito-Sagaki]\label{contain}
For each $\la \in P_+$ and $w,v \in W_\af$, we have $\bW _{ww_0} ( \la )_{\Z} \subset \bW _{vw_0} ( \la )_{\Z}$ if $w \le_\si v$. If we have $\la \in P_{++}$ in addition, then we have $\bW _{ww_0} ( \la )_{\Z} \subset \bW _{vw_0} ( \la )_{\Z}$ if and only if $w \le_\si v$.
\end{lem}

\begin{proof}
Apply the inclusion relation of the (labels of the) global basis in \cite[Corollary 5.2.5]{NS16} (see also \cite[\S 2.8]{Kas05}).
\end{proof}

\begin{cor}\label{SL2-stable}
For each $\la \in P_+$, $w \in W_\af$, and $i \in \tI_\af$, we have $\bW _{s_i ww_0} ( \la )_{\bK} \subset \bW _{ww_0} ( \la )_{\bK}$ if $s_i w \le_\si w$. In this case, $\bW _{ww_0} ( \la )_{\bK}$ inherits an action of $\SL (2,i)_{\bK}$ from $\bX ( \la )_{\bK}$.
\end{cor}

\begin{proof}
The first part of the assertion is the special case of Lemma \ref{contain}. Given this, it remains to notice that a lift of $s_i \in W_\af$ sends $\bv_{ww_0 \la}$ to $\pm \bv_{s_iww_0 \la}$, and hence the Bruhat decomposition of $\bI ( i )_{\bK}$ (into two $\bI_{\bK}$-double cosets) implies that $\bW _{ww_0} ( \la )_{\bK}$ is stable under $\bI ( i )_{\bK}$.
\end{proof}

\begin{lem}\label{WX-comm}
For each $\la \in P_+$ and $w \in W_\af$, it holds:
\begin{enumerate}
\item Each $\beta \in Q^{\vee}$ defines a $\dot{U}_\Z$-module automorphism $\tau_\beta$ on $\bX ( \la )_\Z$ determined by $\tau_{\beta}( \bv_{\la} ) := \bv_{t_{\beta} \la}$. Moreover, we have $\tau_{\beta}\bB ( \bX ( \la ) ) = \bB ( \bX ( \la ) )$;
\item We have $\theta^* ( \bX ( \la )_\Z ) \cong \bX ( - w_0 \la )_\Z$ as $\dot{U}_\Z$-modules. Moreover, we have $\theta^* \bB ( \bX ( \la ) ) = \bB ( \bX ( - w_0 \la ) )$;
\item We have $\bW _w ( \la ) _\Z = \bW _w ( \la )_\C \cap \bX ( \la )_\Z$;
\item We have a $U_\Z^-$-cyclic vector of $\theta^* (\bW _w ( \la ) _\Z )$ with weight $- w \la = w w_0 ( - w_0 \la )$. In particular, we have
$$\theta^* ( \bW_w ( \la )_\Z ) \cong \bW_{ww_0}^- ( -w_0 \la )_\Z \hskip 3mm \text{and} \hskip 3mm \theta^* ( \bW_w^- ( \la )_\Z ) \cong \bW_{ww_0} ( -w_0 \la )_\Z.$$
\end{enumerate}
\end{lem}

\begin{proof}
We borrow the setting of \cite[\S 8.1 and \S 8.2]{Kas94}.

We prove the first assertion. Since $\bv_{\la}$ and $\bv_{t_{\beta} \la}$ obeys the same relation, $\tau_\beta$ defines an automorphism as $\tg_\C$-modules. The latter assertion follows from Theorem \ref{b-compat}.

We prove the second assertion. The defining equation of $\theta^* ( \bv_{\la} )$ is the same as the cyclic vector $\bv_{- w_0 \la} \in \bX ( - w_0 \la )_\C$ as $\tg_\C$-modules. This yields a $\tg_\C$-module isomorphism $\eta : \theta^* ( \bX ( \la )_\C ) \longrightarrow \bX ( - w_0 \la )_\C$. Since $\theta$ exchanges $U^\pm_\Z$ and $\bv_\la$ is cyclic, we deduce that $\eta ( \theta^* ( \bX ( \la )_\Z ) ) = \dot{U}_\Z \bv_{- \la} \subset \bX ( - w_0 \la )_\C$. By Theorem \ref{b-compat}, we conclude $\theta^* \bB ( \bX ( \la ) ) = \bB ( \bX ( - w_0 \la ) )$ as required.

We prove the third assertion. By Theorem \ref{b-compat}, the $\Z$-basis of $\bW ( \la ) _\Z$ is formed by the non-zero elements of $\bB ( - \infty ) \bv_{w_0 \la}$ and forms a direct summand of $\bX ( \la )_\Z$ as $\Z$-modules. Hence, the case $w = w_0$ follows. For $w \in W$, we apply \cite[Lemma 8.2.1]{Kas94} repeatedly to deduce the assertion from the $w = w_0$ case by using $\bB ( - \infty ) \bv_{w \la} \subset \bB ( - \infty ) \bv_{w_0 \la}$. For $w = u t_{\beta} \in W_\af$ with $u \in W, \beta \in Q^{\vee}$, we additionally apply $\tau_{w_0 \beta}$ to conclude the assertion.

We prove the fourth assertion. The vector $\theta^* ( \bv_{w \la} )$ is a $U_\Z^-$-cyclic vector of $\theta^* (\bW _w ( \la ) _\Z )$, and its weight is
$$- w \la = ww_0 (- w_0 \la ).$$
Hence, we conclude the assertion (using the fact that $\theta$ is an involution).
\end{proof}

\begin{thm}[\cite{Kas02} Proposition 8.6 and \cite{BN04} Corollary 4.15]\label{W-unit}
Let $\la \in P_+$. The unique $(d$-degree zero$)$ $\dot{U}^{\ge 0}_\Z$-module map
\begin{eqnarray}
\Psi_\la : \bW ( \la )_{\Z} \hookrightarrow \bigotimes _{i \in \tI} \bW ( \varpi_i )_{\Z}^{\otimes \left< \al_i^{\vee}, \la \right>},\label{ptensor}
\end{eqnarray}
that sends $\bv_\la$ to the tensor product of $\{\bv_{\varpi_i}\}_{i \in \tI}$'s, is injective and defines a direct summand as $\Z$-modules.
\end{thm}

\begin{proof}
We set $\bX^{\otimes} := \bigotimes _{i \in \tI} \bX ( \varpi_i )_\Z^{\otimes \left< \al_i^{\vee}, \la \right>}$. The map $\Psi_\la$ exists as $\bigotimes _{i \in \tI} \bv_{\varpi_i}^{\otimes \left< \al_i^{\vee}, \la \right>}$ obeys the same defining condition as the extremal weight vector $\bv_{\la} \in \bX ( \la )_\Z$, and the comultiplication of $\tg$ induces an algebra map $\triangle:\dot{U}_\Z \subset \dot{U}_{\Z} \otimes \dot{U}_\Z$ (\cite[\S 23.1.5]{Lus93}).

The map $\Psi_\la$ is injective by \cite[Corollary 4.15 and Remark 4.17]{BN04}. By \cite[Theorem 8.5]{Kas02}, the $\mathsf{q} = 1$ specializations of the global bases yield the $\dot{U}_{\Z}$-spans of the extremal weight vectors in $\bX ( \la )_\Z$ and $\bX^{\otimes}$ up to the action of some rings of (partially symmetric) polynomials with integer coefficients, respectively. This yields a splitting of $\Psi_\la$ as $\Z$-modules, where the $\Z$-module structure of the RHS is coming from the crystal lattice of (the $\mathsf q$-version of) $\bX^{\otimes}$ (as in \cite[(8.8)]{Kas02}). We call this crystal lattice $\mathsf L_1$.

We have another $\Z [\mathsf q]$-lattice inside (the $\mathsf q$-version of) $\bX^{\otimes}$ obtained by the tensor product of the crystal lattices of (the $\mathsf q$-versions of) $\bX ( \varpi_i )$'s. We call this $\Z [\mathsf q]$-lattice $\mathsf L_2$.

By \cite[Proposition 8.6]{Kas02}, the global basis of $\mathsf L_1$ is written in terms of the (tensor product) global bases of $\mathsf L_2$ by an upper-unitriangular matrix $\mathsf C$ (valued in $\Q [ \mathsf q ]$, with finitely many non-zero entry in each row). Hence, the corresponding bases in $\bX^{\otimes}$ are related by an upper-unitriangular matrix $C$ obtained as the $\mathsf{q} = 1$ specialization of $\mathsf C$ (valued in $\Q$). In view of the fact that $\triangle$ is an algebra morphism (and \cite[Theorem 8.5]{Kas02}), we deduce $\mathsf L_1 \subset \mathsf L_2$. Hence, we have $\Z \otimes_{\Z [\mathsf q]} \mathsf L_1 \subset \Z \otimes_{\Z [\mathsf q]} \mathsf L_2 \subset \bX^{\otimes}$. In particular, the matrix $C$ must be valued in $\Z$. This implies $\Z \otimes_{\Z [\mathsf q]} \mathsf L_1 = \Z \otimes_{\Z [\mathsf q]} \mathsf L_2$.

Hence, the $\Z$-module splitting in the second paragraph is indeed what we wanted.\end{proof}

\subsection{Projectivity of the module $\bW  ( \la )_{\Bbbk}$}

Let $\Bbbk$ be a field. We equip $\mathring{U}^+_\Bbbk := \bigoplus_{\la \in P} U^+_\Bbbk a_{\la}^0$ with the structure of an algebra by setting $a_{\la}^0 a_{\mu}^0 = \delta_{\la, \mu} a_{\la}^0$ $(\la, \mu \in P)$,
$$E_0^{(m)} a_{\la}^0 = a_{\la - m \vartheta}^0 E_0^{(m)}, \hskip 2mm E_i^{(m)} a_{\la}^0 = a_{\la + m \al_i}^0 E_i^{(m)} \hskip 3mm (i \in \tI, m \in \Z_{\ge 0}, \la \in P),$$
and requires
$$( \xi a_{\la}^0 ) ( \xi' a_{\mu}^0 ) = ( \xi \xi' ) a_{\mu}^0 \hskip 3mm \text{when} \hskip 3mm a_{\la}^0 \xi' a_{\mu}^0 = \xi' a_{\mu}^0 \hskip 3mm  ( \la, \mu \in P, \xi, \xi' \in U^+_{\Bbbk} ).$$
Similarly, we define
$$\mathring{U}^{\ge 0}_\Bbbk:= \bigoplus_{\la \in P} U^{0,-}_{\Bbbk} U^{+}_\Bbbk a_{\la}^0 \hskip 5mm \text{and} \hskip 5mm \mathring{U}^{0}_\Bbbk:= \bigoplus_{\la \in P} U^{0,-}_{\Bbbk} U^{0,+}_\Bbbk a_{\la}^0$$
and regard them as subspaces of the completions of $\dot{U}^{\ge 0}_\Bbbk$ and $\dot{U}^{0}_\Bbbk$ with respect to the idempotents by setting
$$\sum_{\la = \overline{\Lambda}}a_{\Lambda} = a_{\la}^0 \hskip 5mm (\la \in P \subset P^\af).$$
The eigendecomposition with respect to $d \in \wth$ now becomes an external gradings of algebras $\mathring{U}^+_\Bbbk$ and $\mathring{U}^{\ge 0}_\Bbbk$. We have algebra inclusions
$$\mathring{U}^+_\Bbbk \subset \mathring{U}^{\ge 0}_\Bbbk \supset \mathring{U}^0_\Bbbk.$$
Note that we have a surjective algebra map $\dot{U}^{\ge 0}_\Z \rightarrow \dot{U}^0_\Z$ since $a_{\Lambda} \dot{U}^{\ge 0}_\Z a_{\Gamma} \neq 0$ implies $\left< d, \Lambda \right> \ge \left< d, \Gamma \right>$ and $a_{\Lambda} \dot{U}^{0}_\Z a_{\Gamma} \neq 0$ implies $\left< d, \Lambda \right> = \left< d, \Gamma \right>$ for $\Lambda, \Gamma \in P^\af$. This induces a surjective algebra map $\mathring{U}^{\ge 0}_\Z \rightarrow \mathring{U}^0_\Z$, that specializes to $\mathring{U}^{\ge 0}_\Bbbk \rightarrow \mathring{U}^0_\Bbbk$.

We sometimes regard $\bW ( \la )_{\Bbbk}$ ($\la \in P_+$) as a graded $\mathring{U}^+_\Bbbk$-module or a graded $\mathring{U}^{\ge 0}_\Bbbk$-module by means of (\ref{a0-a}), whose gradings are given by the $d$-degrees. 

\begin{prop}\label{proj-cover}
Let $\Bbbk$ be a field and let $\la \in P_+$. The module $\bW  ( \la )_{\Bbbk}$ is the projective cover of $V ( \la )_{\Bbbk}$ in the category of $\dot{U}_\Bbbk^{\ge 0}$-modules that are $\dot{U}_\Bbbk^0$-integrable and whose $P$-weights are contained in $\mathrm{Conv} \, W \la \subset P \otimes_\Z \R$, where $\mathrm{Conv}$ denote the $\R$-convex hull.
\end{prop}

\begin{lem}\label{Wk-free}
For each $\la \in P_+$, the module $\bW ( \la )_{\Bbbk}$ is free over a polynomial ring and we have a finite-dimensional quotient $W ( \la )_{\Bbbk}$ with $\dim_{\Bbbk} \, a^0_\la W ( \la )_{\Bbbk} = 1$.
\end{lem}

\begin{rem}\label{Wk-remark}
The modules $\bW ( \la )_{\Bbbk}$ and $W ( \la )_{\Bbbk}$ ($\la \in P_+$) are analogues of global and local Weyl modules in \cite{CP01,LNSSS2,CI15} (cf. Theorem \ref{qwc}). By construction, their characters are the same as these defined over $\C$ (see \cite[\S 1]{Kat18} for more discussion).
\end{rem}

\begin{proof}[Proof of Lemma \ref{Wk-free}]
Note that the endomorphism $\tau_{\beta}$ in Lemma \ref{WX-comm} induces an endomorphism of $\bB ( \bX ( \varpi_i ) )$ for each $\beta \in Q$ and $i \in \tI$, and hence we know the existence of the injective endomorphisms of $\bW ( \la )_{\Bbbk}$ that forms a polynomial algebra whose graded dimension is equal to
$$\mathrm{gdim} \, \mathrm{End}_{\g [z]_\C} \bW ( \la )_\C = \prod_{i \in \tI} \prod_{j=1}^{\left< \al_i^{\vee}, \la \right>}\frac{1}{1-q^j},$$
that can be read-off from \cite[Proposition 4.13]{BN04} (or \cite[Theorem 1.4]{Kat18}), through Theorem \ref{W-unit}. By \cite{BN04,CI15,Kat18}, it exhausts the $P$-weight $\la$-part of $\bW ( \la )_{\Bbbk}$, that generates $\bW ( \la )_{\Bbbk}$ as a $\mathring{U}^{\ge 0}_\Bbbk$-module. Therefore, we conclude the result.
\end{proof}

The rest of this subsection is devoted to the proof of Proposition \ref{proj-cover}.

We consider the Demazure functor $\mathscr D_{w}$ for $w \in W_\af$ with respect to $\mathring{U}^+_\Bbbk$ (cf. \cite{Jos85,Kat18,CK18}). In view of \cite[\S 2.8]{Kas05}, the character part of the calculation in \cite[Theorem 4.13]{Kat18} carries over to our setting, and hence we have
$$\mathbb L^{\bullet} \mathscr D_{t_{\beta}} ( \bW ( \la )_{\Bbbk}) = \mathscr D_{t_{\beta}} ( \bW ( \la )_{\Bbbk}) \cong \bW ( \la )_{\Bbbk} \otimes_{\Bbbk} \Bbbk_{- \left< \beta, w_0 \la \right> \delta} \hskip 8mm \beta \in Q^{\vee}_<.$$
From this, we also derive that
$$\mathbb L^{\bullet} \mathscr D_{t_{\beta}} ( W ( \la )_{\Bbbk}) = \mathscr D_{t_{\beta}} ( W ( \la )_{\Bbbk}) \cong W ( \la )_{\Bbbk} \otimes_{\Bbbk} \Bbbk_{- \left< \beta, w_0 \la \right> \delta} \hskip 8mm \beta \in Q^{\vee}_<$$
using the Koszul resolution as in \cite[\S 5.1.4]{CK18} by Lemma \ref{Wk-free}.

The $d$-gradings of $\mathring{U}^+_\Bbbk, \bW ( \la )$, and $W ( \mu )_{\Bbbk}$ are concentrated in non-negative degrees. It follows that the $d$-grading of $\mathrm{Ext}^{\bullet} _{\mathring{U}^+_\Bbbk} ( \bW ( \la ), W ( \mu )_{\Bbbk}^* )$ is bounded from the above. Moreover, we have
$$\mathrm{Ext}^{\bullet} _{\mathring{U}^+_\Bbbk} ( \mathscr D_{t_{\beta}w_0} ( \bW ( \la )_{\Bbbk}), W ( \mu )_{\Bbbk}^* ) \cong \mathrm{Ext}^{\bullet} _{\mathring{U}^+_\Bbbk} ( \bW ( \la )_{\Bbbk}, \mathscr D_{t_{-w_0 \beta}w_0} ( W ( \mu )_{\Bbbk} )^* )$$
for every $\la, \mu \in P_+$ and $\beta \in Q^{\vee}_<$ by (a repeated application of) \cite[Lemma 8]{Mat89c} (see also \cite[Proposition 5.7]{FKM} for the current formulation). By varying $\beta$, we conclude that
\begin{equation}
\mathrm{Ext}^{\bullet} _{\mathring{U}^+_\Bbbk} ( \bW ( \la )_{\Bbbk}, W ( \mu )_{\Bbbk}^* ) \equiv \{ 0 \} \hskip 5mm \la \neq -w_0 \mu\label{ext-orth}
\end{equation}
since $- \left< \beta, w_0 \la \right> \neq - \left< - w_0 \beta, w_0 \mu \right> = \left< \beta, \mu \right>$ for some choice of $\beta$.

Consider the simple integrable $\mathring{U}_{\Bbbk}^0$-module quotient $L ( \la )_{\Bbbk}$ of $V ( \la )_{\Bbbk}$ for each $\la \in P_+$.
\begin{lem}
The set $\{L ( \la )_{\Bbbk}\}_{\la \in P_+}$ is the complete collection of the isomorphism classes of the irreducible $d$-graded $\mathring{U}_{\Bbbk}^0$-integrable $\mathring{U}^{\ge 0}_{\Bbbk}$-modules $($up to $d$-grading shifts$)$.
\end{lem}

\begin{proof}
Since $V ( \la )_{\Bbbk}$ surjects onto every irreducible $\mathring{U}_{\Bbbk}^0$-modules with highest weight $\la$ (\cite[\S 1.20]{APW91}), it follows that $V ( \la )_{\Bbbk}$ surjects onto every irreducible $d$-graded $\mathring{U}^{\ge 0}_{\Bbbk}$-module with highest weight $\la$. Since the $P$-weight $\la$-part of $V ( \la )_{\Bbbk}$ is one-dimensional, we conclude the result.
\end{proof}

We return to the proof of Proposition \ref{proj-cover}. We have
$$\mathrm{gch} \, W ( \la )_{\Bbbk} \equiv \mathrm{gch} \, V ( \la )_{\Bbbk} \equiv \mathrm{gch} \, L ( \la )_{\Bbbk} \mod \sum_{\la > \mu \in P_+} \Z [q] \, \mathrm{gch} \, L ( \mu )_{\Bbbk}$$
by \cite{LNSSS2} (cf. \cite{CI15}). Therefore, $W ( \la )_{\Bbbk}$ admits a ($d$-graded) Jordan-H\"older series as a $\mathring{U}^{\ge 0}_\Bbbk$-module whose irreducible constituents are of the form $\{L ( \mu )_{\Bbbk}\}_{\mu \le \la}$ (up to $d$-grading shifts). It follows that
$$\mathrm{Ext}^{\bullet} _{\mathring{U}^+_\Bbbk} ( \bW ( \la )_{\Bbbk}, V ( \mu )_{\Bbbk} ) = \mathrm{Ext}^{\bullet} _{\mathring{U}^+_\Bbbk} ( \bW ( \la )_{\Bbbk}, L ( \mu )_{\Bbbk} ) \equiv \{ 0 \} \hskip 5mm \la > \mu \in P_+$$
by a repeated application of the short exact sequences to (\ref{ext-orth}). Since the both of $\bW ( \la )_{\Bbbk}$ and $L ( \la )_{\Bbbk}$ are $\dot{U}^0_\Bbbk$-integrable, we find
$$\mathrm{Ext}^{1} _{\mathring{U}^{\ge 0}_\Bbbk} ( \bW ( \la )_{\Bbbk}, L ( \mu )_{\Bbbk} ) \equiv \{ 0 \} \hskip 5mm \la > \mu \in P_+.$$

From these, it suffices to prove
\begin{equation}
\mathrm{Ext}^{1} _{\mathring{U}^{\ge 0}_\Bbbk} ( \bW ( \la )_{\Bbbk}, L ( \la )_{\Bbbk} ) \equiv \{ 0 \} \hskip 5mm \la \in P_+\label{la-la-ext}
\end{equation}
in order to deduce the assertion.

In view of the fact $\la + \al_i$ ($i \in \tI$) is not a $P$-weight of $\bW ( \la )_{\Bbbk}$, the Drinfeld presentation of $\mathring{U}^{\ge 0}_\Bbbk$ (\cite[\S 3]{BN04}) forces that the space of $\mathring{U}^{\ge 0}_\Bbbk$-module endomorphisms of $\bW ( \la )_{\Bbbk}$ is generated by the images of the imaginary weight vectors $\tilde{P}_{i,m\delta}$ for $i \in \tI$ and $m > 0$ (defined in \cite[(3.7)]{BN04}). In view of its descriptions (\cite[Proposition 3.17]{BN04} or \cite[Lemma 4.5]{CFK}), we find that $\tilde{P}_{i,m}$ acts on $\bv_\la \in \bW ( \la )_{\Bbbk}$ by zero if $\left< \al^{\vee}_i, \la \right> > m$. From this, we derive that the $P$-weight $\la$-part of $\bW ( \la )_{\Bbbk}$ is maximal possible as a cyclic module with a cyclic vector of $P$-weight $\la$ in the category of $\mathring{U}_\Bbbk^{\ge 0}$-modules that is $\mathring{U}_\Bbbk^0$-integrable and whose $P$-weights are contained in $\mathrm{Conv} \, W \la \subset P \otimes_\Z \R$. Therefore, (\ref{la-la-ext}) vanishes and we conclude Proposition \ref{proj-cover}.

\subsection{Frobenius splitting of $\bQ_{G,\tJ}$}\label{subsec:fsQ}

\begin{thm}\label{W-quot}
For each $\Lambda \in P^{\af}_+$, we have a surjective map $L ( \Lambda )_\C \rightarrow \bW^- ( \bar{\Lambda} )_\C$ of $\g [z^{-1}]_\C$-modules. In addition, this map yields a surjection $L ( \Lambda )_\Z \rightarrow \bW^- ( \bar{\Lambda} )_\Z$ of $U^-_\Z$-modules.
\end{thm}

\begin{proof}
By \cite[Theorem A]{KL17}, the graded $\g [z]_\C$-module $\theta^* ( L ( \Lambda )_\C )$ admits a filtration by the grading shifts of $\{ \bW ( \mu )_\C \}_{\mu \in P_+}$. Since the ($d$-)degrees of $L ( \Lambda )_\C$ is concentrated in $\Z_{\le 0}$ and the degree $0$-part of $L ( \Lambda )_\C$ is $V ( \bar{\Lambda} )_\C$, the first quotient of $\theta^* ( L ( \Lambda )_\C )$ in our filtration must be $\theta^* ( \bW^- ( \bar{\Lambda} )_\C )$. Hence, we obtain the surjection $\eta : L ( \Lambda )_\C \rightarrow \bW^- ( \bar{\Lambda} )_\C$ of $\g [z^{-1}]_\C$-modules.

Since the both modules share the $U^-_\C$-cyclic vector and compatible with the negative global basis, we conclude that the $\Z$-basis of $\bW^- ( \bar{\Lambda} )_\Z$ is obtained as a $\Z$-basis of $L ( \Lambda )_\Z$ that is not annihilated by $\eta$. Hence, we conclude that $\eta$ induces a surjection $L ( \Lambda )_\Z \rightarrow \bW^- ( \bar{\Lambda} )_\Z$ of $U^-_\Z$-modules.
\end{proof}

\begin{cor}\label{LW-def-eq}
Let $\Bbbk$ be a field. For $\Lambda \in P^{\af}_+$, we have a $U^-_\Bbbk$-module generator set $\{ \bu _m \}_{m \in \Z_{\ge 0}}$ of $\ker ( L ( \Lambda )_\Bbbk \rightarrow \bW^- ( \bar{\Lambda} )_\Bbbk )$ that satisfies:
\begin{itemize}
\item Each element $\bu_m$ satisfies $a_{\Lambda_m} \bu_m = \bu_m$ for some $\Lambda_m \in P^\af$;
\item For each $m \in \Z_{\ge 0}$, we have $\overline{\Lambda}_m \not\in \mathrm{Conv} \, W \overline{\Lambda} \subset P \otimes_\Z \R$.
\end{itemize}
\end{cor}

\begin{proof}
Note that $L ( \Lambda )_\Bbbk$ has at most countable rank over $\Bbbk$, which implies that the generator set is at most countable. As the both of $L ( \Lambda )_\Bbbk$ and $\bW^- ( \bar{\Lambda} )_\Bbbk$ admit $P^\af$-weight decompositions, we deduce the first assertion. Since the both modules are $\dot{U}^0_\Bbbk$-integrable and $U^-_\Bbbk$-cyclic, the second assertion follows by Proposition \ref{proj-cover}.
\end{proof}

\begin{prop}\label{LW-comm}
For each $\Lambda, \Gamma \in P^{\af}_+$, we have the following commutative diagram of $U^-_\Z$-modules:
$$
\xymatrix{
L ( \Lambda + \Gamma )_{\Z} \ar@{^{(}->}[r] \ar@{->>}[d]& L ( \Lambda )_\Z \otimes_\Z L ( \Gamma )_\Z\ar@{->>}[d]\\
\bW^- ( \overline{\Lambda + \Gamma} )_{\Z} \ar@{^{(}->}[r]^{\mathsf{m} \hskip 6mm} & \bW^- ( \overline{\Lambda} )_\Z \otimes_\Z \bW^- ( \overline{\Gamma} )_\Z.
}
$$
Moreover, all the maps define direct summands as $\Z$-modules.
\end{prop}

\begin{proof}
The injectivity of the top horizontal arrow and the fact that it defines a direct summand as $\Z$-modules is Corollary \ref{L-dualmult}.

The surjectivity of the vertical arrows are Theorem \ref{W-quot}. Since they are obtained by annihilating parts of $\Z$-bases, these maps define direct summands as $\Z$-modules.

Since all the modules are generated by the cyclic vectors $\bv_{\Lambda + \Gamma}$ or $\bv_{\Lambda} \otimes \bv_{\Gamma}$ as $\g [z^{-1}]_\C$-modules or $\g [z^{-1}]^{\oplus 2}_\C$-modules, Theorem \ref{WC-surj} (twisted by $\theta$) implies the injectivity of $\mathsf{m}$ after extending the scalar to $\C$. Hence, we deduce
\begin{equation}
\mathsf{m} ( \bW^- ( \overline{\Lambda + \Gamma} )_{\Z} ) \subset \bW^- ( \overline{\Lambda} )_\Z \otimes_\Z \bW^- ( \overline{\Gamma} )_\Z.\label{W-incl-Z}
\end{equation}

Therefore, it suffices to prove that (\ref{W-incl-Z}) has torsion-free cokernel (as a $\Z$-module) to complete the proof. By a repeated use of (\ref{W-incl-Z}), we arrive the setting of Theorem \ref{W-unit} in view of Theorem \ref{WC-surj}. Thus, the map $\mathsf{m}$ defines a direct summand of $\bW ( \overline{\Lambda} )_\Z \otimes_\Z \bW ( \overline{\Gamma} )_\Z$ as $\Z$-modules. 
\end{proof}

\begin{cor}\label{WXcomm-inj}
For each $\la, \mu \in P_+$ and $w \in W$, we have the following commutative diagram of $U_\Z^+$-modules:
$$
\xymatrix{
\bX ( \la + \mu )_{\Z} \ar@{^{(}->}[r] & \bX ( \la )_\Z \otimes_\Z \bX ( \mu )_\Z\\
\bW _w ( \la + \mu )_{\Z} \ar@{^{(}->}[r] \ar@{^{(}->}[u]& \bW _w ( \la )_\Z \otimes_\Z \bW _w ( \mu )_\Z\ar@{^{(}->}[u].
}
$$
Moreover, the vertical inclusions are compatible with positive global basis, and the horizontal inclusions define direct summands as $\Z$-modules. In addition, all the inclusions commute with the automorphism $\tau_{\beta}$ $(\beta \in Q^{\vee})$ of $\bX ( \la )_{\Z}, \bX ( \mu )_{\Z}$, and $\bX ( \la + \mu )_{\Z}$.
\end{cor}

\begin{proof}
In view of Proposition \ref{LW-comm}, the $w = w_0$ case follows from Lemma \ref{WX-comm} 1). Thanks to Theorem \ref{WC-surj} and Lemma \ref{WX-comm} 3), we deduce the general case from the $w = w_0$ case.
\end{proof}

Let $w \in W_\af$ and $\tJ \subset \tI$. We define $P_+$- and $P_{\tJ,+}$- graded $\Z$-modules:
$$R ^{\af} := \bigoplus_{\Lambda \in P_+^\af} L ( \Lambda )_\Z ^{\vee} \hskip 5mm \text{and} \hskip 5mm  R _{w}( \tJ ) := \bigoplus_{\la \in P_{\tJ,+}} \bW_{ww_0} ( \la )_\Z ^{\vee}.$$

\begin{lem}\label{R-welldef}
The $\Z$-duals of the horizontal maps in Proposition {\rm\ref{LW-comm}} equip $R ^{\af}$ and $R _{w}( \tJ )$ structures of $(P_+$- and $P_{\tJ,+}$-$)$ graded commutative rings. 
\end{lem}

\begin{proof}
The maps in Proposition \ref{LW-comm} are characterized as the $d$-degree zero maps of cyclic $U^-_\Z$-modules, that are unique up to a scalar. Therefore, the composition
$$\bW ( \la + \mu + \gamma )_\Z \hookrightarrow \bW ( \la + \mu )_\Z \otimes_\Z \bW ( \gamma )_\Z \hookrightarrow \bW ( \la )_\Z \otimes_\Z \bW ( \mu )_\Z \otimes_\Z \bW ( \gamma )_\Z$$
is the same map as
$$\bW ( \la + \mu + \gamma )_\Z \hookrightarrow \bW ( \la )_\Z \otimes_\Z \bW ( \mu + \gamma )_\Z \hookrightarrow \bW ( \la )_\Z \otimes_\Z \bW ( \mu )_\Z \otimes_\Z \bW ( \gamma )_\Z$$
for every $\la, \mu, \gamma \in P_+$ as the images of the cyclic vectors are the same. Taking their restricted duals implies the associativity of the multiplication of $R _{e}( \tJ )$. In view of Lemma \ref{WX-comm} 1) and Corollary \ref{WXcomm-inj}, we deduce  the associativity of the multiplication of $R _{w}( \tJ )$ for each $w \in W_\af$. The associativity of $R ^{\af}$ is proved similarly (cf. \cite{Kat18b}). The commutativity of $R ^{\af}$ and $R _{w}( \tJ )$ follow as the $\mathsf{q} = 1$ coproduct of $\dot{U}_\Z$ is symmetric (\cite[Lemma 3.1.4, \S 23.1.5]{Lus93}). 
\end{proof}

\begin{cor}[of the proof of Lemma \ref{R-welldef}]\label{transR}
For each $w \in W_\af$ and $\beta \in Q^{\vee}$, we have an isomorphism of $( R_w )_{\Z}$ and $( R_{wt_{\beta}} )_{\Z}$ as graded commutative rings equipped with $U_\Z^+$-actions up to grading twists $($given by Lemma {\rm\ref{WX-comm}}$)$. \hfill $\Box$
\end{cor}

We set $R := R_e$. Note that $R_{w} ( \tJ )\subset R_w$ is a subring. We also define
$$R^+ ( \tJ ) := \bigoplus_{\la \in P_{\tJ, +}} \mathrm{Span}_\Z \prod_{i \in \tI} ( \bX ( \varpi_i )_\Z^{\vee} )^{\left< \al_i^{\vee}, \la \right>} \subset \bigoplus_{\la \in P_+} \bX ( \la )_\Z ^{\vee} =: \widetilde{R},$$
where the multiplication is defined through the projective limit formed by the duals of Corollary \ref{WXcomm-inj}. Here we warn that the ($\Z$-)rank of some $P^\af$-weight space of $\bX ( \la )_\Z$ can be infinity, and hence the inclusion $R^+ \subsetneq \widetilde{R}$ has a huge cokernel. Note that the rank of the $P^\af$-weight spaces of $\bX ( \varpi_i )_\Z$ are bounded for each $i \in \tI$ (\cite[Proposition 5.16]{Kas02}), and hence $R^+ ( \tJ )$ has only countably many generators of $P_{\tJ,+}$-degrees $\{ \varpi_i \}_{i \in \tI \setminus \tJ}$.

By construction, the rings $R ^{\af}, R _{w}( \tJ )$, and $R^+ ( \tJ )$ are free over $\Z$.

For each $\la \in P_+$ and $w \in W_\af$, we have a unique $P^\af$-weight vector
\begin{equation}
\bv_{w\la}^{\vee} \in \bW_w ( \la )^{\vee}_{\Z}\label{vvee}
\end{equation}
with paring $1$ with $\bv_{w\la} \in \bW_w ( \la )_{\Z}$. This vector $\bv_{w\la}^{\vee}$ yields an $U^+_{\Bbbk}$-cocyclic vector of $\bW_w ( \la )_{\Bbbk}^{\vee}$ for a field $\Bbbk$. By the construction of the ring structure on $R_w$, we have $\bv_{ww_0\la}^{\vee} \cdot \bv_{ww_0\mu}^{\vee} = \bv_{ww_0(\la+\mu)}^{\vee}$ in $R_w$ for every $\la, \mu \in P_+$.

\begin{lem}\label{surj-rel}
For each $w,v\in W_\af$ and $\tJ \subset \tI$, the ring $R_{w}( \tJ )$ is a quotient of $R_{v}( \tJ )$ if $w \le_\si v$. In addition, the ring $R_{w}$ is a quotient of $R_{v}$ if and only if $w \le_\si v$.
\end{lem}

\begin{proof}
We have $\bW _{ww_0} ( \la )_\Z \subset \bW _{vw_0} ( \la )_\Z$ if and only if $\bv_{w w_0 \la} \in \bW_{vw_0} ( \la )_\Z$. Now we apply Lemma \ref{contain} to deduce the result.
\end{proof}

\begin{lem}
We have the following morphisms of rings with $U^+_\Z$-actions
$$R^+ \longrightarrow\!\!\!\!\! \rightarrow R \hookrightarrow R^{\af},$$
that admit $\Z$-module splittings, where the $\dot{U}_{\Z}$-action on $R^\af$ is twisted by $\theta$.
\end{lem}

\begin{proof}
Apply Proposition \ref{LW-comm} and Corollary \ref{WXcomm-inj}.
\end{proof}

For each $w \in W$ and $\tJ \subset \tI$, we set
$$( \bQ_{G, \tJ} ( w ) )_{\Z} := \mathrm{Proj} \, R_{w}( \tJ ) \hskip 5mm \text{ and } \hskip 5mm ( \bQ_{G, \tJ}^{\ra} )_{\Z} := \bigcup _{w \in W} ( \bQ_{G, \tJ} ( w ) )_{\Z}.$$
These schemes and indschemes are flat over $\Z$.

\begin{rem}\label{enh-emb}
In view of (\ref{mproj}), we have embeddings
\begin{equation}
\xymatrix{
\mathrm{Proj} \, R^+_{\bK} \ar[r] & \prod_{i \in \tI} \bP_{\bK} ( \bX ( \varpi_i )_{\bK}^{\wedge} )\\
( \bQ_{G}^{\ra} )_{\bK} \ar[r] \ar@{^{(}->}[u] & \prod_{i \in \tI} \bP_{\bK} ( \bX ( \varpi_i )_{\bK}^{\flat} ) \ar@{^{(}->}[u]},\label{projembamb}
\end{equation}
where we set
$$\bX ( \varpi_i )_{\bK}^{\wedge} := \prod_{n \in \Z} \bX ( \varpi_i )_{n, \bK} \hskip 5mm \text{and} \hskip 5mm \bX ( \varpi_i )_{\bK}^{\flat} := \bigcup_{m \in \Z} \prod_{n > m} \bX ( \varpi_i )_{n, \bK}$$
for each $i \in \tI$. 
Here all the spaces in (\ref{projembamb}) admit actions of $\SL( 2,i )_{\bK}$ ($i \in \tI$), while only the bottom two spaces in (\ref{projembamb}) admit actions of $G [\![z]\!]_{\bK}$. Nevertheless, the top two spaces in (\ref{projembamb}) have some advantages since they are $\theta$-stable (unlike the bottom two), and consequently also contain $\theta ( ( \bQ_{G}^{\ra} )_{\bK} )$.
\end{rem}

\begin{thm}[\cite{Kat18b} Corollary B]\label{flag-split}
Let $p$ be a prime. Then, the ring $R^{\af} _{\F_p}$ admits a Frobenius splitting, that is $\bI$- and $\bI^-$-canonically split. \hfill $\Box$
\end{thm}

\begin{thm}\label{R-can}
Let $p$ be a prime. The ring $R _{\F_p}$ admits a Frobenius splitting, that is $\bI$-canonically split.
\end{thm}

\begin{proof}
The $\bI$-canonical Frobenius splitting $\phi$ of $R^\af$ gives rise to the following maps, whose composition is the identity:
$$L ( \Lambda )_{\F_p} \stackrel{\phi^{\vee}}{\longrightarrow} L ( p \Lambda )_{\F_p} \longrightarrow L ( \Lambda )_{\F_p} \hskip 5mm \Lambda \in P^\af_+.$$

In view of Proposition \ref{LW-comm}, it prolongs to
$$
\xymatrix{
L ( \Lambda )_{\F_p} \ar[r]^{\phi^{\vee}} \ar@{->>}[d]^{\pi_{\Lambda}} & L ( p \Lambda )_{\F_p}\ar[r] \ar@{->>}[d]^{\pi_{p \Lambda}} & L ( \Lambda )_{\F_p} \ar@{->>}[d]\\
\bW^- ( \bar{\Lambda} )_{\F_p} \ar@{-->}[r]^{\phi^{\vee}_{\bW}} & \bW^- ( p \bar{\Lambda}  )_{\F_p}\ar[r] & \bW^- ( \bar{\Lambda} )_{\F_p}
} \hskip 5mm \Lambda \in P^\af_+.
$$
The right square is automatic (and is canonically defined) from the adjunction of the Frobenius push-forward (by taking the restricted dual). In order to show that $\phi$ descends to a Frobenius splitting of $R_{\F_p}$, it suffices to show that the dotted map $\phi^{\vee}_{\bW}$ is a well-defined linear map (induced from $\phi^{\vee}$ and so that the left square is commutative).

By Corollary \ref{LW-def-eq}, $\ker \, \pi_{\Lambda}$ is generated by the $P$-weight $( P \setminus \mathrm{Conv} \, W \bar{\Lambda} )$-part of $L ( \Lambda )_{\F_p}$. By the cyclicity of $L ( \Lambda )_{\F_p}$ as $U^-_\Z$-modules and Proposition \ref{B-can-inv}, we deduce that $\phi^{\vee} ( \ker \, \pi_{\Lambda} )$ is contained in the $U^-_{\F_p}$-submodule of $L ( p \Lambda )_{\F_p}$ generated by the $P$-weight $p ( P \setminus \mathrm{Conv} W \bar{\Lambda} )$-part of $L ( p \Lambda )_{\F_p}$. The latter is contained in $\ker \, \pi_{p\Lambda}$ by
$$p ( P \setminus \mathrm{Conv} \, W \bar{\Lambda} ) \subset P \setminus p \mathrm{Conv} \, W \bar{\Lambda}.$$

Therefore, we conclude that $\phi^{\vee}_{\bW}$ is a well-defined linear map, and hence $\theta^* ( R_{\F_p} )$ admits a Frobenius splitting induced from $\phi$. The unipotent part of the $\bI^-$-canonical splitting condition is in common with subrings. It remains to twist the grading given by $\al_0^{\vee}$ with that given by $- \vartheta^{\vee}$ and twist the $\bI^-$-action into the $\bI$-action by $\theta$ to conclude that our Frobenius splitting on $R_{\F_p}$ is $\bI$-canonical.
\end{proof}

\begin{cor}\label{compat}
Let $p$ be a prime, and let $w \in W$. The $\bI$-canonical splitting of $R_{\F_p}$ obtained in Theorem {\rm\ref{R-can}} induces an $\bI$-canonical splitting of $( R_w ) _{\F_p}$.
\end{cor}

\begin{proof}
We set $L^{ww_0} ( \Lambda )_\Z := U^-_\Z \bv_{ww_0 \Lambda}$, where $\Z \bv_{ww_0 \Lambda}$ is the $P^{\af}$-weight $ww_0 \Lambda$-part of $L ( \Lambda )_\Z$, that is rank one over $\Z$.

The subspace $\bW _{ww_0}^- ( \la )_\Z \subset \bW^- ( \la )_\Z$ is the image of $L^{ww_0} ( \Lambda )_{\Z} \subset L ( \Lambda )_\Z$ (with $\la = \bar{\Lambda}$) under Theorem \ref{W-quot} as $L^{ww_0} ( \Lambda )_{\Z}$ is spanned by a subset of $\bB ( \Lambda )$ (\cite[(0.3)]{Kas94}). Our Frobenius splitting $\phi$ is obtained from that of $R^\af_{\F_p}$, which is compatible with $\bigoplus_{\la \in P_+} L^{ww_0} ( \Lambda )_{\F_p}^{\vee}$ by \cite[Corollary B]{Kat18b}. Applying $\theta$, we conclude that $\phi$ must descend to a Frobenius splitting of $( R_{w} ) _{\F_p}$.
\end{proof}

\begin{cor}\label{af-uniq}
An $\bI$-canonical splitting of $R_{\F_p}$ is unique.
\end{cor}

\begin{proof}
The behavior of the vectors in (\ref{vvee}) (with $w = e$) under an $\bI$-canonial Frobenius splitting is uniquely determined as they form a polynomial ring isomorphic to $\F_p P_+$ such that each of its ($P_+$-)graded component is a multiplicity-free $P^\af$-weight space in $\bW ( \la )_{\F_p}^{\vee}$'s. By Proposition \ref{B-can-inv}, this completely determines the behavior of our splitting (through its dual map).
\end{proof}

\begin{cor}\label{af-comp}
An $\bI$-canonical splitting of $R_{\F_p}$ is compatible with $(R_w)_{\F_p}$ for every $w \in W_\af$ such that $w \le_\si e$.
\end{cor}

\begin{proof}
By \cite[Proposition 4.1.17 and Remark 4.1.18 (i)]{BK05} and \cite[Theorem 4.12]{Kat18} (the algebraic portion of the latter stems from \cite[Lemma 2.6]{Kas05}, that carries over to this setting; cf. \cite[Lemma 4.4 and Theorem 4.7]{Kat18}), we derive that an $\bI$-canonical Frobenius splitting of $( R_{w} )_{\F_p}$ ($w \in W$) gives rise to an $\bI$-canonical splitting of $R_{\F_p}$ that is compatible with $( R_{w} )_{\F_p}$ (arguing by restricting to the $\SL ( 2, i )$-actions for each $i \in \tI$). In particular, the $\bI$-canonical splitting of $(R_w)_{\F_p}$ ($w \in W$) also uniquely exist and compatible with that of $R_{\F_p}$ by Corollary \ref{af-uniq} and Corollary \ref{compat}. By Corollary \ref{transR}, we further deduce that the $\bI$-canonical splitting of $(R_w)_{\F_p}$ ($w \in W_\af$) uniquely exists.

Let $w \in W$ such that $s_0 w = s_{\vartheta} w t_{- w^{-1} \vartheta^{\vee}} \le_\si w$. Then, $R_{s_0w}$ is a quotient of $R_w$. Again by \cite[Proposition 4.1.17 and Remark 4.1.18 (i)]{BK05}, the set of $\bI$-canonical splittings of $( R_{s_0w} )_{\F_p}$ is in bijection with that of $( R_{w} )_{\F_p}$ compatible with $( R_{s_0 w} )_{\F_p}$. By Corollary \ref{transR} and the above paragraph, we find that the (unique) $\bI$-canonical splitting of $( R_{s_0 w} )_{\F_p}$ is compatible with $( R_w )_{\F_p}$, and hence also compatible with $R_{\F_p}$ and $( R_{t_{- w^{-1} \vartheta^{\vee}}} )_{\F_p}$. This forces the $\bI$-canonical splitting of $R_{\F_p}$ to be compatible with that of $( R_{u t_{- w^{-1} \vartheta^{\vee}}} )_{\F_p}$ for every $u,w \in W$ such that $-w^{-1}\vartheta^{\vee} \in \Delta_+^{\vee}$.

The set $( \{ -w^{-1}\vartheta^{\vee} \}_{w \in W} \cap Q^{\vee}_+ )$ is precisely the set of short positive coroots in $\Delta_+^{\vee}$. This spans $Q^{\vee}_+$ as monoids by inspection. Since $w \in W_\af$ with $w \le_\si e$ is written as $w = u t_{\beta}$ for some $u \in W$ and $\beta \in Q^{\vee}_+$ (\cite[Lecture 13, Proposition 1]{Pet97}), we conclude the assertion.
\end{proof}

\begin{thm}\label{F-uniq}
Let $p$ be a prime and let $\tJ \subset \tI$. The ring $R^+ ( \tJ )_{\F_p}$ admits a Frobenius splitting that is $\bI$- and $\bI^-$-canonically split. This splitting is compatible with $R_{w}( \tJ )_{\F_p}$ and the image of $R^+( \tJ )_{\F_p} \subset R^+ _{\F_p}$ under the quotient map $R^+_{\F_p} \rightarrow \!\!\!\!\! \rightarrow \theta^* ( ( R _{ww_0} )_{\F_p} )$ for each $w \in W_\af$.
\end{thm}

\begin{proof}
Since the case of $\tJ \neq \emptyset$ follows by the restriction to a part of the $P_+$-grading, we concentrate into the case $\tJ = \emptyset$.

The ring structure of $R^+_{\F_p}$ is determined by $R_{\F_p}$ through the application of $U^-_\Z$ before taking duals. By Corollary \ref{af-comp} (and its proof), it defines an $\bI$-canonical splitting $\phi$ of $R^+_{\F_p}$ compatible with $( R_w )_{\F_p}$ for each $w \in W_\af$. The ring $R^+_{\F_p}$ admits an $\SL ( 2, i )_{\F_p}$-action that integrates the actions of $E_i^{(n)}$ and $F_i^{(n)}$ ($n \in \Z_{>0}$) for each $i \in \tI_\af$. By \cite[Proposition 2.10]{BK05}, this splitting $\phi$ is also $\bI^-$-canonical.

The $\bI$-cocyclic $P^\af$-weight vector $\bv_{w w_0\la}^{\vee} \in \bW _{ww_0} ( \la )^{\vee}_{\F_p}$ is uniquely characterized by its $P^\af$-weight. Hence, we obtain a map
$$\bX ( \la )^{\vee}_{\F_p} \supset \mathrm{Span}_\Z \, \prod_{i \in \tI} ( \bX ( \varpi_i)_{\F_p}^{\vee} )^{\left< \al_i, \la \right>} \longrightarrow \!\!\!\!\! \rightarrow \bW _{ww_0} ( \la )_{\F_p}^{\vee} \longrightarrow \!\!\!\!\! \rightarrow \F_p \bv_{ww_0\la}^{\vee}.$$
It gives rise to the ring surjections
$$R^+_{\F_p} \longrightarrow \!\!\!\!\! \rightarrow ( R_w )_{\F_p} \longrightarrow \!\!\!\!\! \rightarrow \bigoplus_{\la \in P_+} \F_p \bv_{w w_0 \la}^{\vee}$$
that is compatible with $\phi$ by construction in the first surjection and by examining the $P^\af$-weights in the second surjection (we denote this composition surjective ring map by $\xi$). Consider the ideal
$$I (w) := R^+_{\F_p} \cap \bigcap _{g \in \bI^- ( \overline{\F}_p )} g ( \overline{\F}_p \otimes_{\F_p} \ker \, \xi )\subset R^+_{\overline{\F}_p}.$$
(Here the action of $\bI^- ( \overline{\F}_p )$ is obtained by the unipotent one-parameter subgroups $\{ \rho_{-\al_i}\}_{i \in \tI_\af}$ defined through the exponentials, that are well-defined as we have all the divided powers. The passage from $\F_p$ to $\overline{\F}_p$ is necessary to ensure the scheme-theoretic invariance since we want an intersection that is equivalent to the geometric $\Ga$-actions through $\rho_{- \al_i}$ for all $i \in \tI_\af$ by finding the Zariski dense subsets of $\Ga$. For this purpose, we want the image of each $\rho_{-\al_i}$ to be infinite, that cannot be achieved by sending $\Ga ( \F_p ) \cong \F_p$.) This ideal is the maximal $U_{\F_p}^-$-invariant ideal of $R^+_{\F_p}$ that is contained in $\ker \, \xi$. Let us denote the quotient ring by
$$Q = \bigoplus_{\la \in P_+} Q ( \la ) := R^+_{\F_p} / I ( w ).$$
By the construction of $I (w)$, we deduce that
$$U_{\F_p}^- \bv _{w w_0 \la} \subset Q ( \la )^{\vee}\subset \bX ( \la )_{\F_p}$$
for each $\la \in P_+$ (otherwise we can derivate a vector in $I ( w )$ to obtain a non-zero element of $\bigoplus_{\la \in P_+} \F_p \bv_{w w_0 \la}^{\vee}$). Since $\theta^*  ( \bW _{ww_0} ( - w_0 \la )_{\F_p} )$ is a cyclic $U_{\F_p}^-$-submodule of $\bX ( \la )_{\F_p}$, we have $U_{\F_p}^- \bv _{w w_0 \la} = \theta^*  ( \bW _{ww_0} ( - w_0 \la )_{\F_p} )$. In particular, we deduce a vector space surjection
$$R^+_{\F_p} / I ( w ) \longrightarrow \!\!\!\!\! \rightarrow \theta^* ( ( R _{ww_0} )_{\F_p} )$$
(cf. Corollary \ref{WXcomm-inj}). Since the RHS is naturally a ring, we conclude
$$R^+_{\F_p} / I ( w ) \cong \theta^* ( ( R _{ww_0} )_{\F_p} )$$
by the maximality of $I ( w )$.

The ideal $I (w) \subset R^+_{\F_p}$ also splits compatibly by $\phi$ (since $\phi$ is $\bI^-$-canonical and $\ker \, \xi$ splits compatibly). In particular, each $\theta^* ( ( R_{ww_0} )_{\F_p} )$ compatibly split under $\phi$ as required.
\end{proof}

\begin{cor}\label{FQ-uniq}
For each $\tJ \subset \tI$, the indscheme $( \bQ_{G, \tJ}^{\ra} )_{\F_p}$ admits an $\bI$- and $\bI^-$-canonical Frobenius splitting that is compatible with $\bQ_{G, \tJ} ( w )_{\F_p}$ for each $w \in W_\af$.
\end{cor}

\begin{proof}
The condition of the canonical splitting can be checked by line bundle twists, and hence we only need to show whether the Frobenius splitting of $R_w (\tJ) _{\F_p}$ descends to $\bQ_{G, \tJ} ( w )_{\F_p}$ for each $w \in W_\af$. Since $\bQ_{G, \tJ} ( w )_{\F_p}$ is the quotient of an open subset of $X := \mathrm{Spec} \, R_w (\tJ) _{\F_p}$ by $H$ (corresponding to the $P_{\tJ,+}$-grading) and our Frobenius splitting is $H$-fixed, it suffices to see whether the localization to the non-irrelevant locus preserves the Frobenius splitting. This follows if the localization of $\mathsf{Fr}_* \mathcal O_X = \mathcal O_{X^{(1)}}$ as a $\mathsf{Fr}_* \mathcal O_X$-module and as a $\mathcal O_X$-module are the same. It holds as the localization by a multiplicative set $S$ is the same as the localization by $S^p$.
\end{proof}

Recall that a scheme $X$ defined over a field $\Bbbk$ is called weakly normal if every finite bijective birational ($\Bbbk$-)morphism $f : Y \rightarrow X$ from a (reduced) scheme over $\Bbbk$ is in fact an isomorphism (\cite{Man80} and \cite[\S 1.2.3]{BK05}).

\begin{cor}\label{FQ-wn}
For each $\tJ \subset \tI$, the indscheme $( \bQ_{G, \tJ}^{\ra} )_{\F_p}$ and the schemes $\bQ_{G, \tJ} ( w )_{\F_p}$ $(w \in W_\af)$ are reduced. In addition, $\bQ_{G, \tJ} ( w )_{\F_p}$ is weakly normal.
\end{cor}

\begin{proof}
For the first assertion, apply \cite[Proposition 1.2.1]{BK05} to Corollary \ref{FQ-uniq}. For the second assertion, apply \cite[Proposition 1.2.5]{BK05} to Corollary \ref{compat}.
\end{proof}

\section{Frobenius splitting of quasi-map spaces}

We retain the settings of the previous section. In particular, we sometimes work over a ring or a non-algebraically closed field. Moreover, the notational convention explained in the beginning of \S \ref{sec:fsQ} continue to apply.

\subsection{The scheme $\sQ'_\tJ ( v, w )$ and its Frobenius splitting}\label{defQ'}

Let $v,w \in W_\af$ and $\tJ \subset \tI$. We set
$$R ^v _w ( \tJ ) := R^+ ( \tJ ) / \left( \ker ( R^+ ( \tJ ) \to R_w ( \tJ ) ) + \ker ( R^+ ( \tJ ) \to \theta^* ( R _{vw_0} ( \tJ' ) ) \right),$$
where we have $\ker ( R^+ ( \tJ ) \to \theta^* ( R _{vw_0} ( \tJ' ) ) ) = \ker ( R^+ ( \tJ ) \to \theta^* ( R _{vw_0} ) )$ for
$$\tJ' := \{i \in \tI \mid -w_0 \al_i = \al_j, \exists j \in \tJ \} \subset \tI.$$
By construction, $R ^v _w ( \tJ )$ is a $P_{\tJ,+}$-graded ring. We set
$$R ^v _w ( \tJ ) := \bigoplus_{\la \in P_{\tJ,+}} R ^v _w ( \tJ, \la ).$$

\begin{lem}\label{Qsurj-mult}
For each $w, v \in W_\af$ and $\tJ \subset \tI$, the multiplication map
$$R ^v _w ( \tJ, \la ) \otimes_\Z R ^v _w ( \tJ, \mu ) \rightarrow R ^v _w ( \tJ, \la + \mu ) \hskip 10mm (\la,\mu \in P_{\tJ,+})$$
is surjective.
\end{lem}

\begin{proof}
We have a quotient
$$R _w ( \tJ ) = \bigoplus_{\la \in P_{\tJ,+}} \bW_{ww_0} ( \la )^{\vee}_\Z \longrightarrow \!\!\!\!\! \rightarrow \bigoplus_{\la \in P_{\tJ,+}} R ^v _w ( \tJ, \la ) = R ^v _w ( \tJ )$$
of homogeneous rings. Corollary \ref{WXcomm-inj} implies that the multiplication map of $R _w ( \tJ )$ is surjective. Hence, so is the quotient ring.
\end{proof}

We set
$$\sQ'_{\tJ} ( v, w ) := \mathrm{Proj} \, R ^v _w ( \tJ ),$$
where our definition of $\mathrm{Proj}$ is (\ref{mproj}). In case $v = w_0 t_{\beta}$ for $\beta \in Q^{\vee}$, we may write $R ^v _w ( \tJ )$ and $\sQ'_{\tJ} ( v, w )$ by $R ^\beta _w ( \tJ )$ and $\sQ'_{\tJ} ( \beta, w )$, respectively.

\begin{lem}
For each $w, v \in W$, the Chevalley involution induces an isomorphism
$$\sQ' ( v, w ) \stackrel{\cong}{\longrightarrow} \sQ' ( w w_0, v w_0 ).$$
\end{lem}
\begin{proof}
Apply Lemma \ref{WX-comm} to the construction of $R ^v _w$ to deduce an isomorphism
$$\theta^* : R ^v _w \stackrel{\cong}{\longrightarrow}R ^{ww_0} _{vw_0}$$
of graded rings, that yields the assertion.
\end{proof}

\begin{lem}\label{Q-flat}
For each $w,v \in W_\af$ and $\tJ \subset \tI$, the scheme $\sQ'_{\tJ} ( v, w )$ is flat over $\Z$.
\end{lem}

\begin{proof}
The ring $R^+ ( \tJ )$ has a $\Z$-basis that is dual to $\bigsqcup_{\la \in P_{\tJ,+}} \bB ( \bX ( \la ) )$. The rings $R _w ( \tJ )$ and $\mathrm{Im} \, \left( R^+ ( \tJ ) \rightarrow \theta^* ( R _{vw_0} ) \right)$ are quotients by subsets of such basis elements by Lemma \ref{WX-comm}. Hence, we have a free $\Z$-basis of $R ^v _w( \tJ )$ as required.
\end{proof}

\begin{lem}\label{Q'finite}
For each $w, v \in W_\af$ and $\tJ \subset \tI$, the scheme $\sQ'_{\tJ} ( v, w )$ is projective $($of finite type$)$ over $\Z$.
\end{lem}

\begin{proof}
By \cite[Proposition 5.16]{Kas02}, we have $\mathrm{rank}_\Z \, R ^v _w ( \varpi_i ) < \infty$ for each $i \in \tI$. By Lemma \ref{Qsurj-mult}, this forces $\sQ'_{\tJ} ( v, w )$ to be a projective scheme (that usually implies finite type by definition. Here we explicitly include it since our $\mathrm{Proj}$ does not yield a finite type scheme in general) as required.
\end{proof}

\begin{lem}\label{non-empty}
Let $\tJ \subset \tI$. We have $\sQ'_{\tJ} ( v, w ) \neq \emptyset$ if $v \le_\si w$.
\end{lem}

\begin{proof}
We have $R ^v _w ( \tJ, \la ) \neq \{ 0 \}$ if $\bv_{v w_0 \la} \in \bW_{ww_0} ( \la )_\Z$ by Lemma \ref{WX-comm} 4). Here $\bv_{v w_0 \la} \in \bW_{ww_0} ( \la )_\Z$ is equivalent to $\bW_{vw_0} ( \la )_\Z\subset \bW_{ww_0} ( \la )_\Z$. Now apply Lemma \ref{contain} to obtain the assertion.
\end{proof}

\begin{lem}\label{trans}
Let $v,w \in W_\af$, $\beta \in Q^{\vee}$, and $\tJ \subset \tI$. We have $\sQ'_{\tJ} ( v, w ) \cong \sQ'_{\tJ} ( vt_{\beta}, w t_{\beta} )$.
\end{lem}

\begin{proof}
We borrow notation from Lemma \ref{WX-comm}. By the definition of our ring $R^v_w ( \tJ )$, the assertion follows if
$$\bW _{wt_{\beta}w_0}( \la )_\Z = \tau_{w_0\beta} \bW _{ww_0}( \la )_\Z \hskip 5mm \text{and} \hskip 5mm \tau_{w_0 \beta} \theta^* ( \bW _{v}( -w_0 \la )_\Z ) = \theta^* ( \bW _{vt_{\beta}}( -w_0 \la )_\Z )$$
holds for each $w,v \in W_\af$, $\beta \in Q^{\vee}$, and $\la \in P_+$. These assertions themselves follow by chasing the weights of the cyclic vectors.
\end{proof}

\begin{lem}\label{Q-split}
Let $p$ be a prime and let $\tJ \subset \tI$. For each $w,v \in W_\af$, the ring $R ^v _w ( \tJ )_{\F_p}$ admits a Frobenius splitting that is compatible with the quotient $R^{v'}_{w'} ( \tJ ) _{\F_p}$ for $v',w' \in W_\af$ such that $v \le_\si v' \le_\si w' \le_\si w$. In particular, the scheme $\sQ'_{\tJ} ( v, w )_{\F_p}$ is reduced and weakly normal. 
\end{lem}

\begin{proof}
By construction, $\ker ( R^+ ( \tJ )_{\F_p} \to ( R _w )_{\F_p} )$ and $\ker ( R^+ ( \tJ )_{\F_p} \to \theta^* ( ( R _{vw_0} )_{\F_p} ) )$ are ideals of a ring $R^+ ( \tJ )_{\F_p}$ that are compatible with the canonical Frobenius splitting of $R^+ ( \tJ )_{\F_p}$ by Theorem \ref{F-uniq}. Hence, so is their sum. It must be compatible with every quotient of the form $R^{v'}_{w'} ( \tJ ) _{\F_p}$ with the above condition by Lemma \ref{surj-rel}, Corollary \ref{compat}, Lemma \ref{WX-comm} 4), and Lemma \ref{non-empty}. This proves the first assertion. We apply \cite[Proposition 1.2.1 and 1.2.5]{BK05} to deduce the second assertion.
\end{proof}

\begin{cor}\label{Q'-cFS}
Let $p$ be a prime. For each $w, v \in W$, the scheme $\sQ'_{\tJ} ( v, w )_{\F_p}$ admits a Frobenius splitting compatible with $\sQ'_{\tJ} ( v', w' )_{\F_p}$ for every $w',v' \in W_\af$ such that $v \le_\si v' \le_\si w' \le_\si w$.
\end{cor}
\begin{proof}
Apply Theorem \ref{F-rel} to Lemma \ref{Q-split}.
\end{proof}

\begin{rem}
Unlike the case of Corollary \ref{FQ-uniq}, the space $\mathrm{Spec} \, R ^v _w ( \tJ )_{\F_p}$ is not irreducible in general. In fact, we discard some of the irreducible components of $\mathrm{Spec} \, R ^v _w ( \tJ )_{\F_p}$ from Lemma \ref{Q-split} to Corollary \ref{Q'-cFS}.
\end{rem}

\begin{cor}\label{Q'red}
The ring $R^v_w ( \tJ )$ is reduced. In particular, the scheme $\sQ'_{\tJ} ( v, w )$ is reduced.
\end{cor}

\begin{proof}
By Lemma \ref{Q-flat}, every non-zero element of $R^v_w ( \tJ )$ is annihilated by reduction mod $p$ for only finitely many primes. Now it remains to apply Lemma \ref{Q-split}.
\end{proof}

\subsection{Modular interpretation of $\bQ_{G,\tJ}^{\ra}$}\label{sec:mi}

We have an identification
$$W_\af \cong N_{G (\!(z)\!)} ( H ( \bK ) ) / H ( \bK )$$
regardless of the (algebraically closed) base field $\bK$. We denote a lift of $w \in W_\af$ in $N_{G (\!(z)\!)} ( H ( \bK ) ) / H ( \bK )$ by $\dot{w}$.

\begin{lem}\label{bQ-dense}
For each $w \in W$ and $\tJ \subset \tI$, the scheme $\bQ_{G,\tJ} ( w )_\bK$ contains an affine Zariski open $\bI_\bK$-orbit $\bO ( \tJ, w )_\bK$ that is isomorphic to
$$\bI_\bK / \left( H_\bK \cdot ( \mathrm{Ad} ( \dot{w}\dot{w}_0 ) ( [P (\tJ),P(\tJ)] (\!(z)\!) ) \cap \bI_\bK)\right)$$
as a scheme over $\bK$. $($By abuse of notation, here we identify the set of $\bK$-valued points $( \mathrm{Ad} ( \dot{w}\dot{w}_0 ) ( [P (\tJ),P(\tJ)] (\!(z)\!) ) \cap \bI_\bK)$ with its Zariski closure in $\bI_\bK)$. It is an open neighbourhood of the unique $(H \times \Gm)_\bK$-fixed point of $\bO ( \tJ, w )_\bK$.
\end{lem}

\begin{proof}
Recall that $\bv_{ww_0\la}^{\vee} \cdot \bv_{ww_0\mu}^{\vee} = \bv_{ww_0(\la + \mu)}^{\vee}$ for each $\la, \mu \in P_+$. Since $\bW_{ww_0} ( \la )$ is compatible with the positive global basis, the ring
\begin{equation}
\Z [w] := \sum_{\la \in P_{\tJ,+}} ( \bv_{ww_0\la}^{\vee} )^{-1} \bW_{ww_0} ( \la )^{\vee} _\Z \subset ( U _\Z^+ )^{\vee}\label{emb-coord}
\end{equation}
admits its dual basis. By construction, $\Z [w]$ is the coordinate ring of an affine Zariski open set of $\bQ _G ( w )_\Z$. In addition, it inherits a natural $P^\af$-grading from $R_{w} ( \tJ )$. Therefore, $\Z [w]$ defines an open neighbourhood of a $(H \times \Gm)$-fixed point of $\bQ_{G,\tJ} ( w )$ obtained by the linear forms $\{\bv_{ww_0\la}\}_\la$. We set $\C [w] := \C \otimes_\Z \Z [w]$ and $\bK [w] := \bK \otimes_\Z \Z [w]$.

The defining relation of $( \C [e] )^{\vee}$ in terms of the $U_\C^+$-action is
$$
U_\C^+ ( \mathrm{Ad} ( \dot{w}_0 ) ( [\mathfrak p ( \tJ), \mathfrak p ( \tJ)] ) \otimes \C [z] \cap \gI ).
$$
by (the limit of) Chari-Fourier-Khandai \cite[Proposition 3.3]{CFK}. Since $\bW _{uw_0} ( \la )_\Z \subset \bW _{w_0} ( \la )_\Z$ for each $u \in W$, the defining relation of $( \C [u] )^{\vee}$ in terms of the $U_\C^+$-action is
$$U_\C^+ ( \mathrm{Ad} ( \dot{u} \dot{w}_0 ) ( [\mathfrak p ( \tJ ), \mathfrak p ( \tJ )] ) \otimes \C [z] \cap \gI )$$
by applying the action of $\dot{u} \in N_G ( H )_\C$ that lifts $u \in W$. In particular, we have
\begin{equation}
\mathrm{Spec} \, \C [w] \cong \bI_\C / \left( H_\C \cdot \mathrm{Ad} ( \dot{w} \dot{w}_0 ) ( [P ( \tJ )_\C, P ( \tJ )_\C][\![z]\!] ) \cap \bI_\C \right)\label{quotC}
\end{equation}
as schemes over $\C$. We put $\bI_\Z ^1 := [ \bI_\Z, \bI_\Z]$. Let $\mathbf L_{\C}$ and $\mathbf R_{\C}$ be pro-unipotent subgroups of $\bI_\C ^1$ whose closed points are 
$$\mathrm{Specm} \, \C [w] \hskip 5mm \text{and} \hskip 5mm \mathrm{Ad} ( \dot{w} \dot{w}_0 ) ( [P ( \tJ )_\C, P ( \tJ )_\C][\![z]\!] ) \cap \bI^1_\C,$$
respectively and they are stable by the natural $( H \times \Gm )_\C$-action on $\bI_\C$. The isomorphism (\ref{quotC}) gives rise to an isomorphism
\begin{equation}
m_\C : \mathbf L_{\C} \times \mathbf R_{\C} \stackrel{\cong}{\longrightarrow} \bI_\C ^1\label{multC}
\end{equation}
of schemes over $\C$, where $m_\C$ is the multiplication map as all the groups are pro-unipotent and the isomorphism between the sets of closed points is compatible with truncations by $P^\af$-weight considerations. The Hopf algebra $\Z [\bI ^1_\Z] := ( U ^+_\Z )^{\vee}$ is the coordinate ring of $\bI^1_{\Z}$ (cf. \cite[Theorem 1.3]{Kat18b}). By sending $\Z [\bI ^1]$ by the restriction morphisms $\C [\bI ^1_{\C}] \to \C [\mathbf L_{\C}]$ and $\C [\bI ^1_{\C}] \to \C [\mathbf R_{\C}]$, we have the corresponding group schemes $\mathbf L_{\Z}$ and $\mathbf R_{\Z}$ over $\Z$ (we denote their coordinate rings as $\Z [\mathbf L]$ and $\Z [\mathbf R]$, respectively). For each real root $\alpha \in \Delta_{\af, +}$, one of the restricted dual rings $\C [\mathbf L_{\C}]^{\vee}$ nor $\C [\mathbf R_{\C}]^{\vee}$ contains the image of a primitive element of $U^+_\Z$ with $\wth$-weight $\al$ (obtained by conjugations of $\{E_i\}_{i \in \tI_\af}$'s cf. \cite[Proposition 40.1.3]{Lus93} or \cite[Lemma 6.6]{Gar78}), and hence the corresponding one-parameter subgroup lands in either of $\mathbf L_{\Z}$ or $\mathbf R_{\Z}$. In view of \cite[Theorem 3.13]{BN04} (or \cite[Theorem 5.8]{Gar78}), they generates a closed normal subgroup scheme $\mathbf N_\Z^-$ of $\mathbf L_\Z$ that is a projective limit of extensions of $\Ga$ over $\Z$. (The same is true and exhaust the whole $\mathbf R_{\Z}$ if $\tJ = \emptyset$.)

We examine the action of the imaginary PBW generators $\{ \tilde{P}_{i,m\delta}\}_{i \in \tI \setminus \tJ, m > 0}$ (weight of $\tilde{P}_{i,m\delta}$ is $m \delta$) in \cite[(3.7)]{BN04}. By applying them on the direct sum of $P$-weight $\la$-parts of $\bX ( \la )_\C$ for all $\la \in P_{\tJ, +}$, we obtain a quotient group scheme $\mathbf L_{\Z} \to \mathbf T_{\Z}$ whose kernel is $\mathbf N_\Z^-$. By \cite[Proposition 3.22]{BN04}, the group scheme $\Spec \, \Z [\tilde{P}_{i,m\delta}\mid m > 1]^{\vee}$ for each $i \in \tI$ is isomorphic to a projective limit of extensions of $\Ga$ (given by truncations with respect to the duals of $\{ \tilde{P}_{i,m\delta} \}_{m > N}$ for $N \in \Z_{> 0}$), which is flat over $\Z$. Thus, so is $\mathbf T_{\Z}$.

From these, we deduce that $\mathbf L_{\Z} \subset \bI_\Z^1$ defines a group subscheme over $\Z$ such that $\Z [\bI_\Z^1] \rightarrow \Z [\mathbf L]$ splits as $\Z$-modules. The $P$-weight decomposition of the coordinate ring defines a splitting $\mathbf T_{\Z} \subset \mathbf L_{\Z} \subset \bI^1_{\Z}$. Since $[L ( \tJ ), L ( \tJ )]$ defines a product of connected and simply connected simple algebraic group, its (pro-unipotent part of the) Iwahori subgroup (we take the product over simple factors) defines a subgroup scheme $\mathbf R^0_\Z \subset \mathbf R_{\Z}$ such that
\begin{equation}
\mathbf R^0_\Z \cap \mathbf T_\Z = \Spec \, \Z\label{RT-int}
\end{equation} (scheme-theoretically). Since $[L ( \tJ ), L ( \tJ )] \subset G$ defines a closed group subscheme over $\Z$ (thanks to the fact that the corresponding modified enveloping algebras share the same $\Z$-basis by \cite{Kas91,Kas94,Lus09}), we deduce that the map $\Z [\bI_\Z^1] \rightarrow \Z [\mathbf R^0_\Z]$ must be surjective. The one-parameter subgroups corresponding to the unipotent radical of $P ( \tJ )$ generates a subgroup $\mathbf R^+_{\Z} \subset \mathbf R_{\Z} \subset \bI_\Z^1$ such that $\Z [\bI_\Z^1] \rightarrow \Z [\mathbf R^+_\Z]$ is surjective. In view of the $P$-weight decomposition, we deduce that $\mathbf R^0_\Z \times \mathbf R^+_{\Z} \cong \mathbf R_{\Z}$. From (\ref{RT-int}), we derive $\Z [\mathbf L] \cap \Z [\mathbf R] = \Z$. Therefore, we have a surjection
$$m^*_\Z : \Z [\bI^1] \longrightarrow\!\!\!\!\!\rightarrow \Z [\mathbf L] \otimes_\Z \Z[\mathbf R].$$
The ideal $\ker \, m^*_\Z$ must be flat over $\Z$, and hence $\ker \, m^*_\C \neq 0$ if $\ker \, m^*_\Z \neq 0$. We know that $m^*_\C$ must be an isomorphism by (\ref{multC}). This is a contradiction, and hence $m^*_\Z$ must be an isomorphism (cf. \cite[Theorem 5.8]{Gar78}).

In view of the flatness of the group schemes over $\Z$ afforded by the above, we conclude that
$$\mathbf R_{\bK} \subset \mathrm{Ad} ( \dot{w} \dot{w}_0 ) ( [P ( \tJ ), P ( \tJ )][\![z]\!] ) \cap \bI^1_{\bK}$$
is in fact an isomorphism as schemes. By construction, $\Z [w]$ is precisely the subring of $\Z [\bI^1_\bK]$ that is invariant under the $\mathbf R_{\Z}$-action. It follows that the image of the composition map
$$\Z [w] \hookrightarrow \Z [\bI^1] \stackrel{\cong}{\longrightarrow} \Z [ \mathbf L  ] \otimes \Z [ \mathbf R  ]$$
is equal to $\Z [ \mathbf L  ] \otimes_\Z \Z$. Hence, we have
$$\Spec \, \bK [w] = \mathbf L_{\bK} \cong \bI_{\bK} / \left( H_{\bK} \cdot \mathrm{Ad} ( \dot{w} \dot{w}_0 ) ( [P ( \tJ ), P ( \tJ )][\![z]\!] ) \cap \bI_{\bK} \right).$$

Therefore, each $\bQ_G ( w )_\bK$ contains $\mathrm{Spec} \, \bK [w]$ as a Zariski open $\bI$-orbit, and it admits a unique $(H \times \Gm)_\bK$-fixed point as required.
\end{proof}

\begin{cor}\label{Re-int}
The ring $R ( \tJ )_\bK$ is integral for each $\tJ \subset \tI$.
\end{cor}

\begin{rem}
We refer Proposition \ref{R-normal} for general case of Corollary \ref{Re-int}.
\end{rem}

\begin{proof}[Proof of Corollary \ref{Re-int}]
It suffices to prove the case $\tJ = \emptyset$ since a subring of an integral ring is integral. We borrow notation from the proof of Lemma \ref{bQ-dense}. We have an inclusion
$$R _e = \bigoplus_{\la \in P_+} \bW ( \la )_{\Z}^{\vee} \hookrightarrow \Z [e] \otimes \bigoplus_{\la \in P_+} \Z \bv_{w_0 \la}^{\vee}$$
as the Rees construction of (\ref{emb-coord}). Taking tensor product $\bK \otimes_\Z \bullet$ preserves the inclusion in view of the definition of $\Z [e]$. Hence, $( R_e )_{\bK}$ is integral as a subring of an integral ring.
\end{proof}

\begin{prop}\label{wbQ-mod}
Let $\tJ \subset \tI$. The indscheme $( \bQ_{G,\tJ} ^{\ra} )_\bK$ admits a $G (\!(z)\!)$-action. We have a subset $\bQ_\bK'$ of the set of $\bK$-valued points of $( \bQ_{G,\tJ} ^{\ra} )_\bK$ that is in bijection with
$$G (\!(z)\!) / \left( H ( \bK ) \cdot [P ( \tJ ), P ( \tJ )] (\!(z)\!) \right).$$
\end{prop}

\begin{proof}
In view of Lemma \ref{bQ-dense}, the proof for general $\tJ$ is completely parallel to the case of $\tJ = \emptyset$. Hence, we concentrate into the case $\tJ = \emptyset$ during this proof.

Let $w \in W_\af$ and $\beta \in Q^{\vee}$. We have an isomorphism $R_w \cong R_{w t_{\beta}}$ as rings with $U_\Z^+$-action by Corollary \ref{transR}.

In particular, we have an isomorphism $\bQ_G ( w )_\bK \cong \bQ_G ( w t_{\beta} )_\bK$ of schemes with $\bI$-actions for each $\beta \in Q^{\vee}$. This implies that $\bQ_G ( w t_{\beta} )_\bK$ has a Zariski open subset of the shape
\begin{equation}
\bI / \left( H_ \cdot ( \mathrm{Ad} ( \dot{w} \dot{t}_{\beta} \dot{w}_0 ) ( N (\!(z)\!) ) \cap \bI)\right) \cong \bI /  \left( H \cdot ( \mathrm{Ad} ( \dot{w}\dot{w}_0 ) ( N (\!(z)\!) ) \cap \bI)\right), \label{orbit-desc}
\end{equation}
where we used that $\mathrm{Ad} ( \dot{w}_0 )  ( N (\!(z)\!) )$ is invariant under the $\mathrm{Ad} ( \dot{t}_{\beta} )$-action.

The ring $( R_{w} )_{\bK}$ admits an action of $\SL ( 2, i )$ whenever $s_i w \le_\si w$ ($i \in \tI_\af$) since each $\bW_{ww_0} ( \la )_{\bK}$ ($\la \in P_+$) admits an action of $\SL ( 2, i )$ by Corollary \ref{SL2-stable}. Hence, $\mathrm{Proj} \, ( R_{w} )_{\bK}$ admits an action of $\SL ( 2, i )$ if $s_i w \le_\si w$. In particular, the pro-algebraic group $\bI ( i )$ acts on $\mathrm{Proj} \, ( R_{w} )_{\bK}$ if $s_i w \le_\si w$. Hence, the ind-limit
$$( \bQ_G^{\ra} )_\bK = \bigcup_{w \in W_\af} \bQ_G ( w )_{\bK} = \varinjlim_w \mathrm{Proj} \, ( R_{w} )_\bK$$
admits an action of $\bI ( i )$ for each $i \in \tI_\af$ (that coincide on $\bI$). By rank two calculations, they induce an action of $N_{G(\!(z)\!)} ( H ( \bK ) )$ on $( \bQ_G^{\ra} )_{\bK}$. The intersections of $N_{G(\!(z)\!)} ( H ( \bK ) )$ and $\bI ( \bK )$ or $\bI ( i ) ( \bK)$ ($i \in \tI_\af$) inside $G (\!(z)\!)$ define common actions on $( \bQ_G^{\ra} )_{\bK}$. The system of groups $(\bI (\bK), N_{G(\!(z)\!)} ( H ( \bK ) ), \bI(i) (\bK) ; i \in \tI_\af )$ (in the sense of \cite[Definition 5.1.6]{Kum02}) admits a map from the system of groups in \cite[\S 6.1.16]{Kum02} for $\tg$ (with exponential maps replaced by one-parameter subgroups). Hence, \cite[Theorem 6.1.17]{Kum02} asserts that $G(\!(z)\!)$ acts on $( \bQ_G^{\ra} )_{\bK}$.

The Bruhat decomposition of $\SL ( 2, i )$ asserts that $\bO ( w )_{\bK} \sqcup \bO ( s_i w )_{\bK}$ admits a $\SL ( 2, i )$-action. This induces an action of $\SL ( 2, i ) ( \bK )$ ($i \in \tI_\af$) on the union
$$\bQ_\bK' := \bigsqcup_{w \in W_\af} \bO ( w )_{\bK} ( \bK ) \subset \bQ_G^{\ra}( \bK ).$$
Taking into account the fact that each $\bO ( w )_{\bK}$ admits an $\bI$-action (Lemma \ref{bQ-dense}), we conclude that $\bQ_\bK'$ admits an action of $G (\!(z)\!)$ thanks to the Iwasawa decomposition (cf. \cite[Theorem 2.5]{IM65})
$$G (\!(z)\!) = \bigsqcup_{w \in W_\af} \bI ( \bK ) \dot{w} \dot{w}_0 H ( \bK ) \cdot N (\!(z)\!)$$
as required.
\end{proof}

For $\la,\mu \in P_+$, we have a unique injective $\dot{U}_\Z^0$-module map
$$V ( \la + \mu )_\Z \longrightarrow V ( \la )_\Z \otimes_\Z V ( \mu )_\Z$$
obtained by sending $\bv^0_{\la + \mu}$ to $\bv^0_\la \otimes \bv^0_{\mu}$, that is in fact a $\Z$-direct summand (Corollary \ref{L-dualmult}). By extending the scalar, we obtain a unique injective $(\dot{U}_\bK^0, \bK [\![z]\!])$-bimodule map
$$\eta_{\la, \mu} : V ( \la + \mu )_\bK \otimes \bK [\![z]\!] \longrightarrow \left( V ( \la )_\bK \otimes \bK [\![z]\!] \right) \otimes_{\bK[\![z]\!]} \left( V ( \mu )_\bK \otimes \bK [\![z]\!] \right).$$

\begin{lem}\label{eta-compat}
For each $\la, \mu, \gamma \in P_+$, we have
$$\eta_{\la + \mu, \gamma} \circ ( \mathrm{id} \boxtimes \eta_{\la, \mu} ) = \eta_{\la, \mu + \gamma} \circ ( \mathrm{id} \boxtimes \eta_{\mu, \gamma}).$$
\end{lem}

\begin{proof}
Straight-forward from the construction (cf. Lemme \ref{R-welldef}).
\end{proof}

\begin{prop}\label{bQ-mod}
Assume that $\mathsf{char} \, \bK \neq 2$. For each $w \in W_\af$ and $\tJ \subset \tI$, we have an $\bI$-equivariant rational map
$$\psi_w : \bQ_{G,\tJ} ( w )_{\bK} \dashrightarrow \bigcup_{m \in \Z} \prod_{i \in \tI \setminus \tJ} \P_\bK ( V ( \varpi_i ) _\bK \otimes z^m \bK [\![z]\!]) = \prod_{i \in \tI \setminus \tJ} \P_\bK ( V ( \varpi_i ) _\bK \otimes \bK (\!(z)\!)),$$
that gives rise to a $G(\!(z)\!)$-equivariant rational map
$$\psi : ( \bQ_{G,\tJ} ^{\mathrm{rat}} )_{\bK} \dashrightarrow \prod_{i \in \tI \setminus \tJ} \P_\bK ( V ( \varpi_i ) _\bK \otimes \bK (\!(z)\!)).$$
In addition, the set $( \mathrm{Im} \, \psi ) ( \bK )$ defines a closed $($ind-$)$subscheme of $\prod_{i \in \tI \setminus \tJ} \P_\bK ( V ( \varpi_i ) _\bK \otimes \bK (\!(z)\!))$.
\end{prop}

\begin{proof}
We have a surjective map
$$\bW ( \varpi_i )_\bK\rightarrow V ( \varpi_i )_\bK \otimes \bK [z]$$
as $\dot{U}^{\ge 0}_\bK$-modules since we have the corresponding map over $\dot{U}^{\ge 0}_\Z$ such that the $P$-weight $\varpi_i$-part is the same (and $V ( \varpi_i )_\bK$ is cyclic as a $\dot{U}^{0}_\bK$-module). The identification of the $P$-weight $\varpi_i$-part also implies that this map commutes with the action of $\tau_{\beta}$ ($\beta \in Q^{\vee}$), and extends to a surjective map
$$\bX ( \varpi_i )_\bK\rightarrow V ( \varpi_i )_\bK \otimes \bK [z,z^{-1}]$$
of $\dot{U}_\bK$-modules.

This gives a rational map
$$\P ( \bW ( \varpi_i )_\bK ) \dashrightarrow \P ( V ( \varpi_i )_\bK \otimes \bK [z] )$$ 
and its graded completion
$$\P ( \bW ( \varpi_i )_\bK^{\wedge} ) \dashrightarrow \P ( V ( \varpi_i )_\bK \otimes \bK [\![z]\!] ).$$
Taking Corollary \ref{WXcomm-inj} into account, we have an embedding:
\begin{equation}
\bQ_{G,\tJ} ( e )_\bK \hookrightarrow \prod_{i \in \tI \setminus \tJ} \P ( \bW ( \varpi_i )_{\bK}^{\wedge} ).\label{bQ-proj-emb}
\end{equation}

This yields a rational map
$$\psi_e : \bQ_{G,\tJ} ( e )_\bK \dashrightarrow \prod_{i \in \tI \setminus \tJ} \P ( V ( \varpi_i )_\bK \otimes \bK [\![z]\!] )$$
as a composition. This map is $G [\![z]\!]$-equivariant by construction.

For each $w \in W$, we can choose $\beta \in Q^{\vee}_+$ such that $\bQ_{G,\tJ} ( w )_\bK \cong \bQ_{G,\tJ} ( w t_{\beta} )_\bK \subset \bQ _{G,\tJ}( e )_\bK$ by Corollary \ref{transR} and Lemma \ref{surj-rel}. Hence we obtain the map $\psi_w$ for every $w \in W_\af$ as the composition of the above maps (up to grading shifts). We fix the effect of grading shifts from $\psi_w$ to $\psi_{wt_{\beta}}$ by applying $\tau_{w_0 \beta}$ on $\bW ( \varpi_i )_{\bK}^{\wedge}$ for each $i \in \tI \setminus \tJ$ in (\ref{bQ-proj-emb}). Then, the $\tau_{w_0\beta}$-action (on the completions of $\{ \bX ( \varpi_i )_{\bK}\}_{i \in \tI \setminus \tJ}$) transfers $\psi_w$ to $\psi_{wt_{\beta}}$ for each $\beta \in Q^{\vee}$, and it is compatible with the restrictions to $\bQ_G ( \bullet )$'s by Lemma \ref{WX-comm} 1). Hence, we obtain the map $\psi$ of indschemes. This map is $G (\!(z)\!)$-equivariant in our sense. This proves the first assertion.

From now on, we concentrate into the second assertion.

Each $(H \times \Gm)$-fixed point of $\bO ( w )_{\bK}$ ($w \in W_\af$) is contained in the domain of $\psi$, and their images are distinct by inspection. It follows that $\bQ_\bK'$ is contained in the domain of $\psi$, and the restriction of $\psi$ to $\bQ_\bK'$ is injective by examining the stabilizer of the $\bI ( \bK )$-actions at each $(H \times \Gm)$-fixed points.

In view of Theorem \ref{W-unit}, the maps $\{\eta_{\la,\mu}\}_{\la,\mu}$ ($\la,\mu \in P_{\tJ,+}$) induce a commutative diagram of $\dot{U}^{\ge 0}_{\bK}$-modules:
\begin{equation}
\xymatrix{
\bW ( \la )_\bK \ar@{^{(}->}[r] \ar[d]_{\kappa_\la} & \displaystyle{\bigotimes_{i \in \tI \setminus \tJ}} \bW ( \varpi_i )_\bK^{\otimes \left< \al_i^{\vee}, \la \right>}\ar@{->>}[dr] & \\
V ( \la )_\bK \otimes \bK [z] \ar@{^{(}->}[r] & \left( \displaystyle{\bigotimes_{i \in \tI \setminus \tJ}} V ( \varpi_i )_\bK^{\otimes \left< \al_i^{\vee}, \la \right>} \right) \otimes \bK [z] & \ar@{->>}[l] \displaystyle{\bigotimes_{i \in \tI \setminus \tJ}} \left( V ( \varpi_i )_\bK^{\otimes \left< \al_i^{\vee}, \la \right>} \otimes \bK [z]\right).
}\label{WV-quot}
\end{equation}
Here the map $\kappa_\la$ is well-defined by examining the degree $0$-part and the action of $E_0 = F_{\vartheta} \otimes z$ (where $F_{\vartheta}$ is a non-zero vector in the $- \vartheta$-weight space of $\gn^-$). The above commutative diagram also commutes with the translation by $\tau_{\beta}$ ($\beta \in Q^{\vee}$) by construction. Moreover, we have $\kappa_\la ( \bv_{w \la} ) \neq 0$ for each $w \in W_\af$. Therefore, the map $\kappa_\la$ must be surjective whenever its $d$-degree belongs to $\sum_{i \in \tI \setminus \tJ} \Z \left< \al_i^{\vee}, \la \right>$.

For each $i \in \tI \setminus \tJ$, the $\Z [\frac{1}{2}]$-integral structure of $V ( 2 \varpi_i ) \otimes _{\C} \C [z,z^{-1}]$ at the odd degree and even degree must be the same as $\dot{U}_\Z$-modules (as we can connect the extremal weight vectors of even degree part and the odd degree part using the $\mathfrak{sl} ( 2 )$-strings of length $3$). Therefore, for $\la = \varpi_i, \varpi_i + \varpi_j, 2 \varpi_i$ ($i,j \in \tI \setminus \tJ$), the map $\kappa_{\la}$ is surjective.

Consider a (representative of the) image
$$\psi ( x ) = (x_i) \in \prod_{i \in \tI \setminus \tJ} V ( \varpi_i ) _{\bK} \otimes \bK (\!(z)\!)$$
of a $\bK$-valued point $x \in \bQ_{G, \tJ} ( e )_\bK$ under $\psi$. We consider its lifts $\widetilde{x}_i \in \bW ( \varpi_i )_{\bK}^{\wedge}$ and $\widetilde{x}_j \in \bW ( \varpi_j )_{\bK}^{\wedge}$. They must belong to
$$\mathrm{Im} \, \left( \bW ( \varpi_i + \varpi_j )_{\bK}^{\wedge} \longrightarrow ( \bW ( \varpi_i )_{\bK} \otimes \bW ( \varpi_j )_{\bK} )^{\wedge} \right)$$
in order to satisfy the defining relations of $R_{\bK}$. In view of the commutative diagram (\ref{WV-quot}) for $\la = \varpi_i + \varpi_j$, we deduce an equation
\begin{equation}
x_i \otimes_{\bK [\![z]\!]} x_j \in \mathrm{Im} \, \eta_{\varpi_i,\varpi_j} \hskip 5mm i,j \in \tI \setminus \tJ.\label{thinPlucker}
\end{equation}

Since the relation of the ring $\bigoplus_{\la \in P_{\tJ,+}} V ( \la )^*_{\bK}$ is generated by $P$-degrees $2 \varpi_i$ and $\varpi_i + \varpi_j$ for $i,j \in \tI$ (\cite[Theorem 3.5.3]{BK05}), this defines an element of
\begin{equation}
G (\!(z)\!) / [P ( \tJ ), P ( \tJ) ]  (\!(z)\!) \subset \overline{G / [P ( \tJ ), P ( \tJ) ]} ( \bK (\!(z)\!) ) \label{bQ-coset}
\end{equation}
through quadratic relations (see \cite[\S 4]{FM99}), where
$$\overline{G / [P ( \tJ ), P ( \tJ) ]} = \Spec\, \bK [G / [P ( \tJ ), P ( \tJ) ]]$$
is the basic affine space. Therefore, applying some $\tau_{\beta}$ ($\beta \in Q^{\vee}$) if necessary, we conclude that if a $\bK$-valued point $x$ of $\bQ_{G, \tJ} ^{\ra}$ belongs to the domain of $\psi$, then $\psi ( x )$ belongs to the set of $\bK$-valued points of the image given by (\ref{bQ-coset}). Taking Proposition \ref{wbQ-mod} (and its $G(\!(z)\!)$-action) and the equation (\ref{thinPlucker}) into account, we conclude the second assertion.
\end{proof}

\begin{cor}\label{uniq-preimage}
Keep the setting of Proposition {\rm\ref{bQ-mod}}. The map $\psi$ induces a bijection of $\bK$-valued points between the domain and the range.
\end{cor}

\begin{proof}
By abuse of convention, we regard $\bQ'_{\bK}$ as subsets of the both of $\bQ_G^\ra ( \bK )$ and $( \mathrm{Im} \, \psi ) ( \bK )$. Let $x' \in \bQ_G^\ra ( \bK )$ be a preimage of a $( H \times \Gm ) ( \bK )$-fixed point $x \in \bQ'_{\bK}$ through $\psi$. Consider the embedding
$$( \bQ_{G,\tJ} ^{\ra} )_\bK \hookrightarrow \prod_{i \in \tI \setminus \tJ} \P ( \bX ( \varpi_i )_{\bK}^{\wedge} )$$
that prolongs (\ref{bQ-proj-emb}). Let $x' = ( x'_i )_{i \in \tI \setminus \tJ}$ (resp. $x = ( x_i )_{i \in \tI \setminus \tJ}$) be the coordinate of $x'$ (resp. $x$) through the above embedding. We can regard $x'_i \in \bX ( \varpi_i )_{\bK}^{\wedge}$, and it admits a decomposition
$$x'_i = \prod_{\mu} x'_i [\mu],$$
where $\mu \in P$ runs over the $P$-weights of $\bX ( \varpi_i )_{\bK}^{\wedge}$ (or $\bW ( \varpi_i )_{\bK}^{\wedge}$).
Let $d'_i [\mu]$ be the degree of the lowest $d$-degree non-zero contribution of $x'_i [\mu]$ (or $\infty$ if $x'_i [\mu]  = 0$) for each $i \in \tI \setminus \tJ$ and $\mu \in P$. Let $d_i$ be the $d$-degree of $x_i$ for each $i \in \tI \setminus \tJ$ (remember that $x$ is $( H \times \Gm ) ( \bK )$-fixed). Note that we have $x_i [u \varpi_i] = x'_i [u \varpi_i]$ for every $i \in \tI$ and $u \in W$ since
\begin{equation}
\bW ( \varpi_i )_{\bK} \hskip 5mm \text{and} \hskip 5mm V( \varpi_i)_{\bK} \otimes \bK [z]  \hskip 5mm \text{shares the same $P$-weight $u \varpi_i$-parts.}\label{Pshare}
\end{equation}

For each $\gamma \in Q^{\vee}$, we have a collection of the automorphisms of the vector spaces $\{\bX ( \la )_{\bK}\}_{\la \in P_+}$ by shifting the ($d$-)gradings of the weight $\mu$-parts by $\left< \gamma, \mu \right>$ ($\mu \in P$). This defines an automorphism of $R^+_\bK$, and hence defines an automorphism of $\bQ_G^{\ra} ( \bK )$. By using this twist for an appropriate $\gamma \in Q^{\vee}$, we can assume that
\begin{equation}
d'_i [\mu] \gg d'_i[u \varpi_i] = d_i \hskip 5mm \mu \neq u \varpi_i\label{x-degest}
\end{equation}
for each $i \in \tI \setminus \tJ$ and fixed $u \in W$ from the fact that every $P$-weight of $W ( \varpi_i )_{\bK}$ is $\le \varpi_i$. We have $x_i \in \bW_{ww_0} ( \varpi_i )_{\bK}$ for  some $w \in W_\af$ and all $i \in \tI \setminus \tJ$. The inequality (\ref{x-degest}) implies $x_i' \in \bW_{ww_0} ( \varpi_i )_{\bK}^{\wedge}$ for all $i \in \tI \setminus \tJ$ (\cite[Theorem 5.17]{Kas02}). Hence, we have $x' \in \bQ_{G, \tJ} ( w )_{\bK}$ for $w \in W_\af$ such that $x$ is the unique $( H \times \Gm )$-fixed point of $\bO ( \tJ, w )_{\bK}$.

From (\ref{Pshare}), the $( H \times \Gm )$-action shrinks $x'$ to $x$. In view of its proof, the Zariski open set $\bO ( \tJ, w )_{\bK}$ in the Lemma \ref{bQ-dense} is obtained by localizing along the vectors $\{ \bv_{ww_0\la}^{\vee} \}_{\la \in P_{\tJ,+}}$, that pairs nontrivially with $x'$. Therefore, we have necessarily $x' \in \bO ( \tJ, w )_{\bK} ( \bK )$. As $\bO ( \tJ, w )_{\bK} ( \bK ) \subset \bQ'_{\bK}$ by Proposition \ref{wbQ-mod} and its proof, we deduce $x = x'$. Since $\bQ'_{\bK}$ is stable under the above action of $Q^{\vee}$, as well as $G (\!(z)\!)$, we conclude that $x = x'$ holds for a preimage of $x'$ of every $x \in \bQ'_{\bK}$ as required.
\end{proof}

\begin{thm}\label{bQ-int}
Assume that $\mathsf{char} \, \bK \neq 2$. For each $\tJ \subset \tI$, we have a closed immersion of indschemes
$$( \bQ_{G,\tJ} ^{\ra} )_{\bK} \longrightarrow \prod_{i \in \tI \setminus \tJ} \P ( V ( \varpi_i )_\bK \otimes \bK (\!(z)\!) ).$$
In particular, the set of $\bK$-valued points of the indscheme $( \bQ_{G,\tJ} ^{\ra} )_\bK$ is in bijection with
$$G (\!(z)\!) / \left( H ( \bK ) \cdot [P ( \tJ ), P ( \tJ) ] (\!(z)\!) \right).$$
\end{thm}

\begin{proof}
By construction, the locus $E \subset \bQ_{G,\tJ}^{\ra}$ on which $\psi$ (borrowed from Proposition \ref{bQ-mod}) is not defined is an ind-subscheme. The map $\psi$ is $G (\!(z)\!)$-equivariant. It follows that $E$ admits a $G (\!(z)\!)$-action as indschemes. The map (borrowed from the proof of Proposition \ref{bQ-mod})
$$\bQ_{G,\tJ} ( e )_{\bK} \longrightarrow \P ( \bW ( \varpi_i )_{\bK}^{\wedge})$$
sends an irreducible component of $( E \cap \bQ_{G,\tJ} ( e )_{\bK} )$ onto a closed subscheme of $\P ( \ker \, ( \bW ( \varpi_i )_{\bK}^{\wedge} \to V ( \varpi_i )_{\bK} \otimes \bK [\![z]\!])$ for some $i \in \tI \setminus \tJ$. In particular, $( E \cap \bQ_{G,\tJ} ( e )_{\bK} ) \subset \bQ_{G,\tJ} ( e )_{\bK}$ is a closed subscheme. Taking into account the $\{ \tau_{\beta} \}_{\beta}$-actions (Lemma \ref{WX-comm}), we have $E \neq \emptyset$ if and only if $( E \cap \bQ_{G,\tJ} ( e )_{\bK} ) \neq \emptyset$, and $\psi$ is a closed immersion if and only if $\psi_e$ is a closed immersion (cf. Lemma \ref{surj-rel}). 

The one-parameter subgroup $a = ( \xi, 1 ) : \Gm \to ( H \times \Gm )$ for $\xi \in Q^{\vee}_<$ attracts every point of $\P ( \bW ( \varpi_i )_{\bK})$ into its $( H \times \Gm )$-fixed points (by setting $t \to 0$). In particular, $E$ has a $( H \times \Gm )$-fixed point that is not realized by $\bQ_\bK'$. Since the $( H \times \Gm )$-fixed point of bounded $d$-degree is captured by the corresponding degree terms of $R_{\bK}$, it follows that the indscheme $E$ intersects with $\sQ'_{\tJ} ( v, e )$ for some $v \in W_\af$. The intersection of $\sQ'_{\tJ} ( v, e )_{\bK}$ and $\bQ_\bK'$ (as the set of $\bK$-valued points) defines a closed subset $Y$ of (a product of) finite-dimensional projective space by Proposition \ref{bQ-mod} and Corollary \ref{uniq-preimage}. In particular, $Y$ acquires the structure of a proper scheme through $\psi$. From this view-point, Corollary \ref{uniq-preimage} provides a $\psi$-section of $Y$ that defines a bijection of $\bK$-valued points with a Zariski open subset $\sQ'_{\tJ} ( v, e )_{\bK}\setminus E$ of $\sQ'_{\tJ} ( v, e )_{\bK}$. Since $\bQ_\bK'$ admits a homogeneous $G (\!(z)\!)$-action, we can think of $\psi$ as an everywhere defined section of a vector bundle over the image (whose fiber is a product of $\ker \, ( \bW ( \varpi_i )_{\bK}^{\wedge} \to V ( \varpi_i )_{\bK} \otimes \bK [\![z]\!])$'s). It implies that $\sQ'_{\tJ} ( v, e )_{\bK}\setminus E$ can be seen as an everywhere defined section of a vector bundle over the image (whose fiber is a product of finite-dimensional subspaces of $\ker \, ( \bW ( \varpi_i )_{\bK} \to V ( \varpi_i )_{\bK} \otimes \bK [z])$'s). Since $\sQ'_{\tJ} ( v, e )_{\bK}$ is a finite type scheme defined over an algebraically closed field $\bK$, it has a dense subset formed by its $\bK$-valued points (\cite[Corollarie 10.4.8]{EGAIV-3}). Thus, the section $\psi$ can be seen as that of schemes. Therefore, $\sQ'_{\tJ} ( v, e )_{\bK} \setminus E$ is proper by itself. In conclusion, $\sQ'_{\tJ} ( v, e )_{\bK} \setminus E$ can be seen as a connected component of $\sQ'_{\tJ} ( v, e )_{\bK}$ (as being open and closed subset) which does not intersect with $E$. It follows that if $E \neq \emptyset$ as a scheme, then we find that the scheme $\sQ'_{\tJ} ( v, e )$ must have at least two connected components. Therefore, the projective coordinate ring $R^v_e ( \tJ )_{\bK}$ of $\sQ'_{\tJ} ( v, e )$ must be non-integral. The same is true if we replace $v$ with a smaller element with respect to $<_\si$.

Hence, we can find $f,g \in R^v_e ( \tJ )_{\bK} \setminus \{ 0 \}$ such that $fg = 0$ for $v \ll_\si e$. Since $E$ and its complement are $\Gm$-stable, we can assume that $f$ and $g$ are $\Gm$-eigenfunctions. Since $E$ defines a connected component of $\sQ'_{\tJ} ( v, e )$ for every $v \ll_\si e$, we can fix the degrees of $f$ and $g$ for every $v \ll_\si e$. For a fixed degree, the $d$-graded component of $R^v_e ( \tJ )_{\bK}$ and $R_e ( \tJ )_{\bK}$ are in common for $v \ll_\si e$. Since $R^v_e ( \tJ )_{\bK}$ is a quotient ring of $R_e ( \tJ )_{\bK}$, we can find $\Gm$-eigenfunctions $f,g \in R_e ( \tJ )_{\bK} \setminus \{ 0 \}$ such that $fg = 0$. This implies that the ring $R_\bK$ is also non-integral. This contradicts with Corollary \ref{Re-int}, and hence we deduce $E = \emptyset$ as an ind-scheme. Therefore, we conclude that $\psi$ is in fact a genuine morphism (instead of a rational map) of ind-schemes.

Next, we prove that $\psi$ (or rather $\psi_e$) defines a closed immersion. By Lemma \ref{bQ-dense} and the fact that $\psi$ is a morphism, we deduce that $\psi$ induces an isomorphism between the function fields of $\bQ_{G, \tJ} ( e )_{\bK}$ and its image under $\psi$. An irreducible component $Z$ of a reduced scheme $\sQ_{\tJ}' ( v, e )_{\bK}$ shares a Zariski dense subset with a unique orbit $\bO ( \tJ, w )_\bK$ ($w \in W_\af$). Hence, the variety $( Z \cap \bO ( \tJ, w )_\bK )$ and its image under $\psi$ are presented by the quotients of a common polynomial ring $\bK [\bO ( \tJ, w )_\bK]$ of infinite variables. In view of the $d$-degree bounds from the above and below offered by $\bW _{vw_0}^- ( \varpi_i )_{\bK}$ and $\bW _{w_0} ( \varpi_i )_{\bK}$ for each $i \in \tI \setminus \tJ$, these two quotients of $\bK [\bO ( \tJ, w )_\bK]$ factor through a common polynomial ring of finite variables. Thus, we deduce that $\psi$ also induces an isomorphism between the function fields of $Z$ and $\psi ( Z )$. Once we fix an integer $d_0$, two rings $R^v_e ( \tJ )_{\bK}$ and $R_e ( \tJ )_{\bK}$ share the same ($d$-)grading $d_0$-part for $v \ll_\si e$. Hence, it suffices to prove that $\psi$ restricts to a closed immersion of $\sQ_{\tJ}' ( v, e )_{\bK}$ for $v \ll_\si e$. Here $R^v_e ( \tJ )_{\bK}$ is weakly normal by Lemma \ref{Q-split}. Let $R'$ be the multi-homogeneous coordinate ring of the (reduced induced structure of the) closed subscheme of $\prod_{i \in \tI \setminus \tJ} \P ( V ( \varpi_i )_\bK \otimes \bK (\!(z)\!) )$ defined by $\psi ( \sQ_{\tJ}' ( v, e )_{\bK} )$. Then, if we consider via each maximal integral quotient, we deduce that the weak normalization of $R'$ is precisely $R^v_e ( \tJ )_{\bK}$ (up to irrelevant locus). This implies that $\psi ( \sQ_{\tJ}' ( v, e )_{\bK} )$ defines a closed subscheme for each $v \in W_\af$ as required.

Since $\psi$ is a closed immersion, the rest of assertion follows from Proposition \ref{bQ-mod}. These complete the proof.
\end{proof}

\begin{cor}[of the proof of Theorem \ref{bQ-int}]\label{ZdensebQ}
Let $w \in W_\af$ and $\tJ \subset \tI$. Every two rational functions on $\bQ_{G,\tJ} ( w )_{\bK}$ are distinguished by a pair of $\bK$-valued points of $\sQ'_{\tJ} ( v, w )_{\bK}$ for some $v \in W_\af$. In particular, the union $\bigcup _{v \in W_\af} \sQ'_{\tJ} ( v, w )_{\bK}$ is Zariski dense in $\bQ_{G,\tJ} ( w )_{\bK}$. \hfill $\Box$
\end{cor}

\begin{cor}\label{pos-fp}
The conclusions of Theorem {\rm\ref{si-Bruhat}} and Theorem {\rm\ref{QG-proj}} holds when we replace $\C$ with an algebraically closed field $\bK$ of characteristic $\neq 2$.
\end{cor}

\begin{proof}
In view of Theorem \ref{bQ-int}, all the results in Theorem \ref{si-Bruhat} and Theorem \ref{QG-proj} are consequences of the corresponding set-theoretic considerations.
\end{proof}

\begin{cor}\label{Q'nonempty}
We have $\sQ' ( v, w )_{\bK} \neq \emptyset$ if and only if $v \le_\si w$.
\end{cor}

\begin{proof}
By Lemma \ref{non-empty}, it remains to prove that $\sQ' ( v, w )_{\bK} = \emptyset$ if $v \nless_\si w$. By Theorem \ref{bQ-int} and the definition of $\sQ' ( v, w )_{\bK}$, the $(H \times \Gm)$-fixed point of $\bO ( u )_{\bK}$ is contained in $\sQ' ( v, w )_{\bK}$ ($v,w\in W_\af$) only if $v \le_\si u \le _\si w$. Since $\le_\si$ is a partial order, the condition $v \nless_\si w$ implies that $\sQ' ( v, w )_{\bK}$ has no $(H \times \Gm)$-fixed point. Since $\sQ' ( v, w )_{\bK}$ is projective by Lemma \ref{Q'finite}, it carries a $(H \times \Gm)$-fixed point if it is nonempty. Thus, we conclude $\sQ' ( v, w )_{\bK} = \emptyset$ as required.
\end{proof}

\subsection{Coarse representability of the scheme $\bQ_G^{\ra}$}\label{cmp}

Material in this subsection is rather special throughout this paper, and is irrelevant to the arguments in the later part, such as the normality of quasi-map spaces.

In this subsection, we assume that $\mathsf{char} \, \bK \neq 2$, and we also drop subscripts $\bK$ from $( \bQ_G^{\mathrm{rat}} )_{\bK}$ and its subschemes in order to simplify notation.

Let $\Aff _\bK$ be the category of affine schemes over $\bK$. We identify $\Aff _\bK^{op}$ with the category of commutative rings over $\bK$. Let $Zar _\bK$ denote a big Zariski site over $\bK$ \cite[\href{https://stacks.math.columbia.edu/tag/020N}{Section 020N}]{stacks-project}. For $X \in Zar_\bK$, the assignment
$$Zar_{\bK}^{op} \ni U \mapsto \mathrm{Hom} _{Zar_\bK} ( U, X ) \in \mathrm{Sets}$$
defines a sheaf $h_X$ on $Zar_{\bK}$ \cite[\href{https://stacks.math.columbia.edu/tag/00WR}{Definition 00WR}]{stacks-project}.

For the definition on the coarse moduli functors, we refer to \cite[Definition 1.10]{Vie95}. However, we employ some modified definition given in the below:

\begin{defn}[Strict indscheme]
Let $\mathfrak X = \bigcup_{n \ge 0} X_n$ be an increasing union of schemes in $Zar_{\bK}$. We call $( \mathfrak X, \{ X_n \}_n )$ (or simply refer as $\mathfrak X$) a strict indscheme if each inclusion $X_k \subset X_{k+1}$ ($k \ge 0$) is a closed immersion.
\end{defn}

\begin{defn}[Filtered sheaf on $Zar_{\bK}$]
A filtered (pre)sheaf $( \mathcal F, \{ \mathcal F_n \}_{n\ge 0} )$ on $Zar_{\bK}$ is a family of (pre)sheaves such that $\mathcal F_k \subset \mathcal F_{k+1}$ for each $k \in \Z_{\ge 0}$ and $\mathcal F = \bigcup_{n} \mathcal F_n$. Let $( \mathcal F, \{ \mathcal F_n \}_{n\ge 0} )$ and $( \mathcal G, \{ \mathcal G_n \}_{n\ge 0} )$ be filtered (pre)sheaves on $Zar_{\bK}$. A morphism $f : \mathcal F \rightarrow \mathcal G$ of (pre)sheaves is said to be continuous if for each $n \in \Z_{\ge 0}$, there is some $m \in \Z_{\ge 0}$ such that
$$f ( \mathcal F_n ) \subset \mathcal G_m \hskip 5mm \text{ and } \hskip 5mm \mathcal G_n \cap \mathrm{Im}\,f \subset f ( \mathcal F_m ).$$
Let $( \mathfrak F, \{F_n\}_n )$ be a strict indscheme. Then, we call $h_{\mathfrak F} := ( \bigcup_{n} h_{F_n}, \{h_{F_n}\}_{n\ge 0} )$ the filtered sheaf associated to $\mathfrak F$.
\end{defn}

\begin{defn}[Coarse ind-representability]\label{ind-rep}
Let $\mathcal X$ be a filtered (pre)sheaf on $Zar_{\bK}$. Let $\mathfrak X$ be a strict indscheme over $\bK$. We say that $\mathcal X$ is coarsely ind-representable by $\mathfrak X$ if the following conditions hold:
\begin{itemize}
\item We have a continuous morphism $u : \mathcal X \rightarrow h_{\mathfrak X}$ of filtered (pre)sheaves;
\item We have $\mathcal X ( \Bbbk ) = h_{\mathfrak X} ( \Bbbk )$ for an overfield $\Bbbk \supset \bK$;
\item Let $\mathfrak Y$ be a strict indscheme and we have a continuous morphism $f : \mathcal X \rightarrow h_{\mathfrak Y}$, then it factors as:
$$
\xymatrix{
\mathcal X \ar[r]^{u} \ar[dr]_{f} & h_{\mathfrak X}\ar@{..>}[d]^g\\
& h_{\mathfrak Y}\ar@{}[lu]_{\circlearrowright}
},$$
where $g$ is a morphism of presheaves. It is automatic that $g$ is continuous, and hence is induced by a morphism of indschemes.
\end{itemize}
\end{defn}

We consider the assignment $\mathcal Q$ on $\Aff^{op}_\bK$ defined as:
$$\Aff^{op}_\bK \ni R \mapsto \mathcal Q ( R ) := G ( R (\!(z)\!)) / ( H ( R ) N ( R (\!(z)\!) ) ) \in \mathrm{Sets}.$$
For each $n \in \Z_{\ge 0}$, we consider an assignment
$$\Aff^{op}_\bK \ni R \mapsto \mathcal Q_n ( R ) := \{ g \mod H ( R ) N ( R (\!(z)\!) ) \in \mathcal Q ( R ) \mid (\star) \} \in \mathrm{Sets},$$
where
\begin{itemize}
\item[$(\star)$] $g \bv_{\varpi_i}$ has at worst pole of order $n$ on $V ( \varpi_i )_\Z \otimes_{\Z} R (\!(z)\!)$ for each $i \in \tI$.
\end{itemize}

The assignments $( \mathcal Q, \{ \mathcal Q_n \}_n)$ define a filtered presheaf on $Zar_\bK$ that we denote by $\mathcal Q$.

\begin{lem}\label{Qfilt}
The indscheme $\bQ_G^{\ra}$ defines a filtered sheaf on $Zar_\bK$ given by a strict indscheme structure.
\end{lem}

\begin{proof}
A scheme over $\bK$ defines a sheaf over $Zar_\bK$, and so is its increasing union. In view of Proposition \ref{bQ-mod}, the pole order $n$ condition amounts to choose the $\bI$-orbits $\bO (u t_{\beta} )$ ($u \in W, \beta \in Q^{\vee}$) such that $\left< \beta, \varpi_i \right> \ge - n$ (for every $i \in \tI$), that makes the smaller one to be a closed subscheme of the larger one (cf. Lemma \ref{surj-rel}).
\end{proof}

\begin{prop}\label{bQcoarse}
The scheme $\bQ_G^{\ra}$ coarsely ind-represents the presheaf $\mathcal Q$.
\end{prop}

\begin{proof}
We first construct an injective continuous morphism $\mathcal Q \rightarrow h_{\bQ_G^{\ra}}$ as filtered presheaves on $Zar_\bK$ such that $\mathcal Q ( \Bbbk ) = h_{\bQ_G^{\ra}} ( \Bbbk )$ for an overfield $\Bbbk \supset \bK$.

For $R \in \Aff_\bK^{op}$, the set $\mathcal Q ( R )$ is represented by a class of $g \in G ( R (\!(z)\!))$ modulo the right action of $H ( R ) N ( R (\!(z)\!) )$. It defines a point of $\bQ_G^{\ra} ( R )$ by applying $g$ on $\{[\bv_{\varpi_i}]\} _{i \in \tI} \in \prod_{i \in \tI} \P _{R} ( V ( \varpi_i )_{\Z} \otimes_\Z R (\!(z)\!) )$. Since the $G(R(\!(z)\!))$-stabilizer of $\bv_{\varpi_i}$ is precisely $H ( R ) N ( R (\!(z)\!) )$, we conclude an inclusion $\mathcal Q ( R ) \subset h_{\bQ_G^{\ra}} ( R )$. By examining the construction, we deduce that this defines an injective continuous morphism of filtered presheaves.

By the Bruhat decomposition, we have
$$G ( \Bbbk (\!(z)\!) ) / N ( \Bbbk (\!(z)\!) ) = ( G / N ) ( \Bbbk (\!(z)\!) )$$
for an overfield $\Bbbk \supset \bK$ (and hence $\Bbbk (\!(z)\!)$ is a field). In view of Theorem \ref{bQ-int}, we conclude that $\mathcal Q \subset h_{\bQ_G^{\ra}}$ is an inclusion of filtered presheaves with $\mathcal Q ( \Bbbk ) = h_{\bQ_G^{\ra}} ( \Bbbk )$ when $\Bbbk \supset \bK$ is an overfield.

We verify the versality property. Suppose that we have a strict indscheme $( \mathfrak X, \{ X_n \}_n )$ and we have a continuous morphism $\mathcal Q \rightarrow h_{\mathfrak X}$. By Lemma \ref{bQ-dense}, we deduce that each $\bI$-orbit of $\bQ_G^{\ra}$ defines a subfunctor of $\mathcal Q_n$ for some $n \ge 0$. The zero-th ind-piece $( \bQ_G^{\ra} )_0$ in Lemma \ref{Qfilt} is $\bQ_G ( e )$.

For each $t_{\beta} \le_{\si} e$ ($\beta \in Q^{\vee}_+$), we find a reduced expression $t_{\beta}^{-1} = s_{i_1} \cdots s_{i_{\ell}}$ (that we record as $\bi := (i_1,\ldots, i_{\ell})$) and form a scheme
$$Z ( \bi )^{\circ} := \bI (i_1)\times^{\bI} \bI (i_2) \times^{\bI} \cdots \times^{\bI} \bO ( t_{\beta} )$$
and the map
$$Z ( \bi )^{\circ} = \bI (i_1)\times^{\bI} \bI (i_2) \times^{\bI} \cdots \times^{\bI} \bO ( t_{\beta} ) \rightarrow \bQ_G ( e )$$
(see e.g. Kumar \cite[Chapter V\!I\!I\!I]{Kum02}; cf. \cite[\S 6]{Kat18}). In view of Lemma \ref{bQ-dense}, the image of this map contains an open neighbourhood of $\bO ( t_{\beta} )$. In view of the $G(\!(z)\!)$-action (or the various $\SL ( 2, i )$-actions for $i \in \tI_\af$) on $\bQ_G^{\mathrm{rat}}$ and $\mathcal Q$, we have a morphism
$$f_{\bi} : h_{Z ( \bi )^{\circ}} \longrightarrow \mathcal Q_0.$$
By varying $\bi$ (and consequently varying $t_{\beta} \le _{\si} e$), we deduce that the union of the image of the morphisms $\{ f_{\bi} \}_{\bi}$ exhausts $\mathcal Q_0 ( \Bbbk )$ for an overfield $\Bbbk \supset \bK$. From the Yoneda embedding, we derive a map
$$Z ( \bi )^{\circ} \longrightarrow X_n$$
of schemes for some fixed $n \in \Z$. This map factors through a scheme $Z$ that glues (among $\bi$'s) all the closed points that maps to the same points in $\mathcal Q_0$. Such a scheme is integral as $Z ( \bi )^{\circ}$'s are so and the gluing identifies the Zariski open dense subset $\bO ( e )$ for distinct $\bi$'s. In addition, we have a birational map $\pi : Z \rightarrow \bQ_G ( e )$, and hence we have
$$Z ( \Bbbk ) = \bQ_G ( e ) ( \Bbbk ) = \mathcal Q _0 ( \Bbbk ) \hskip 3mm \text{for an overfield} \hskip 3mm\Bbbk \supset \bK.$$

We prove that $Z = \bQ_G ( e )$ by induction. For each $m \in \Z_{\ge 0}$, let $\bQ_G ( e ) _{<m}$ (resp. $\bQ_G ( e ) _{\le m}$) be the union of $\bI$-orbits in $\bQ_G ( e )$ of the shape $\bO ( v )$ for $\ell^{\si} ( v ) < m$ (resp. $\le m$).

Assume that the map $\pi$ is an isomorphism when restricted to $\bQ_G ( e ) _{<m}$, and we prove the same is true when restricted to $\bQ_G ( e ) _{\le m}$. The $m = 1$ case is afforded by $\bO ( e ) \subset Z$ already used in the construction of the above.

We have a partial compactification $Z ( \bi )$ of $Z ( \bi )^{\circ}$ with a map $f^+_{\bi}$ given as:
$$Z ( \bi ) := \bI (i_1)\times^{\bI} \bI (i_2) \times^{\bI} \cdots \times^{\bI} \bQ_G ( t_{\beta} ) \stackrel{f^+_{\bi}}{\longrightarrow} \bQ_G ( e ).$$
Note that we have a surjective morphism induced by $\bO ( t_{\beta} ) \to \Spec \, \bK$
$$\eta_{\bi} : Z ( \bi ) \longrightarrow \bI (i_1)\times^{\bI} \bI (i_2) \times^{\bI} \cdots \times^{\bI} \Spec \, \bK,$$
where we denote the image (the RHS term) by $Z' ( \bi )$. Since $Z' ( \bi )$ is a finite successive $\P^1$-fibration, it is proper. The map $f_{\bi}^+$ is proper as the product map
$$( \eta_{\bi} \times f_{\bi}^+ ) : Z ( \bi ) \hookrightarrow Z' ( \bi ) \times \bI ( i_1 ) \cdots \bI ( i_{\ell} ) \bQ_G ( t_{\beta} ) = Z' ( \bi ) \times \bQ_G ( e )$$
is a closed immersion. In view of the isomorphism $\bQ_G ( t_{\beta} ) \cong \bQ_G ( e )$, we transplant $\bQ_G ( e )_{< m}$ to $\bQ_G ( t_{\beta} )_{< m}$.

\begin{claim}\label{fpt}
For each closed point $x \in \bQ_G ( e )_{\le m}$, the scheme
$$( f_{\bi}^+ )^{-1} ( x ) \setminus \left( ( f_{\bi}^+ )^{-1} ( x ) \cap ( \bI (i_1)\times^{\bI} \bI (i_2) \times^{\bI} \cdots \times^{\bI} \bQ_G ( t_{\beta} )_{<m} ) \right) \subset ( f_{\bi}^+ )^{-1} ( x )$$
is a closed subscheme that is zero-dimensional. In other words, it is a finite union of points $($that is potentially an empty set$)$.
\end{claim}

\begin{proof}
For each sequence $(j_1,\ldots,j_s) \in \tI_\af^s$ ($s \in \Z_{> 0}$) and $v\le_\si e$ such that $\ell ^\si ( v ) = \ell$, the image of the map
$$f : \bI ( j_1 ) \times^{\bI} \bI ( j_2 ) \times^{\bI} \cdots \times^{\bI} \bI ( j_s ) \times^{\bI} \bQ_G ( v ) \longrightarrow \bQ_G^{\mathrm{rat}}$$
induced by the multiplication is a union of $\bI$-orbits $\bO ( v' )$ with $\ell ^\si ( v' ) \ge \ell - s$ (as we have $\ell^{\si} ( s_i w ) \in \{\ell^{\si} ( w ) \pm 1 \}$ for each $i \in \tI_\af$ and $w \in W_\af$ by \cite[Lecture 13, Proposition $\ell_s$]{Pet97}). In addition, if the image of the map $f$ contains $\bO ( v' )$ for $\ell ^\si ( v' ) = \ell - s$, then the map $f$ is an isomorphism along $\bO ( v' )$ (as the isomorphism between open subsets). By collecting these for $\bI$-orbits in the closed subset $\bQ_G ( t_{\beta} ) \setminus \bQ_G ( t_{\beta} )_{<m}$ of $\bQ_G ( t_{\beta} )$ in the construction of the (proper) map $f_{\bi}^+$, we conclude the result. 
\end{proof}

We return to the proof of Proposition \ref{bQcoarse}. By Claim \ref{fpt}, we deduce that
$$\overline{( f_{\bi}^+ )^{-1} ( x ) \cap \bI (i_1)\times^{\bI} \bI (i_2) \times^{\bI}  \cdots \times^{\bI} \bQ_G ( t_{\beta} )_{<m}} \subset ( f_{\bi}^+ )^{-1} ( x )$$
is a union of connected components of $( f_{\bi}^+ )^{-1} ( x )$ for each closed point $x \in \bQ_G ( e )_{\le m}$.

Requiring that regular functions on $( f_{\bi}^+ )^{-1} ( \bQ_G ( e )_{\le m} )$ to be constant along all the fibers yield sections in $(f^+_\bi)_* \cO_{Z ( \bi )}$. From this (for arbitrary $t_{\beta} \le_\si e$ and $\bi$) and the induction hypothesis, we conclude that $\pi^{-1} ( \bQ_G ( e )_{\le m} ) \subset Z$ is a union of proper schemes over $\bQ_G ( e )_{\le m}$ that contains a Zariski open subset $\bQ_G ( e )_{< m}$. In view of Corollary \ref{FQ-wn} (when $\mathsf{char} \, \bK > 0$) or Proposition \ref{R-normal} (when $\mathsf{char} \, \bK = 0$; cf. \cite[Theorem A]{KNS17}), we deduce that
$$( Z \supset ) \hskip 3mm \pi^{-1} ( \bQ_G ( e )_{\le m} ) \to \bQ_G ( e )_{\le m}$$ defines an isomorphism as schemes (as $\pi^{-1} ( \bQ_G ( e )_{\le m} ) \to \bQ_G ( e )_{\le m}$ is finite bijective, and birational, cf. \cite[\href{https://stacks.math.columbia.edu/tag/02LQ}{Section 02LQ}]{stacks-project}). Therefore, induction on $m$ proceeds and we conclude $Z \cong \bQ_G ( e )$ as schemes. Thus, we obtain a morphism $\bQ_G ( e ) \to X_n$ of schemes. 

By rearranging $\bQ_G( e )$ by the right $Q^{\vee}$-translations, we deduce a morphism $\bQ_G ^{\ra} \to \mathfrak X$ as indschemes. This yields a continuous morphism $h_{\bQ_G^{\ra}} \to h_{\mathfrak X}$. Therefore, $h_{\bQ_G^{\ra}}$ is an initial object in the category of sheaves on $Zar_\bK$ ind-representable by strict indschemes that admits a continuous morphism from $\mathcal Q$ as required.
\end{proof}

\begin{cor}
For each $\tJ \subset \tI$, the scheme $\bQ_{G,\tJ}^{\ra}$ coarsely ind-represents the filtered presheaf $\mathcal Q_{\tJ}$ defined by
$$\Aff^{op}_\bK \ni R \mapsto \mathcal Q_{\tJ} ( R ) := G ( R (\!(z)\!)) / \left( H ( R ) \cdot [P_{\tJ}, P_{\tJ}] ( R (\!(z)\!) ) \right) \in \mathrm{Sets}.$$
\end{cor}

\begin{proof}
By construction, we have a continuous morphism of presheaves $\mathcal Q \rightarrow \mathcal Q_{\tJ}$ (by transplanting subsheaves $\mathcal Q_n$ to $\mathcal Q_{\tJ}$ via this map). Thus, the coarse ind-representability of $\bQ_{G}^{\ra}$ implies that the maximal indscheme $X$ obtained by gluing points of $\bQ_{G}^{\ra}$ admits a continuous morphism $\mathcal Q_{\tJ} \rightarrow h_X$ coarsely ind-represents the filtered presheaf $\mathcal Q_{\tJ}$. Every two rational functions on $\bQ_{G}(w)$ ($w \in W_\af$) are distinguished by some pair of $\bK$-valued points (Corollary \ref{ZdensebQ}). Since we have $\mathcal Q_{\tJ} ( \bK ) = \bQ_{G, \tJ}^{\ra} ( \bK )$, we conclude that $X = \bQ_{G, \tJ}^{\ra}$.
\end{proof}

\subsection{The properties of the schemes $\sQ'_{\tJ} ( v, w )$}\label{Q'Jvw}

In the rest of this section, we assume $\mathsf{char} \, \bK \neq 2$.

\begin{lem}\label{Q-mod}
Let $\tJ \subset \tI$. For each $\beta \in Q^{\vee}_+$, the set of $\bK$-valued points of $\sQ'_{\tJ} ( \beta, e )_{\bK}$ is in bijection with the collection of elements $\{u_\la ( z )\}_{\la \in P_{\tJ,+}}$ such that
\begin{itemize}
\item We have $u_{\la} ( z ) \in V ( \la )_{\bK} \otimes \bigoplus_{j = 0}^{- \left< w_0 \beta, \la \right>} \bK z^j \setminus \{0\}$;
\item For each $\la, \mu \in P_+$, we have $\eta_{\la, \mu} ( u_{\la} ( z ) \otimes u_{\mu} ( z ) ) = u _{\la + \mu} ( z )$;
\end{itemize}
modulo the action of $H(\bK)$.
\end{lem}

\begin{proof}
We have
$$\bQ_{G,\tJ} ( e )_{\bK} = ( \bQ_{G,\tJ}^{\ra} )_{\bK} \cap \prod_{i \in \tI \setminus \tJ} \P ( V ( \varpi_i ) _\bK \otimes \bK [\![z]\!]) \subset \prod_{i \in \tI \setminus \tJ} \P ( V ( \varpi_i ) _\bK \otimes \bK (\!(z)\!))$$
by Theorem \ref{bQ-int}. By the symmetry of the construction of $\sQ'_\tJ ( \beta, e )$ in terms of $\theta$ (cf. Remark \ref{enh-emb}), we conclude that
$$\sQ'_\tJ ( \beta, e )_\bK = ( \bQ_{G,\tJ}^{\ra} )_{\bK} \cap \prod_{i \in \tI \setminus \tJ} \P ( V ( \varpi_i ) _\bK \otimes \bK [\![z]\!]) \cap \prod_{i \in \tI \setminus \tJ} \P ( V ( \varpi_i ) _\bK \otimes \bK [\![z^{-1}]\!] z^{- \left< \beta, w_0 \varpi_i \right>} )$$
inside $\prod_{i \in \tI \setminus \tJ} \P ( V ( \varpi_i ) _\bK \otimes \bK [\![z, z^{-1}]\!] )$, that is our degree bound. In view of this, it suffices to remember that the second condition is the same as the Pl\"ucker relation that defines $G (\!(z)\!) / H ( \bK ) \cdot [P( \tJ ), P( \tJ )] (\!(z)\!)$ in the last two paragraphs of the proof of Proposition \ref{bQ-mod}.
\end{proof}

Let $w,v \in W_\af$ and $\tJ \subset \tI$. For each $\la \in P_{\tJ, +}$, we have a line bundle $\cO_{\sQ'_{\tJ} ( v, w )} ( \la )$ on $\sQ'_{\tJ} ( v, w )$ such that we have a map
$$R ^v _w ( \tJ, \la ) \rightarrow \Gamma ( \sQ'_{\tJ} ( v, w ), \cO_{\sQ'_{\tJ} ( v, w )} ( \la ) )$$
that commutes with the multiplications. This yields $\cO_{\sQ'_{\tJ} ( v, w )} ( \la )$ for each $\la \in P_{\tJ}$ by tensor products.

\begin{lem}\label{ampleness}
For each $w, v \in W_\af$, the line bundle $\cO_{\sQ'_{\tJ} ( v, w )} ( \la )$ is very ample if $\left< \al_i^{\vee}, \la \right> > 0$ for every $i \in \tI \setminus \tJ$.
\end{lem}

\begin{proof}
We can assume that $\sQ'_{\tJ} ( v, w ) \neq \emptyset$ without the loss of generality.

By Lemma \ref{Qsurj-mult}, we have
$$\sQ'_{\tJ} ( v, w ) \hookrightarrow \prod_{i \in \tI \setminus \tJ} \P_\Z ( R ^v _w ( \varpi_i )^{\vee} ).$$
From this and again by Lemma \ref{Qsurj-mult}, we deduce that the following diagram is commutative:
$$\xymatrix{
\prod_{i \in \tI \setminus \tJ} \P_\Z ( R ^v _w ( \tJ, \varpi_i )^{\vee} ) \ar@{^{(}->}[r] ^{\xi}&  \P_\Z ( \bigotimes_{i \in \tI \setminus \tJ} R ^v _w ( \tJ, \varpi_i )^{\vee} ) & \ar[l]_{\hskip 8mm \kappa} \P_\Z ( R ^v _w ( \tJ, \rho_\tJ )^{\vee} )\\
& \sQ'_{\tJ} ( v, w ) \ar@{_{(}->}[ul] \ar[ur]& 
},$$
where the $\xi$ is the Veronese embedding and $\kappa$ is induced from the multiplication map. The ring $R ^v _w ( \tJ )$ is reduced by Corollary \ref{Q'red}. Since $\left< \al_i^{\vee}, \rho_\tJ \right> > 0$ for every $i \in \tI \setminus \tJ$ and the multiplication maps are surjective, a non-zero element of $R ^v _w ( \tJ, \rho_\tJ )$ is not supported in the irrelevant locus. It follows that $R ^v _w ( \tJ, \rho_\tJ ) \subset H^0 ( \sQ'_\tJ ( v, w ), \cO_{\sQ'_\tJ ( v, w )} ( \rho_\tJ ) )$ and it is spanned by a product of linear functions that separates closed points of $\sQ'_\tJ ( v, w )$. In particular, $\cO_{\sQ'_\tJ ( v, w )} ( \rho_\tJ )$ is very ample. Since we have embeddings
$$R ^v _w ( \tJ, \rho_\tJ ) \subset R ^v _w ( \tJ, \la ) \subset H^0 ( \sQ'_\tJ ( v, w ), \cO_{\sQ'_\tJ ( v, w )} ( \la ) )$$
obtained through multiplications corresponding to the duals of extremal weight vectors (so that the image is base-point-free along a $( H \times \Gm)$-stable Zariski open neighbourhood of each $( H \times \Gm)$-fixed point of a projective variety), we conclude that $\cO_{\sQ'_\tJ ( v, w )} ( \la )$ is also very ample as required.
\end{proof}

\begin{thm}\label{QQ-isom}
For each $w = u t_{\beta'} \in W_\af$ $(u \in W, \beta' \in Q^{\vee})$ and $\beta \in Q^{\vee}_+$, we have an isomorphism $\sQ' ( \beta, w )_\C \cong \sQ ( \beta - \beta', u )$ as varieties. Moreover, $\sQ' ( \beta, w )_{\bK}$ is irreducible and its dimension is given as
$$\dim \, \sQ' ( \beta, w )_{\bK} = 2 \left< \beta - \beta', \rho \right> + \dim \, \sB ( u ).$$
\end{thm}

\begin{proof}
By Lemma \ref{Q-split}, we know that $( R ^ \beta_ v )_{\bK}$ is a quotient of $( R ^ \beta _ w )_{\bK}$ for $v \in W_\af$ such that $v \le_\si w$. Hence, we have $\sQ' ( \beta, w )_{\bK} \cap \bQ_G ( v )_{\bK} = \sQ' ( \beta, v )_{\bK}$. By Lemma \ref{trans}, the scheme $\sQ' ( \beta, w )_{\bK} \cap \bO ( v )_{\bK}$ is isomorphic to $\sQ' ( \beta - \gamma, u' )_{\bK} \cap \bO ( u' )_{\bK}$, where $v = u' t_{\gamma}$ ($u' \in W, \gamma \in Q_+^{\vee}$). By Lemma \ref{Q-mod} and Theorem \ref{aff-Zas-p}, we have:
\begin{itemize}
\item $( \sQ' ( \beta, w )_{\bK} \cap \bO ( v )_{\bK} ) = ( \sQ' ( \beta, v )_{\bK} \cap \bO ( v )_{\bK} )$ for each $v \in W_\af$ such that $v \le_\si w$;
\item The variety $( \sQ' ( \beta, w )_{\bK} \cap \bO ( v )_{\bK} )$ is irreducible for every $v \in W_\af$.
\end{itemize}
We have an equality
$$\dim \, ( \sQ' ( \beta, w )_\C \cap \bO ( v )_\C ) = \dim \, \sQ ( \beta - \gamma, u' ) = 2 \left< \beta - \gamma, \rho \right> + \dim \, \sB ( u' )$$
for $v = u' t_{\gamma}$ ($u' \in W, \gamma \in Q^{\vee}_+$) by Lemma \ref{Q-mod} and Theorem \ref{Dr} (whenever $\sQ' ( \beta, w )_\C \cap \bO ( v )_\C \neq \emptyset$). In addition, Lemma \ref{Q-flat} implies
\begin{itemize}
\item We have $\dim \, \sQ' ( \beta, w ) _\C = \dim \, \sQ' ( \beta, w )_{\bK}$.
\end{itemize}

In particular, we have the desired dimension formula if $\sQ' ( \beta, w )_{\bK}$ is irreducible with its Zariski open dense subset $( \sQ' ( \beta, w )_{\bK} \cap \bO ( w )_{\bK} )$.

Since $\sQ ( \beta, w )$ and $\sQ' ( \beta, w )_\C$ shares a same open subset and the former is irreducible, we have $\sQ ( \beta, w ) = \sQ' ( \beta, w )_\C$ as closed subvarieties of $( \bQ _G^{\ra} )_\C$ if $\sQ' ( \beta, w )_\C$ is irreducible.

Therefore, it suffices to prove that $\sQ' ( \beta, w )_{\bK}$ is irreducible (with its Zariski open dense subset $( \sQ' ( \beta, w )_{\bK} \cap \bO ( w )_{\bK} )$). By Theorem \ref{cover} and Corollary \ref{Q'nonempty}, it suffices to prove that
\begin{equation}
( \sQ' ( \beta, w )_{\bK} \cap \bO ( s w )_{\bK} ) \subset \overline{(\sQ' ( \beta, w )_{\bK} \cap \bO ( w )_{\bK} )}\label{Q'-incl}
\end{equation}
for every $w \in W_\af$ and every reflection $s \in W_\af$ such that $\ell^{\si} ( s w ) = \ell ^{\si} ( w ) + 1$ and $w_0 t_{\beta} \le_\si s w \le_\si w$. Here $\bQ_G ( s v )_{\bK} \subset \bQ _G ( v )_{\bK}$ is an irreducible component of the boundary by Theorem \ref{si-Bruhat} 3) (cf. Corollary \ref{pos-fp}). This boundary component is these cut out as (a part of) the zero of $\bv_{v w_0 \la}^{\vee} \in \bW_{vw_0} ( \la )^{\vee}$ ($\la \in P_+$). Thus, we deduce that $\overline{( \sQ' ( \beta, w )_{\bK} \cap \bO ( s w )_{\bK} )}$ contains an irreducible component of
$$\overline{(\sQ' ( \beta, w )_{\bK} \cap \bO ( w )_{\bK} )} \cap \{ f = 0\}$$
for some single equation $f$ (an instance of $\bv_{vw_0 \la}^{\vee}$'s) if it is nonempty. By the comparison of dimensions, this forces
\begin{equation}
\overline{(\sQ' ( \beta, w )_{\bK} \cap \bO ( w )_{\bK} )} \cap \{ f = 0\} \cap \bO ( s w )_{\bK} \subset ( \sQ' ( \beta, w )_{\bK} \cap \bO ( s w )_{\bK} )\label{cut-out-by-div}
\end{equation}
to be an irreducible component if the LHS is nonempty. Consider the $( H \times \Gm )$-invariant curve $C$ that connects the (unique) $( H \times \Gm )$-fixed points $p_{sw}$ in $\bO ( s w )_{\bK}$ and $p_w$ in $\bO ( w )_{\bK}$. The curve $C$ is the closure of the orbit of an one-parameter unipotent subgroup action corresponding to a root in $\Delta_{\af,+}$ applied to $p_w$. In particular, we have $C \subset \bO (sw)_{\bK} \sqcup \bO ( w )_{\bK}$. By the degree bound (from the above) offered by Lemma \ref{Q-mod}, we conclude that $C \subset \overline{(\sQ' ( \beta, w )_{\bK} \cap \bO ( w )_{\bK} )}$. Therefore, the LHS of (\ref{cut-out-by-div}) is nonempty. Hence, the irreducibility of $( \sQ' ( \beta, w )_{\bK} \cap \bO ( s w )_{\bK} )$ forces
$$\overline{(\sQ' ( \beta, w )_{\bK} \cap \bO ( w )_{\bK} )} \cap \{ f = 0\} \cap \bO ( s w )_{\bK} = ( \sQ' ( \beta, w )_{\bK} \cap \bO ( s w )_{\bK} ).$$

Therefore, we conclude (\ref{Q'-incl}). This implies $\sQ ( \beta - \beta', u ) = \sQ' ( \beta, w )_{\C}$ as an irreducible (reduced) closed subvariety of $( \bQ_G^{\ra} )_\C$, and $\sQ' ( \beta, w )_{\bK}$ is irreducible in general. Its dimension $\dim \sQ ( \beta - \beta', u )$ comes from the dimension of $( \sQ' ( \beta, w )_{\bK} \cap \bO ( w )_{\bK} )$, that is a Zariski open dense subset of $\sQ' ( \beta, w )_{\bK}$.
\end{proof}

\begin{cor}\label{pure-irr}
For each $w,v \in W_\af$, the dimension of an irreducible component of $\sQ' ( v, w )_{\bK}$ is $\ell^\si ( v ) - \ell^\si ( w )$ if $\sQ' ( v, w )_{\bK} \neq \emptyset$ $($that is equivalent to $v \le_\si w$ by Corollary {\rm\ref{Q'nonempty}}$)$. In particular, $\sQ' ( v, w )_{\bK}$ is equidimensional.
\end{cor}

\begin{proof}
The case $v = w_0 t_{\beta}$ with $\beta \in Q^{\vee}$ follows from Theorem \ref{QQ-isom}. We have $\sQ' ( w, w )_{\bK} = \Spec \, \bK$ ($w \in W_\af$) since $\dim \, \bW_{ww_0} ( \la )_{\bK} \cap \bW^- _{ww_0} ( \la )_{\bK} = 1$ for every $\la \in P_+$ by the $P^\af$-weight comparison.

Let $v = u t_{\beta}$ ($u \in W$). In view of the proof of Theorem \ref{QQ-isom}, the variety $\sQ' ( v, w )_{\bK}$ is obtained from $\sQ' ( w_0 t_{\beta}, w )_{\bK}$ by a $( \ell ( w_0 ) - \ell ( u ) )$-successive hyperplane cuts by the $\theta$-twists of the $\bI$-stable boundaries of $\bQ_G ( \bullet )_{\bK}$ (cf. Remark \ref{enh-emb}). Each of these hyperplane cuts lowers the dimension of an irreducible component by at most one if the intersection is nonempty. Since $\ell ( w_0 ) - \ell ( u ) = \ell^\si ( w_0t_{\beta} ) - \ell^\si ( v )$, the dimension inequality $\ge$ always hold.

Every irreducible component of $\sQ' ( v, w )_{\bK}$ is $(H \times \Gm)$-stable by construction. Since we have only one $(H \times \Gm)$-fixed point in each $\bI$-orbit of $( \bQ^\ra_G )_{\bK}$, all but one $( H \times \Gm)$-fixed point of $\bQ_G(w)_{\bK}$ ($w \in W_\af$) lies in the boundary. Thus, if an irreducible component of $\sQ' ( v, w )_{\bK}$ does not meet any $\bI$-stable boundary divisors of $\bQ_G ( w )_{\bK}$ and the $\theta$-twists of any $\bI$-stable boundary divisors of $\bQ_G ( vw_0 )_{\bK}$, then such an irreducible component does not contain a $(H \times \Gm)$-fixed point unless $v = w$. This is a contradiction (to the properness of $\sQ' ( v, w )_{\bK}$). Hence, every irreducible component of $\sQ' ( v, w )_{\bK}$ meets its boundary cut out by an $\bI$-stable boundary divisor of $\bQ_G(w)_{\bK}$ or the $\theta$-twist of an $\bI$-stable boundary divisor of $\bQ_G ( vw_0 )_{\bK}$.

Assume to the contrary to deduce contradiction. Then, we have $\sQ' ( v, w )_{\bK}$ whose irreducible components have dimension $(\ell^\si ( v ) - \ell^\si ( w ))$, but one of its irreducible component is in fact contained in $\sQ' ( v', w' )_{\bK}$ with $\sQ' ( v', w' )_{\bK} \subsetneq \sQ' ( v, w )_{\bK}$ for $v', w' \in W_\af$, and hence it gives an irreducible component of $\sQ' ( v', w' )_{\bK}$ with its dimension $> (\ell^\si ( v' ) - \ell^\si ( w' ))$. From this, we can lower the dimension of an irreducible component of $\sQ' ( v', w' )_{\bK}$ by intersecting with a divisor to raise $v'$ or lower $w'$ successively to reach to the case $v' = w'$ (as in the second paragraph). Since dimension drops at most one in each step, we find that $\dim \, \sQ' ( w, w )_{\bK} > 0$ for some $w \in W_\af$. This is a contradiction, and hence this case does not happen.

These imply that the dimension equality always hold as required.
\end{proof}

\begin{rem}\label{rpi}
The analogous assertions for Theorem \ref{QQ-isom} and Corollary \ref{pure-irr} holds for the case $\tJ \subsetneq \tI$ (cf. Corollaries \ref{gen-norm} and \ref{dim-general}).
\end{rem}

\begin{thm}\label{coh}
For each $w, v \in W_\af$, $\tJ \subset \tI$, and $\la \in P_{\tJ,+}$, we have
$$H^i ( \sQ'_\tJ ( v, w )_{\bK}, \cO _{\sQ'_{\tJ} ( v, w )_{\bK}} ( \la ) ) \cong \begin{cases} R ^v _w ( \la )_{\bK} & (i = 0, \la \in P_{\tJ,++})\\ \{ 0 \} & (i \neq 0) \end{cases}.$$
Moreover, if $w',v' \in W_\af$ satisfies $\sQ'_\tJ ( v', w' ) \subset \sQ'_\tJ ( v, w )$ and $\la \in P_{\tJ,+}$, then the restriction map induces a surjection
$$H^0 ( \sQ'_\tJ ( v, w )_{\bK}, \cO _{\sQ'_\tJ ( v, w )_{\bK}} ( \la ) ) \longrightarrow \!\!\!\!\! \rightarrow H^0 ( \sQ'_\tJ ( v', w' )_{\bK}, \cO _{\sQ'_\tJ ( v', w' )_{\bK}} ( \la ) ).$$
\end{thm}

\begin{rem}
The strict dominance condition in Theorem \ref{coh} cannot be removed naively. For $G = \SL ( 3 )$, $\bK = \C$, and $w = e$, we have
$$\dim \, \sQ'_{\{1\}} ( 2 \al_1^{\vee}, e )_{\C} = 8 > 7 = \dim \, \sQ' ( 2 \al_1^{\vee}, e )_{\C}.$$
This results in the non-injectivity of the pullback map
$$H^0 ( \sQ'_{\{1\}} ( 2 \al_1^{\vee}, e )_{\C}, \cO_{\sQ'_{\{1\}} ( 2 \al_1^{\vee}, e )_{\C}} ( 3 \varpi_1 ) ) \longrightarrow\!\!\!\!\!\rightarrow H^0 ( \sQ' ( 2 \al_1^{\vee}, e )_{\C}, \cO_{\sQ' ( 2 \al_1^{\vee}, e )_{\C}} ( 3 \varpi_1 ) ),$$
where the LHS is $R^{2 \al_1^{\vee}}_e ( 3 \varpi_1 )_{\C} = R^{2 \al_1^{\vee}}_e (\{ 1 \}, 3 \varpi_1 )_{\C} \cong S^3 \left( \C^3 \oplus \C^3 z \oplus \C^3 z^2 \right)^{\vee}$, while the RHS is its quotient by the $G$-invariants of $( \C^3 \otimes \C^3 z \otimes \C^3 z^2 )^{\vee}\subset S^3 \left( \C^3 \oplus \C^3 z \oplus \C^3 z^2 \right)^{\vee}$. In fact, the ring $( R ^{2 \al_1^{\vee}} _e )_{\C}$ is reduced but not integral, and the extra irreducible component is contained in the irrelevant locus.
\end{rem}

\begin{proof}[Proof of Theorem \ref{coh}]
We first observe that it is enough to show the surjectivity part of the assertion only for the $w = w'$ case since the $v = v'$ case follows by applying $\theta$ (and the rest of the cases follow by a repeated applications of the two cases).

Since $\bO ( \tJ, w )_{\F_p} \subset \bQ_{G, \tJ} ( w )_{\F_p}$ is affine, the union $Z$ of codimension one $\bI$-orbits contains the support of an ample Cartier divisor $D$. Since our Frobenius splitting of $\bQ_{G, \tJ} ( w )_{\F_p}$ is $\bI$-canonical and it is compatible with the $\bI$-orbit closures (Corollary \ref{FQ-uniq}), we deduce that $\bQ_{G, \tJ} ( w )_{\F_p}$ is $D$-split in the sense of \cite[Definition 3]{BL03} (cf. \cite[Definition 1.2]{Ram87}) by replacing the Frobenius splitting with its $e$-th power for $e \gg 0$ so that $p^e$ exceeds the multiplicity of each component of $D$ (as in \cite[3.5--3.9]{Sm00}). Each irreducible component of $Z$ intersects properly with $\sQ'_\tJ ( v', w )_{\F_p} \subset \sQ'_\tJ ( v, w )_{\F_p}$ ($v' \in W_\af$) whenever the intersection is nonempty by Corollary \ref{pure-irr} (and Remark \ref{rpi}). Thus, we deduce that $\sQ'_\tJ ( v, w )_{\F_p}$ admits a Frobenius $D$-splitting compatible with $\sQ'_\tJ ( v', w )_{\F_p}$ ($v' \in W_\af$). Therefore, the cohomology vanishing part of the assertion and the surjectivity part of the assertion (for $w' = w$) follow from \cite[Proposition 1.13 (ii)]{Ram87}. The cohomology vanishing part of the assertion lifts to $\mathsf{char} \, \bK = 0$ by \cite[Proposition 1.6.2]{BK05}. The surjectivity part of the assertion lifts to $\mathsf{char} \, \bK = 0$ by \cite[Corollary 1.6.3]{BK05} and Lemma \ref{Q-flat}.

It remains to calculate $H^0 ( \sQ'_\tJ ( v, w )_{\bK}, \cO _{\sQ'_\tJ ( v, w )_{\bK}} ( \la ) )$ for each $\la \in P_{\tJ, ++}$. Here $H^0 ( \sQ'_\tJ ( v, w )_{\bK}, \cO _{\sQ'_\tJ ( v, w )_{\bK}} ( \la ) )$ is obtained as the degree $\la$-part of the (graded) normalization of the ring $( R ^v _w ( \tJ ) )_{\bK}$. For each prime $p$, our ring $(R ^v _w ( \tJ ))_{\F_p}$ is weakly normal by Lemma \ref{Q-split}. Let us consider the $P_{\tJ,+}$-graded quotient $R' = \bigoplus_{\la \in P_{\tJ,+}} R' ( \la )$ of $R ^v _w ( \tJ, \la )_{\F_p}$ that annihilates all the irrelevant irreducible components. Since $R ^v _w ( \tJ )_{\F_p}$ is a reduced ring, we have $R ^v _w ( \tJ, \la )_{\F_p} = R' ( \la )$ for $\la \in P_{\tJ,++}$. The ring $R'$ is generated by $\bigoplus_{i \in \tI \setminus \tJ} R' ( \varpi_i )$ as $R ^v _w ( \tJ )_{\F_p}$ is so. We have
\begin{equation}
H^0 ( \sQ'_\tJ ( v, w )_{\F_p}, \cO _{\sQ'_\tJ ( v, w )_{\F_p}} ( \la ) ) \cong R ^v _w ( \tJ, \la )_{\F_p}\label{H-eq-R}
\end{equation}
for sufficiently large $\la \in P_{\tJ, +}$ (see e.g. \cite[I\!I Excecise 5.9]{Har77}). In other words, (\ref{H-eq-R}) hold for $m \la$, where $\la \in P_{\tJ, ++}$ is arbitrary and $m \gg 0$.

Let $R^\sharp ( \la ) := H^0 ( \sQ'_\tJ ( v, w )_{\F_p}, \cO _{\sQ'_\tJ ( v, w )_{\F_p}} ( \la ) )$ for $\la \in P_{\tJ, ++}$. We set $R^\sharp [\la] : = \F_p 1 \oplus \bigoplus_{m \ge 1} R^\sharp ( m \la )$ and $R'[\la] := \bigoplus_{m \ge 0} R' ( m \la ) \subset R'$ for $\la \in P_{\tJ, ++}$. Both are naturally rings. Then, we have a ring extension $R' [ \la ] \subset R^\sharp [ \la ]$ such that their degree $m \gg 0$-parts are the same.

We prove $R' ( \la ) = R^\sharp ( \la )$. The ring $R'[\la]$ is weakly normal as we have an inclusion $R'[\la] \hookrightarrow R ^v _w ( \tJ )_{\F_p}$ (and the reasoning that $R ^v _w ( \tJ )_{\F_p}$ is weakly normal equally applies to $R' [\la]$). Since it is generated by $R' ( \la )$, we deduce that two rings $R'[\la]$ and $R^\sharp[\la]$ share the same spectrum (the irrelevant locus of $R'[\la]$ is one point as it is a $\Z_{\ge 0}$-graded ring, and other points are just the same). This forces $R' ( \la ) = R^\sharp ( \la )$ by the weak normality of $R'[\la]$.

This yields the $H^0$-part of the assertion for $\mathsf{char} \, \bK > 0$ through the extension of scalars. The $H^0$-part of the assertion for $\mathsf{char} \, \bK = 0$ is obtained by taking the generic specialization of the base scheme $\mathrm{Spec} \, \Z$ of $\sQ'_\tJ ( v, w )$.
\end{proof}

\begin{cor}\label{H-inh}
Let $w, v \in W_\af$ and $\tJ \subset \tI$. Assume that the map $\Pi_\tJ : \sQ' ( v, w )_{\bK} \rightarrow \sQ'_\tJ ( v, w )_{\bK}$ defined through the projective coordinate ring is surjective. We have
$$\R^{>0} ( \Pi_\tJ )_* \cO_{\sQ' ( v, w )_{\bK}} \cong \{ 0 \} \text{ and } ( \Pi_\tJ )_* \cO_{\sQ' ( v, w )_{\bK}} \cong \cO_{\sQ'_{\tJ} ( v, w )_{\bK}}.$$
\end{cor}

\begin{rem}\label{cvw}
Since the map $\bQ_G^{\ra} \to \bQ_{G, \tJ}^{\ra}$ is surjective, we can replace $v$ with $v t_{\beta}$ for some $\beta \in \sum_{j \in \tJ} \Z_{\ge 0} \al_j^{\vee}$ to obtain a surjection:
$$\Pi_\tJ : \sQ' ( vt_{\beta}, w )_{\bK} \rightarrow \sQ'_\tJ ( vt_{\beta}, w )_{\bK} = \sQ'_\tJ ( v, w )_{\bK}.$$
For $v = w_0$ and $w = e$, there is an optimal choice of $\beta$, that is sometimes referred to as the Peterson-Woodward threshold (\cite{Woo05}).
\end{rem}

\begin{proof}[Proof of Corollary \ref{H-inh}]
Thanks to \cite[Lemma A.31]{Kum02} and Lemma \ref{ampleness}, it suffices to prove
\begin{align*}
H^{>0} ( \sQ' ( v, w )_{\bK}, \cO_{\sQ' ( v, w )_{\bK}} ( \la ) ) \cong \{ 0 \} \text{ and }\\
H^{0} ( \sQ' ( v, w )_{\bK}, \cO_{\sQ' ( v, w )_{\bK}} ( \la ) ) \cong H^{0} ( \sQ'_{\tJ} ( v, w )_{\bK}, \cO_{\sQ'_{\tJ} ( v, w )_{\bK}} ( \la ) )
\end{align*}
for each $\la \in P_{\tJ,++}$. The $H^{>0}$-part of the assertion follow from Theorem \ref{coh}. The $H^0$-part of the assertion follow as in the latter part of the proof of Theorem \ref{coh} since the surjectivity of $\Pi_\tJ$ guarantees that every non-zero element of $R ^v _w ( \la )_{\F_p}$ is supported outside of the irrelevant locus by pullback and (graded) normalization of the ring $R ^v _w ( \tJ )_{\F_p}$ is the degree $P_{\tJ, +}$-part of the (graded) normalization of the ring $( R ^v _w )_{\F_p}$ by the degree reasons.
\end{proof}

\begin{cor}\label{prRic}
Let $\la \in P_{+}$ and $w, v \in W_\af$. We set $\tJ := \{i \in \tI \mid \left< \al_i^{\vee}, \la \right> = 0 \}$. Assume that $\Pi_{\tJ} ( \sQ' ( v, w )_{\bK} ) = \sQ'_{\tJ} ( v, w )_{\bK}$. Then, we have
$$H^{0} ( \sQ' ( v, w )_{\bK}, \cO_{\sQ' ( v, w )_{\bK}} ( \la ) )^{\vee} = \bW _{ww_0} ( \la )_{\bK} \cap \theta ( \bW _{v} ( -w_0 \la )_{\bK} ).$$
\end{cor}

\begin{proof}
Combine Theorem \ref{coh} for $\sQ'_{\tJ} ( v, w )_{\bK}$ and Corollary \ref{H-inh}.
\end{proof}

\begin{lem}\label{pre-BSDH}
Let $w,v \in W_\af$ and $\tJ \subset \tI$. For each $i \in \tI_\af$ such that $s_i w >_\si w$ and $s_i v >_\si v$, the variety $\sQ'_{\tJ} ( v, w )_\bK$ is $B_i$-stable. In addition, we have a morphism
$$\pi_i : \SL ( 2, i ) \times^{B_i} \sQ'_{\tJ} ( v, w )_\bK \rightarrow \sQ'_{\tJ} ( v, s_i w )_\bK.$$
The map $\pi_i$ is a $\P^1$-fibration if $\sQ'_{\tJ} ( v, w )_\bK$ is $\SL ( 2, i )$-stable $($this never happens under our assumption when $\tJ = \emptyset)$. In general, the fiber of $\pi_i$ is either a point or $\P^1$ if it is nonempty.
\end{lem}

\begin{proof}
For each $\la \in P_{\tJ,++}$, the module $\bW _{v} ( - w_0 \la )_\bK$ is $\bI ( i )$-stable by Corollary \ref{SL2-stable} and (\ref{ord-opp}) and the module $\bW _{ww_0} ( \la )_\bK$ is $\bI$-stable. In particular,
$$\bW _{ww_0} ( \la )_\bK \cap \theta ( \bW _{v} ( -w_0 \la )_\bK ) \equiv \bW _{ww_0} ( \la )_\bK \cap \bW^- _{vw_0} ( \la )_\bK$$
is $B_i$-stable. It follows that the ring $R ^v _w ( \tJ )_\bK$ is $B_i$-stable. Hence the scheme $\sQ'_{\tJ} ( v, w )_\bK$ is also $B_i$-stable.

The map $\pi_i$ is a $\P^1$-fibration if $\sQ'_{\tJ} ( v, w )_{\bK}$ is $\SL ( 2, i )$-stable since we have $\sQ'_{\tJ} ( \beta, w )_{\bK} = \sQ'_{\tJ} ( \beta, s_i w )_{\bK}$ (that in turn holds if and only if $\bW _{s_iw w_0} ( \la )_{\bK} = \bW _{w w_0} ( \la )_{\bK}$ for $\la \in P_{\tJ,++}$ since $\bW^- _{vw_0} ( \la )_\bK$ is $\SL ( 2, i )$-stable; hence it never happens when $\tJ = \emptyset$ by Lemma \ref{contain}) in this case. The fiber of the map $\pi_i$ is either a point or $\P^1$ if $\sQ'_{\tJ} ( v, w )_{\bK}$ is not $\SL ( 2, i )$-stable as the corresponding statement holds true for $\bQ_{G, \tJ} ( w )_{\bK} \subset \bQ_{G, \tJ} ( s_i w )_{\bK}$ by a set-theoretic consideration (and it carries over to any $B_i$-stable locus). Therefore, we conclude the result.
\end{proof}

\begin{prop}\label{BSDH-type2}
Let $w,v \in W_\af$ and $\tJ \subset \tI$. For each $i \in \tI_\af$ such that $s_i w >_\si w$ and $s_i v >_\si v$, we have a surjective map
$$\pi_i : \SL ( 2, i ) \times^{B_i} \sQ'_{\tJ} ( v, w )_\bK \rightarrow \sQ'_{\tJ} ( v, s_i w )_\bK$$
such that $(\pi_i)_* \cO_{\SL ( 2, i ) \times^{B_i} \sQ'_{\tJ} ( v, w )_\bK} \cong \cO_{\sQ'_{\tJ} ( v, s_i w )_\bK}$ and
$$\R^{>0} (\pi_i)_* \cO_{\SL ( 2, i ) \times^{B_i}  \sQ'_{\tJ} ( v, w )_\bK} \cong \{ 0 \}.$$
\end{prop}

\begin{proof}
For each $\la \in P_{\tJ,++}$, the $\la$-graded component
$$R ^v _w ( \tJ, \la )^{\vee} = \bW _{ww_0} ( \la )_\Z \cap \theta^* ( \bW _{v} ( - w_0 \la )_\Z ) \subset \bX ( \la )_\Z$$
of the coordinate ring of $\sQ'_{\tJ} ( v, w )$ is obtained as the $\Z$-span of a subset of positive global basis of $\bW _{ww_0} ( \la )_\Z$.

We set $\sQ^+ _{\tJ} ( v, w )_\bK:= \SL ( 2, i ) \times^{B_i} \sQ'_{\tJ} ( v, w )_\bK$ for simplicity in the below.

In view of \cite[Theorem 4.2.1]{NS16} (cf. Corollary \ref{SL2-stable}), the basis elements of $\theta^* ( \bW _{v} ( - w_0 \la )_\Z )$ afforded by the negative global basis, regarded as a subset of the $\tg$-crystal $\bB ( \bX ( \la ) )$, decomposes into the disjoint union of connected $\mathfrak{sl} ( 2, i )$-crystals (that is the set of labels of a basis of an irreducible $\mathfrak{sl} ( 2 )$-module equipped with combinatorial operations $\widetilde{e}_i$ and $\widetilde{f}_i$ corresponding to $E^{(\bullet)}_i$ and $F^{(\bullet)}_i$, see e.g. \cite[Definition 2.3.1]{Kas91}; it is the $i$-string in \cite{Kas93,Kas05}). Moreover, the basis elements of $\bW _{s_iww_0} ( \la )_\Z$ afforded by the positive global basis, regarded as a subset of the $\tg$-crystal $\bB ( \bX ( \la ) )$, also decomposes into the disjoint union of $\mathfrak{sl} ( 2, i )$-crystals (\cite{Kas05,Kat18}), and the set $S$ of basis elements corresponding to each connected $\mathfrak{sl} ( 2, i )$-crystal satisfies
\begin{equation}
S \cap \bW _{ww_0} ( \la )_\Z = S \hskip 3mm \text{or} \hskip 3mm \{\bv\},
\end{equation}
where $\bv \in S$ is the highest weight vector as a $\mathfrak{sl} ( 2, i )$-crystal, that has dominant weight as a weight of $\mathfrak{sl} ( 2, i )$ (\cite[Lemma 2.6]{Kas05}). In particular, we have
$$S \cap R ^v _w ( \tJ, \la )^{\vee} = S \hskip 3mm \text{or} \hskip 3mm \{\bv\} \hskip 3mm \text{or} \hskip 3mm \emptyset.$$
Therefore, a reinterpretation of \cite[(0.5)]{Kas93} using Theorem \ref{coh} (cf. \cite{Jos85,Kat18}) reads as
\begin{eqnarray*}
\ch \, H^0 ( \cO_{\sQ'_{\tJ} ( v, s_i w )_\bK} ( \la ) ) & = & \ch \, H^0 ( \cO_{\sQ^+ _{\tJ} ( v, w )_\bK} ( \la ) )\\
0 & = & \ch \, H^1 ( \cO_{\sQ^+ _{\tJ} ( v, w )_\bK} ( \la ) ).\label{char-id}
\end{eqnarray*}
for each $\la \in P_{\tJ,++}$. Since the zero-th part of the Demazure functor is the same as taking the maximal integrable $\SL ( 2,i )$-inflation of a (weight semisimple) $B_i$-module, we conclude that inclusion $R ^v _w ( \tJ, \la )^{\vee}_{\bK} \subset R ^v _{s_i w} ( \tJ, \la )^{\vee}_{\bK}$ induces a natural inclusion
\begin{eqnarray}
H^0 ( \sQ'_{\tJ} ( v, s_i w )_{\bK}, \cO_{\sQ'_{\tJ} ( \beta, s_i w )_{\bK}} ( \la ) ) \hookrightarrow H^0 ( \sQ^+_{\tJ} ( v, w )_{\bK}, \cO_{\sQ^+_{\tJ} ( v, w )_{\bK}} ( \la ) ),\label{ind-eq}
\end{eqnarray}
that is an isomorphism by the character comparison. Thus, Lemma \ref{ampleness} implies that 
$$(\pi_i)_* \cO_{\sQ^+_{\tJ} ( v, w )_\bK} \cong \cO_{\sQ'_{\tJ} ( v, s_i w )_\bK} \hskip 5mm \text{and} \hskip 5mm \R^1 (\pi_i)_* \cO_{\sQ^{+}_{\tJ} ( \beta, w )_\bK} \cong \{ 0 \}.$$
Since we have $\R^{\ge 2} (\pi_i)_* \cO_{\sQ^+_{\tJ} ( v, w )_\bK} \cong \{ 0 \}$ by the dimension reason, we conclude the result.
\end{proof}

A particular case of Proposition \ref{BSDH-type2} is worth noting:

\begin{cor}\label{BSDH-type}
Let $\beta \in Q^{\vee}_+$, $w \in W$, and $\tJ \subset \tI$. For each $i \in \tI$ such that $s_i w < w$, we have a surjective map
$$\pi_i : P_i \times^B \sQ'_{\tJ} ( \beta, w )_\bK \rightarrow \sQ'_{\tJ} ( \beta, s_i w )_\bK$$
such that $\R^{\bullet}(\pi_i)_* \cO_{P_i \times^B \sQ'_{\tJ} ( \beta, w )_\bK} \cong \cO_{\sQ'_{\tJ} ( \beta, s_i w )_\bK}$. \hfill $\Box$
\end{cor}

\subsection{Lifting to/from characteristic zero}\label{subsec:tofrom}

\begin{thm}\label{wn-trans}
Let $\gX$ be a Noetherian scheme flat over $\Z$. If $\gX _{\F_p}$ is weakly normal for $p \gg 0$, then $\gX_\C$ is also weakly normal. 
\end{thm}

\begin{proof}
Since the weak normalization commutes with localization \cite[Theorem I\!V.3]{Man80}, we can argue locally. Let $( S,\gm )$ be a local ring of $\gX _{\overline{\Q}}$ and let $S^-$ be the weak normalization of $S$ (\cite[Remark 1]{Yan83}). By the Noetherian hypothesis, we can invert finitely many primes and take a finite algebraic extension of $\Z$ to obtain a ring $A$ such that we have a commutative ring $S_A$ over $A$ and its ideal $\gm_A$ with the following properties:
\begin{itemize}
\item We have $( S_A \otimes_A \overline{\Q}, \gm_{A}  \otimes_A \overline{\mathbb Q}) \cong ( S,\gm )$;
\item The $A$-modules $S_A, \gm_A$, and $S_A / \gm_A$ are torsion-free;
\item The specialization of $A$ to the algebraic closure of a finite field yields a weakly normal (local) ring (\cite[Theorem V.2]{Man80}).
\end{itemize}
As $A$ is a Dedekind domain, we find that $S_A, \gm_A$, and $S_A / \gm_A$ are flat over $A$.

We have $S^- = S [f_1,\ldots,f_n]$, where $f_1,\ldots,f_n$ are integral elements. By multiplying with elements in $\overline{\Q}$, we can assume that $f_1,\ldots,f_n$ are integral over $S_A$. By inverting additional primes in $\Z$ if necessary (to assume that the denominator of $f_i$ in $\mathrm{Frac} ( S )$ does not vanish along specializations and achieve the conditions in the followings), we can further assume that $S_A^- := S_A [f_1,\ldots,f_n]$ is flat over $A$ and it is integral for every specialization of $A$ to a field.

The ring $( S_A^- / \gm_A S_A^- ) \otimes_A \C$ is a finite-dimensional local commutative $\C$-algebra by the weak normality assumption on $S^-$ (see \cite[Remark 1]{Yan83}). In particular, the multiplication action of each element of $( S_A^- / \gm_A S_A^- ) \otimes_A \C$ have a unique eigenvalue. Hence, if we present $( S_A^- / \gm_A S_A^- ) = A [X_1,\ldots,X_m] / \sim$ (where $\sim$ contains the minimal polynomials of the $A$-valued matrix $X_i$), then the minimal polynomial of $X_i$ is of the form $( T - a_i )^{m_i}$ ($a_i \in A$). Therefore, we can assume that each $X_i$ is nilpotent by changing $X_i$ with $X_i + a_i$ if necessary. Hence, we conclude that $(X_1,\ldots,X_m) \subset ( S_A^- / \gm_A S_A^- )$ have eigenvalues zero after specializing to $\overline{\F}_p$.

By assumption, $( S_A ) \otimes_A \overline{\F}_p$ is weakly normal for every possible prime $p$ and every ring homomorphism $A \to \overline{\F}_p$. Hence, the specialization $( S_A^- / \gm_A S_A^- ) \otimes_A \overline{\F}_p$ must contain $\overline{\F}_p$ as its ring direct summand (if it is non-zero). This forces $\mathrm{rank} \, ( S_A^- / \gm_A S_A^- ) = 1$ since we cannot have two linearly independent idempotents that have distinct eigenspaces by the previous paragraph.

Therefore, we deduce that $S_A^- = S_A$, that implies $( S, \gm )$ is itself weakly normal. In view of \cite[Theorem V.2 and Corollary V.3]{Man80}, we conclude the assertion.
\end{proof}

\begin{cor}\label{wn-char0}
For each $\tJ \subset \tI$, $w,v\in W_{\af}$, the scheme $\sQ'_\tJ ( v,w )_\C$ is weakly normal.
\end{cor}

\begin{proof}
Apply Theorem \ref{wn-trans} to Lemma \ref{Q-split}.
\end{proof}

\begin{prop}\label{n-char0}
For each $\tJ \subset \tI$, $v,w \in W$, the variety $\sQ'_{\tJ} ( v, w )_{\overline{\F}_p}$ is normal for $p \gg 0$ provided if $\sQ'_{\tJ} ( v, w )_\C$ is normal. The same is true for the irreducibility.
\end{prop}

\begin{proof}
Since $\sQ'_{\tJ} ( v, w )$ is defined over $\Z$, the scheme $\sQ'_{\tJ} ( v, w )_\C$ is a scalar extension of $\sQ'_{\tJ} ( v, w )_{\overline{\Q}}$. By \cite[\href{https://stacks.math.columbia.edu/tag/038P}{Lemma 038P}]{stacks-project}, we deduce $\sQ'_{\tJ} ( v, w )_{\overline{\Q}}$ is normal. We apply \cite[\href{https://stacks.math.columbia.edu/tag/0364}{Lemma 0364}]{stacks-project} to derive the irreducibility of $\sQ'_{\tJ} ( v, w )_{\overline{\Q}}$. Now apply \cite[Proposition 9.9.4 and Th\'eor\`eme 9.7.7]{EGAIV-3} to $\sQ'_{\tJ} ( v, w )_{\Z} \to \Spec \, \Z$.
\end{proof}

\section{Normality of quasi-map spaces}

In this section, we continue to work under the setting of the previous section with an exception that $\bK = \C$. Also, a point of a scheme (over $\C$) means a closed point unless stated otherwise.

\subsection{Graph space resolution of $\sQ ( \beta )$}\label{GQL}
We refer to \cite{KM94, FP95, BM96, FFKM, GL03} for precise explanations of the material in this subsection. For each non-negative integer $n$ and $\beta \in Q^{\vee}_+$, we set $\sGB_{n, \beta}$ to be the space of stable maps of genus zero curves with $n$-parked points to $( \P^1 \times \sB )$ of bidegree $( 1, \beta )$, that is also called the graph space of $\sB$. A point of $\sGB_{n, \beta}$ represents a genus zero curve $C$ with $n$-marked points, together with a map to $\P^1$ of degree one (obtained by composing the map from $C$ to $\P^1 \times \sB$ with the first projection of the target). Hence, we have a unique $\P^1$-component of $C$ that maps isomorphically onto $\P^1$. We call this component the main component of $C$ and denote it by $C_0$. The space $\sGB_{n, \beta}$ is a normal projective variety by \cite[Theorem 2]{FP95} that have at worst quotient singularities arising from the automorphism of curves (and hence it is smooth as an orbifold). The natural $( H \times \Gm)$-action on $( \P^1 \times \sB )$ induces a natural $( H \times \Gm)$-action on $\sGB_{n, \beta}$.

We have a morphism $\pi_{n, \beta} : \sGB_{n, \beta} \rightarrow \sQ ( \beta )$ that factors through $\mathscr G \mathscr B_{0, \beta}$ (Givental's main lemma \cite{Giv96}; see \cite[\S 8.3]{FFKM}). Let $\mathtt{e}_j : \sB_{n, \beta} \to \sB$ ($1 \le j \le n$) be the evaluation at the $j$-th marked point, and let $\mathtt{ev}_j : \sGB_{n, \beta} \to \sB$ be the $j$-th evaluation map to $\P^1 \times \sB$ composed with the second projection.

Since $\sQ ( \beta )$ is irreducible (Theorem \ref{Dr}), \cite[\S 8.3]{FFKM} asserts that $\sGB_{n,\beta}$ is irreducible (as a special feature of flag varieties, see \cite[\S 1.2]{FP95} and \cite{KP01}).

\subsection{The variety $\sQ ( \beta, v, w )$}\label{Qbvw}

Let $\sGB_{2,\beta}^{\flat}$ denote the subvariety of $\sGB_{2,\beta}$ consisting of points such that the first marked point projects to $0 \in \P^1$, and the second marked point projects to $\infty \in \P^1$ through the projection of quasi-stable curves $C$ to their main component $C_0 \cong \P^1$. Let us denote the restrictions of $\mathtt{ev}_i$ $(i=1,2)$ and $\pi_{2,\beta}$ to $\sGB_{2,\beta}^{\flat}$ by the same letter. By \cite{BF14a, BF14b}, $\sGB_{2,\beta}^{\flat}$ gives a resolution of singularities of $\sQ ( \beta )$ (in an orbifold sense).

Recall that each Schubert cell $\bO_\sB ( w )$ contains a unique $H$-fixed point $p_w$. For each $w \in W$, we set
$$\bO_\sB ^{\mathrm{op}} ( w ) := N^- p_w \subset \sB, \hskip 3mm\text{and}\hskip 3mm \sB^{\mathrm{op}} ( w ) := \overline{\bO_\sB ^{\mathrm{op}} ( w )} = \overline{N^- p_w} \subset \sB.$$
For $w,v \in W$, we define
$$\sGB_{2,\beta}^{\flat} ( w, v ) := \mathtt{ev}_1^{-1} ( \sB ( w ) ) \cap \mathtt{ev}_2^{-1} ( \sB^{\mathrm{op}} ( v ) ) \subset \sGB_{2,\beta}^{\flat}$$
and
$$\sB_{2,\beta} ( w, v ) := \mathtt{e}_1^{-1} ( \sB ( w ) ) \cap \mathtt{e}_2^{-1} ( \sB^{\mathrm{op}} ( v ) ) \subset \sB_{2,\beta}.$$

\begin{thm}[Buch-Chaput-Mihalcea-Perrin \cite{BCMP}]\label{BCMP}
For each $v,w \in W$ and $\beta \in Q^{\vee}_+$, the variety $\sGB_{2,\beta}^{\flat} ( w, v )$ is either empty or a unirational $($and hence connected and irreducible$)$ variety that has rational singularities. The same is true for $\sB_{2,\beta} ( w, v )$. In particular, they are normal. 
\end{thm}

\begin{proof}
Since the both cases are parallel, we concentrate into the case $\sGB_{2,\beta}^{\flat} ( w, v )$.

As $\P^1 \times \sB$ is a homogeneous variety under the group action of $(\SL ( 2 ) \times G)$, Kim-Phandaripande \cite[Theorem 2 and Theorem 3]{KP01} applies and hence $\sGB_{2,\beta}$ is a rational variety. Then, a pair of Schubert subvarieties (with respect to a pair of opposite Borel subgroups of $\SL ( 2 ) \times G$) of $\P^1 \times \sB$ presented as $\{ 0 \} \times \sB ( w )$ and $\{ \infty \} \times \sB^{\mathrm{op}} ( v )$ is used to define $\sGB_{2,\beta}^{\flat} ( w, v )$. Hence, \cite[Proposition 3.2 c)]{BCMP} implies that $\sGB_{2,\beta}^{\flat} ( w, v )$ is either empty or unirational (and hence connected). Since a pair of Schubert varieties with respect to the opposite Borel subgroups forms the dense subset of the pair of translations of Schubert varieties by applying the $G$-action, it must contain a pair of Schubert varieties in general position. Therefore, \cite[Corollary 3.1]{BCMP} implies that $\sGB_{2,\beta}^{\flat} ( w, v )$ has rational singularity. The last assertion is a well-known property of rational singularities \cite[Definition 5.8]{KM98}.
\end{proof}

\begin{cor}\label{cBCMP}
In the setting of Theorem \ref{BCMP}, we additionally assume that $\sGB_{2,\beta}^{\flat} ( w, v )$ is nonempty. Then, we have
$$\dim \sGB_{2,\beta}^{\flat} ( w, v ) = \dim \sQ ( \beta ) - \ell ( w ) - \ell ( vw_0 ).$$
\end{cor}

\begin{proof}
By \cite[Proposition 3.2 b)]{BCMP}, the dimension count of (a Zariski open subset of) $\sGB_{2,\beta}^{\flat} ( w, v )$ can be borrowed from \cite[Theorem 2]{Kle74} applied to the $( \SL ( 2 ) \times G )$-action on $\P^1 \times \sB$. In addition, the map $\pi_{2,\beta}$ is birational. Therefore, we have
\begin{align*}
\dim \sGB_{2,\beta}^{\flat} ( w, v ) &  = \dim \sQ ( \beta ) - \mathrm{codim}_\sB \sB ( w ) - \mathrm{codim}_\sB \sB^{\mathrm{op}} ( v )\\
& = \dim \sQ ( \beta ) - \ell ( w ) - \ell ( vw_0 )
\end{align*}
as required.
\end{proof}

\begin{prop}\label{irrQ0}
For each $v,w \in W$ and $\beta \in Q^{\vee}_+$, the variety $\sQ' ( v t_{\beta}, w )_\C$ is irreducible.
\end{prop}

\begin{proof}
We define
$$\mathring{\sQ}'( vt_{\beta}, w )_\C := \sQ'( vt_{\beta}, w )_\C \setminus \left( ( \bigcup_{w'<_\si w} \sQ'( vt_{\beta}, w' )_\C ) \cup \bigcup_{v'>_\si v} \sQ'( v't_{\beta}, w )_\C \right).$$
In view of Corollary \ref{pure-irr}, every irreducible component of $\sQ' ( vt_{\beta}, w )_\C$ has dimension $\ell^{\si} ( v t_{\beta} ) - \ell^{\si} ( w )$ if it is nonempty. We have $\sQ' ( vt_{\beta}, w )_\C \neq \emptyset$ if and only if $vt_{\beta}\le_\si w$ by Corollary \ref{Q'nonempty}. In particular, we find that $\mathring{\sQ}'( vt_{\beta}, w )_\C \neq \emptyset$ if $vt_{\beta} \le_\si w$, and all of its irreducible components have dimension
\begin{equation}
\ell^{\si} ( v t_{\beta} ) - \ell^{\si} ( w ) = \ell ( v ) - \ell ( w ) + 2 \left< \beta, \rho \right>\label{comp-dim}
\end{equation}
in that case. A point of $\mathring{\sQ}'( vt_{\beta}, w )_\C$ can be seen as a quasi-map of degree $\beta$ with no defects at $0$ and $\infty$ (Theorem \ref{QQ-isom}). In particular, $\sGB_{2,\beta}^\flat ( w, v ) \neq \emptyset$ if $\mathring{\sQ}'( vt_{\beta}, w )_\C \neq \emptyset$. In case $\mathring{\sQ}'( vt_{\beta}, w )_\C \neq \emptyset$, we have
\begin{align}\nonumber
\dim \, \sGB_{2,\beta}^\flat ( w, v ) & = \dim \, \sQ ( \beta ) - \ell ( w ) - \ell ( vw_0 )\\
& = 2 \left< \beta, \rho \right> + \dim \, \sB - \ell ( w_0 ) + \ell ( v ) - \ell ( w ) = (\ref{comp-dim}),\label{dimGBQ}
\end{align}
where the first equality follows from Corollary \ref{cBCMP}, and the second equality follows from Theorem \ref{Dr}. By the surjectivity of $\pi_{2,\beta}$ and (\ref{dimGBQ}), we find that the rational map $\sGB_{2,\beta}^\flat ( w, v ) \dashrightarrow \mathring{\sQ}'( vt_{\beta}, w )_\C$ is surjective and generically finite (it is in fact birational by Theorem \ref{desc-fib} and Proposition \ref{Bconn}). In view of Theorem \ref{BCMP}, the variety $\sGB_{2,\beta}^\flat ( w, v )$ is irreducible if it is nonempty. Therefore, $\mathring{\sQ}'( vt_{\beta}, w )_\C \neq \emptyset$ implies that $\mathring{\sQ}'( vt_{\beta}, w )_\C$ is irreducible. Again by Corollary \ref{pure-irr}, we conclude that $\sQ'( vt_{\beta}, w )_\C$ has a unique irreducible component as required.
\end{proof}

In view of Proposition \ref{irrQ0}, we set $\sQ_{\tJ} ( \beta, v, w ) := \sQ'_{\tJ} ( v t_{\beta}, w )_\C$ for each $\tJ \subset \tI$, $\beta \in Q^{\vee}_+$ and $v,w \in W$ in the below (see also Corollary \ref{gen-norm}). We have $\sQ ( \beta, w_0, w ) = \sQ ( \beta, w )$ by Theorem \ref{QQ-isom}.

\begin{cor}[of proof of Proposition \ref{irrQ0}]\label{irrdimQ0}
For each $\beta \in Q^{\vee}_+$ and $v,w \in W$, the scheme $\sQ ( \beta, v, w )$ is nonempty if and only if $v t_{\beta} \le _\si w$. If it is nonempty, then it has dimension $\ell ( v ) - \ell ( w ) + 2 \left< \beta, \rho \right>$. Moreover, the space $\sQ ( \beta, v, w )$ $($viewed as a subspace of $\sQ ( \beta ))$ contains a quasi-map with no defects at $0$ and $\infty$. \hfill $\Box$
\end{cor}

\begin{rem}
Corollary \ref{irrdimQ0} removes the condition $\gamma \gg 0$ in \cite[Lemma 8.5.2]{FM99}.
\end{rem}

Thanks to Proposition \ref{irrQ0} and its proof, we deduce that the map $\pi_{2,\beta}$ restricts to a $( H \times \Gm)$-equivariant birational proper map
$$\pi_{\beta, w,v} : \sGB_{2,\beta}^{\flat} ( w, v ) \to \sQ ( \beta, v, w )$$
by inspection.

\subsection{From Givental's main lemma}

For each $w, v \in W$, we define subvarieties of $\sB_{2,\beta}$ as:
$$\sB_{2, \beta} [ w ] := \mathtt{e}_1^{-1} ( p_w ) \hskip 5mm \text{ and } \hskip 5mm \sB_{2, \beta} [ w, v ] := \mathtt{e}_1^{-1} ( p_w ) \cap \mathtt{e}_2^{-1} ( p_v ).$$
Similarly, we set
$$\sGB^{\flat}_{2, \beta} [ w, v ] := \mathtt{ev}_1^{-1} ( p_w ) \cap \mathtt{ev}_2^{-1} ( p_v ) \subset \sGB_{2,\beta}^{\flat} \text{ and } \sB_{1, \beta} [ w ] := \mathtt{e}_1^{-1} ( p_w ) \subset \sB_{1, \beta} [ w ].$$

\begin{lem}\label{e-trans}
For each $x,y \in \sB$ and $\beta \in Q^{\vee}_+$, there exists $w \in W$ such that
$$\mathtt{e}_1^{-1} ( x ) \cap \mathtt{e}_2^{-1} ( y ) \cong \sB_{2, \beta} [ w, w_0 ].$$
The same is true for $\mathtt{ev}$ and $\sGB$.
\end{lem}

\begin{proof}
We consider only the case of $\mathtt{e}$ and $\sB$ as the other case is completely parallel. Since $(x,y) \in \sB \times \sB$ and the $G$-action on $\sB$ is transitive, we can assume $y = p_{w_0}$. Since we have $\mathsf{Stab}_G \, y = B$, we can further rearrange $x = p_w$ for some $w \in W$ by (\ref{Bdec}).
\end{proof}

\begin{thm}[Givental's main lemma \cite{Giv96}, see \cite{FFKM} \S 8]\label{desc-fib}
Let $\beta \in Q^{\vee}_+$. Let $(f, D) \in \sQ ( \beta )$ be a quasi-map with its defect $D = \sum_{x \in \P^1 ( \C )} \beta_x \otimes [x]$. Then, we have
$$\pi_{2,\beta}^{-1} (f, D) \cong \sB _{2,\beta_0} [ w ] \times \sB _{2,\beta_\infty} [ w ] \times \prod_{x \in \P^1 ( \C ) \setminus \{0,\infty\}} \sB _{1,\beta_x} [ w ] \subset \sGB_{2,\beta}^{\flat}.$$
In particular, the map $\pi_{2,\beta}$ is an isomorphism along the locus with $D = 0$. \hfill $\Box$
\end{thm}

\begin{rem}
In Theorem \ref{desc-fib}, the first marked points of (the stable maps in) $\sGB_{2,\beta}^{\flat}$ (the marked points at $0 \in \P^1$; i.e. the image of $\mathsf{ev}_1$) is identified with the second marked points of (the stable maps in) $\sB _{2,\beta_0} [ w ]$ (i.e. the image of $\mathsf{e}_2$), and the second marked points of $\sGB_{2,\beta}^{\flat}$ (the marked points at $\infty \in \P^1$) is identified with the second marked points of $\sB _{2,\beta_\infty} [ w ]$. (Other marked points are used to glue pieces of stable maps together.)
\end{rem}

\begin{lem}[\cite{FM99} Lemma 8.5.1]\label{ev-surj}
For each $\beta \in Q^{\vee}_+$ such that $\left< \beta, \alpha_i \right> \ge 1$ for all $i \in \tI$, the evaluation map
$$\mathtt{ev} := ( \mathtt{ev}_1 \times \mathtt{ev}_2 ) : \sGB^{\flat}_{2, \beta} \longrightarrow \sB \times \sB$$
is surjective.
\end{lem}

\begin{proof}
Taking into account the fact that $\sGB^{\flat}_{2, \beta}$ is projective, it suffices to prove that the tangent map associated to $\mathtt{ev}$ is surjective on a dense open subset of $\sGB^{\flat}_{2, \beta}$.

Since the map $\pi_{2,\beta}$ is birational by Theorem \ref{desc-fib}, we replace the problem with the case of a genuine map $f : \P^1 \rightarrow \sB$. Thanks to \cite[Proposition 3.5]{FM99}, $\sQ ( \beta )$ and hence $\sGB^{\flat}_{2, \beta}$ are smooth at $f$. Moreover, its tangent space is described as
$$H^0 ( \P^1, f^* T \sB ), \text{ and } H^1 ( \P^1, f^* T \sB ) = \{ 0 \}$$
and the filtration of $T \sB$ as $G$-equivariant line bundles yields the following associated graded
$$\bigoplus_{\al \in \Delta_+} H^0 ( \P^1, \cO_{\P^1} ( \left< \beta, \al \right> ) ).$$
(Here we used $\left< \beta, \alpha_i \right> \ge -1$ for each $i \in \tI$.) In particular, we have
$$\dim \, H^0 ( \P^1, f^* T \sB ) = \Sigma_{\al \in \Delta_+} \dim \, H^0 ( \P^1, \cO_{\P^1} ( \left< \beta, \al \right> ) ).$$
The effect of fixing the image of two points $0, \infty \in \P^1$ corresponds to imposing divisor twist by $\cO_{\P^1} ( - [0] - [\infty] )$. We have
\begin{align*}
\dim \, H^0 ( \P^1, f^* T \sB \otimes_{\cO_{\P^1}} \cO_{\P^1} ( - [0] - [\infty] ) ) & \le 
\dim \, \bigoplus_{\al \in \Delta_+} H^0 ( \P^1, \cO_{\P^1} ( \left< \beta, \al \right> - 2 ))\\
& = \dim \, \bigoplus_{\al \in \Delta_+} H^0 ( \P^1, \cO_{\P^1} ( \left< \beta, \al \right> )) - 2 |\Delta_+|\\
& = \dim \, \sQ ( \beta ) - 2 \dim \, \sB \\
& = \dim \, \sGB^{\flat}_{2, \beta} - 2 \dim \, \sB.
\end{align*}
Here the first inequality comes from the short exact sequences, the second equality is $\left< \beta, \alpha_i \right> \ge 1$ for each $i \in \tI$, the third one is the smoothness of $\sQ ( \beta )$ at $f$ and $|\Delta_+| = \dim \, \sB$, and the fourth one is the birationality of $\pi_{2,\beta}$ (restricted to $\sGB^\flat_{2,\beta} \subset \sGB_{2,\beta}$).

From this, we deduce that $d \pi_{2,\beta}$ is generically surjective as required.
\end{proof}

\begin{prop}\label{Xconn}
For each $\beta \in Q^{\vee}_{+}$ such that $\left< \beta, \alpha_i \right> \ge 1$ for all $i \in \tI$ and $w, v \in W$, the scheme $\sGB^{\flat}_{2, \beta} [ w, v ]$ is connected and nonempty.
\end{prop}

\begin{proof}
The map $\mathtt{ev}$ borrowed from Lemma \ref{ev-surj} is $G$-equivariant, and $( \sB \times \sB )$ admits an open dense $G$-orbit $\mathbf O$ by the Bruhat decomposition. Since $\sGB^{\flat}_{2, \beta}$ is an irreducible variety, so is its Zariski open set $\mathtt{ev}^{-1} ( \mathbf O )$. If we consider $x \in \mathbf O$, then the irreducible components of $\mathtt{ev}^{-1} ( x )$ (that must be finite as we consider varieties, that are finite types) must be permuted by $\mathsf{Stab} _G \, x \cong H$. Since $H$ is connected, we cannot have a non-trivial action. Therefore, the irreducible component of $\mathtt{ev}^{-1} ( x )$ must be unique. Thanks to the Stein factorization theorem \cite[\href{https://stacks.math.columbia.edu/tag/03H0}{Theorem 03H0}]{stacks-project}, the map $\mathtt{ev}$ factors through the normalization $Y$ of $( \sB \times \sB )$ inside (the function field of) $\sGB^{\flat}_{2, \beta}$ and the map $\sGB^{\flat}_{2, \beta} \to Y$ has connected fiber. Since the general fiber of $\mathtt{ev}$ is connected, so is the map $Y \rightarrow ( \sB \times \sB )$. It follows that $Y \rightarrow ( \sB \times \sB )$ is a birational map (that is also finite by the Stein factorization theorem). This implies that $Y$ and $(\sB \times \sB )$ share the same function field. Thus, we conclude that $Y \cong ( \sB \times \sB )$ by the normality of $( \sB \times \sB )$. This particularly means that every fiber of $\mathtt{ev}$ is connected. Therefore, the assertion follows by choosing $(H \times H)$-fixed points in $( \sB \times \sB )$ as particular cases.
\end{proof}

\begin{prop}\label{Bconn}
For each $\beta \in Q^{\vee}_+$ such that $\left< \beta, \alpha_i \right> \ge 1$ for all $i \in \tI$ and $w, v \in W$, the scheme $\sB _{2,\beta} [ w, v ]$ is connected and nonempty. For each $\beta \in Q^{\vee}_+$ and $w \in W$, the schemes $\sB _{1,\beta} [ w ]$ and $\sB _{2,\beta} [ w ]$ are connected and nonempty.
\end{prop}

\begin{proof}
We prove the first assertion. We have a rational map $\sGB^{\flat}_{2, \beta} [ w, v ] \rightarrow \sB _{2,\beta} [ w, v ]$ obtained by forgetting the map to $\P^1$. Moreover, the locus this map is not defined consists of stable maps whose main components have degree $0$ and have only two marked points. By modifying the universal family by adjoining such two marked points in such a degree $0$ main components, we conclude that $\sGB^{\flat}_{2, \beta} [ w, v ] \rightarrow \sB _{2,\beta} [ w, v ]$ extends to a map of topological spaces. By examining the condition to be a stable map \cite[\S 1.1]{FP95}, we deduce that this map is surjective onto the image. Therefore, the connectedness and the nonemptiness of $\sGB^{\flat}_{2, \beta} [ w, v ]$ implies these of $\sB _{2,\beta} [ w, v ]$, that is the first assertion.

The second assertion is straight-forward from the irreducibility of $\sB _{n,\beta}$, together with the fact that $\mathtt{e}_1$ is $G$-equivariant fiber bundle over $\sB$ (see the proof of Proposition \ref{Xconn}).
\end{proof}

\begin{lem}\label{desc-fib+}
Let $\beta \in Q^{\vee}_+$. Let $(f, D) \in \sQ ( \beta )$ be a quasi-map with its defect $D$ with the following properties:
\begin{itemize}
\item $D = \sum_{x \in \P^1 ( \C )} \beta_x \otimes [x]$;
\item $\left< \beta_0, \al_i \right> \ge 1$ and $\left< \beta_\infty, \al_i \right> \ge 1$ for each $i \in \tI$.
\end{itemize}
Then, $\pi_{\beta, w,v}^{-1} (f, D)$ is connected for every $w,v \in W$ if it is nonempty.
\end{lem}

\begin{proof}
By Proposition \ref{Bconn} and Theorem \ref{desc-fib}, we can forget about the contribution of $\sB _{1,\beta_x} [ w ]$ when $x \neq 0,\infty$. By our assumption and Proposition \ref{Bconn}, we know that
$$\sB _{2,\beta_0} [u,w']\neq \emptyset \text{ and } \sB _{2,\beta_\infty} [u,v']\neq \emptyset$$
and are nonempty and connected for each $u,w',v' \in W$. Since $\sB _{2,\beta_0}$ and $\sB _{2,\beta_\infty}$ are proper, we always find a limit point with respect to the $H$-action. It follows that $\sB _{2,\beta_0} [ u ] \cap \mathtt{e}_2^{-1} ( \sB ( u' ) )$ (resp. $\sB _{2,\beta_\infty} [ u ] \cap \mathtt{e}_2^{-1} ( \sB^{\mathrm{op}} ( u' ) )$) is connected for each $u,u' \in W$ as we can connect every two points by appropriately sending to/from $H$-limit points that are contained in a connected component of the form $\sB _{2,\beta_0} [u,w']$ ($w' \in W$). Thanks to Theorem \ref{desc-fib}, we conclude the assertion.
\end{proof}

\subsection{Normality of $\sQ ( \beta, v, w )$}

Let $\sQ^+ ( \beta, v, w )$ be the normalization of $\sQ ( \beta, v, w )$ for each $\beta \in Q^{\vee}_+$ and $v, w \in W$. We denote the normalization map by $\eta_{\beta, v, w} : \sQ^+ ( \beta, v, w ) \to \sQ ( \beta, v, w )$.

\begin{prop}\label{conn-to-norm}
For each $\beta \in Q^{\vee}_+$ and $w, v \in W$, the variety $\sQ ( \beta, v, w )$ is normal if and only if every fiber of $\pi_{\beta,v,w}$ is connected.
\end{prop}

\begin{proof}
As $\sGB_{2,\beta}^{\flat} ( v, w )$ is normal and $\pi_{\beta,v,w}$ is proper, we know that
$$( \pi_{\beta, v, w} ) _* \cO_{\sGB_{2,\beta}^{\flat} ( v, w )} \cong \cO_{\sQ^+ ( \beta, v, w )}.$$
The properness of $\pi_{\beta,v,w}$ also implies that $\cO_{\sQ^+ ( \beta, v, w )}$ is a coherent sheaf on $\sQ ( \beta, v, w )$. For each closed point $x$ of $\sQ ( \beta, v, w )$, we set
\begin{eqnarray}
\Theta (x) := \dim _\C \cO_{\sQ^+ ( \beta, v, w )} \otimes _{\cO_{\sQ ( \beta, v, w )}} \C_x .\label{eq-number}
\end{eqnarray}

By the Stein factorization theorem, the map $\eta_{\beta, v, w}$ is finite. By Corollary \ref{wn-char0} and \S \ref{Qbvw}, we know that the variety $\sQ ( \beta, v, w )$ is weakly normal. From this, we deduce that $\Theta (x) = \# \eta_{\beta,v,w}^{-1} ( x )$ (cf. \cite[Remark 1]{Yan83}). Moreover, it counts the number of irreducible components of the fiber of $\eta_{\beta,v,w}$.

The coherence of $\cO_{\sQ^+ ( \beta, v, w )}$ implies that the RHS of (\ref{eq-number}) is an upper-semicontinuous function on $\sQ ( \beta, v, w )$, and hence so is $\Theta$.

If we have $\Theta \equiv 1$ on $\sQ ( \beta, v, w )$, then we have $\sQ^+ ( \beta, v, w ) = \sQ ( \beta, v, w )$ by the weak normality of the latter (cf. \cite[Remark 1]{Yan83}). Therefore, the if part of the assertion follows.

If we have $\Theta \not\equiv 1$ on $\sQ ( \beta, v, w )$, then we have $\sQ^+ ( \beta, v, w ) \neq \sQ ( \beta, v, w )$. Hence, the only if part of the assertion follows.

These complete the proof of Proposition \ref{conn-to-norm}.
\end{proof}

\begin{cor}[Braverman-Finkelberg]\label{Qnorm-BF}
For each $\beta \in Q^{\vee}_+$, the variety $\sQ ( \beta )$ is normal.
\end{cor}

\begin{rem}
Our proof of Corollary \ref{Qnorm-BF} is independent of \cite{BF14a} (however based on common former papers \cite{FM99,FFKM}). Hence, we obtain a new proof of the normality of $\sQ ( \beta )$ and $\sZ ( \beta, w_0 )$. Together with Theorem \ref{coh}, Corollary \ref{Qnorm-BF} also makes the contents in \cite{KNS17} logically independent of \cite{BF14a} (cf. Appendix A).
\end{rem}

\begin{proof}[Proof of Corollary \ref{Qnorm-BF}]
Recall that $\sQ ( \beta ) = \sQ ( \beta, w_0, e )$.
We borrow the upper semi-continuous function $\Theta$ that counts the number of connected components of the fiber of $\eta_{\beta, w_0, e}$ from (\ref{eq-number}) in the proof of Proposition \ref{conn-to-norm}. 

By Proposition \ref{conn-to-norm}, it suffices to prove that $\Theta \equiv 1$ on $\sQ ( \beta )$. In other words, it suffices to prove that the fiber $\eta_{\beta, w_0, e}^{-1} ( x )$ is connected for each $x \in \sQ ( \beta )$.

By Theorem \ref{desc-fib}, we deduce that the set of connected components of $\eta_{\beta,w_0,e}^{-1} ( x )$ is in bijection with the set of connected components of $\prod_{y \in \P^1 ( \C )} \sB_{\beta_y} [w]$. By Proposition \ref{Bconn}, this latter space is connected.

Therefore, we conclude the result.
\end{proof}

\begin{thm}\label{Qnorm-w0}
For each $\beta \in Q^{\vee}_+$ and $w \in W$, the varieties $\sQ ( \beta, w_0, w_0 )$ and $\sQ ( \beta, e, w_0 )$ are normal.
\end{thm}

\begin{proof}
We set $(v,w) = (w_0,w_0)$ or $(e,w_0)$. We borrow the upper semi-continuous function $\Theta$ that counts the number of connected components of the fiber of $\eta_{\beta, v, w}$ from (\ref{eq-number}) in the proof of Proposition \ref{conn-to-norm}.

By Proposition \ref{conn-to-norm}, it suffices to prove that $\Theta \equiv 1$ on $\sQ ( \beta, v, w )$ by assuming the contrary to deduce contradiction. For each $x \in \sQ ( \beta, v, w )$ such that $\Theta ( x ) \ge 2$, the fiber $\eta_{\beta, w_0, e}^{-1} ( x )$ is disconnected.

By our choice of $(v,w)$ and Theorem \ref{desc-fib} (cf. Proposition \ref{Bconn} and Corollary \ref{Qnorm-BF}), we deduce that the set of connected components of $\eta_{\beta, v, w}^{-1} ( x )$ is in bijection with the set of connected components of $\sB _{2,\beta_0} [ u, w_0 ]$ or $\sB _{2,\beta_0} [ u, w_0 ] \times \sB_{2,\beta_{\infty}} [u',w_0]$ for some $u,u' \in W$. 

To see whether this is the case, we specialize to the case of $(v, w ) = ( w_0, w_0 )$ (to guarantee that the contribution at $\infty \in \P^1$ in Theorem \ref{desc-fib} is the same as the points in $\C^{\times} \subset \P^1$, that is connected by Proposition \ref{Bconn}). We can choose $\beta_0 < \beta_0' \in Q^{\vee}_+$ such that $\left< \al_i^{\vee}, \beta_0' \right> \ge 1$ for every $i \in \tI$. By Proposition \ref{Bconn}, we deduce that the fiber of the locus $\sZ$ of $\sQ ( \beta, w_0 )$ whose points have defect $\beta_0'$ along $0$ is connected (if it is nonempty) for every $\beta \in Q^{\vee}$. Hence we have $\Theta ( y ) = 1$ for each $y \in \sZ$.

For each $\tilde{x} \in \sGB_ {2, \beta}^{\flat}$ such that $x = \pi_{2,\beta} ( \tilde{x} ) \in \sQ ( \beta, w_0 )$ has defect $\beta_0$ at $0$, we can replace $\beta$ with $\beta + \beta_0' - \beta_0$ and add additional irreducible components $C'$ to the main component $\P^1$ of $\tilde{x}$ (as quasi-stable curves) outside of $0 \in \P^1$ (whose images to $\sB$ have their degrees sum up to $(\beta_0' - \beta_0)$). (This does not change the defect of $x$ at $0$, and also does not change $\Theta ( x )$.) Then, we shrink all the inserting points of $C'$ on $\P^1$ to $0$ to obtain $\tilde{x}' \in \sGB_ {2, \beta}^{\flat}$ in the limit, that exists by the valuative criterion of properness as $\sGB_ {2, \beta}^{\flat}$ is projective. By examining the images of this family on $\sQ ( \beta, w_0 )$ via $\pi_{2,\beta}$, we deduce $y = \pi_{2,\beta} ( \tilde{x}' ) \in \sZ$. Therefore, the semi-continuity of $\Theta$ implies that $\Theta ( x ) \le \Theta ( y ) = 1$, that is $\Theta ( x ) = 1$. Hence, $\sB _{2,\beta_0} [ w, w_0 ]$ must be connected. This is a contradiction, and we conclude that $\Theta \equiv 1$ (for general $(v,w) \in \{ (e,w_0),(w_0,w_0)\}$ by the two paragraphs ahead).

Therefore, Proposition \ref{conn-to-norm} implies the result.
\end{proof}

\begin{cor}\label{BSDH-typeQ}
Let $\beta \in Q^{\vee}_+$ and $w \in W$. For each $i \in \tI$ such that $s_i w < w$, we have a surjective map
$$\pi_i : P_i \times^B \sQ ( \beta, w ) \rightarrow \sQ ( \beta, s_i w )$$
such that $(\pi_i)_* \cO_{P_i \times^B \sQ ( \beta, w )} \cong \cO_{\sQ ( \beta, s_i w )}$ and $\R^{>0} (\pi_i)_* \cO_{P_i \times^B \sQ ( \beta, w )} \cong \{ 0 \}$.
\end{cor}

\begin{proof}
Combine Corollary \ref{BSDH-type} with Theorem \ref{QQ-isom} (and take the generic localizations over $\Z$).
\end{proof}

\begin{cor}\label{Qnorm-w}
For each $\beta \in Q^{\vee}_+$ and $w \in W$, the variety $\sQ ( \beta, w )$ is normal.
\end{cor}

\begin{proof}
The case $w = w_0$ is in Theorem \ref{Qnorm-w0}. Assume that the assertion holds for $w$. Let $i \in \tI$ such that $s_i w < w$. Then, Corollary \ref{BSDH-typeQ} implies that $\cO_{\sQ ( \beta, s_i w )}$ is isomorphic to the normal sheaf of rings $(\pi_i)_* \cO_{P_i \times^B \sQ ( \beta, w )}$. Hence, the assertion holds for $s_i w < w$. This proceeds the induction and we conclude the result.
\end{proof}

\begin{cor}\label{Bconn+}
For each $\beta \in Q^{\vee}_+$ and $w, v \in W$, the subspace
$$\mathtt{e}_1^{-1} ( p_v ) \cap \mathtt{e}_2^{-1} ( \sB ( w ) ) \subset \sB_{2,\beta}$$
is connected.
\end{cor}

\begin{proof}
This space appears in the fiber of $\pi_{\beta, w_0, w}$ along the constant quasimap $\P^1 \rightarrow \{ p_v \} \subset \sB$ with its defect concentrated in $0 \in\P^1$. Hence, the assertion follows from Corollary \ref{Qnorm-w} and Proposition \ref{conn-to-norm}.
\end{proof}

\begin{thm}\label{Qnorm}
For each $\beta \in Q^{\vee}_+$ and $w, v \in W$, the scheme $\sQ ( \beta, v, w )$ is normal.
\end{thm}

\begin{proof}
The combination of Theorem \ref{desc-fib} and Corollary \ref{Bconn+} implies that every fiber of $\pi_{\beta, v, w}$ is connected. Thus, Proposition \ref{conn-to-norm} implies the result.
\end{proof}

\begin{cor}\label{gen-norm}
For each $\tJ \subset \tI$, $\beta \in Q^{\vee}_+$ and $w,v \in W$, the variety $\sQ_{\tJ} ( \beta, v, w )$ is irreducible and normal.
\end{cor}

\begin{proof}
By Corollary \ref{H-inh} and Remark \ref{cvw}, we can rearrange the map $\Pi_\tJ$ to be surjective with connected fibers. Hence, Proposition \ref{irrQ0} implies the irreducibility of $\sQ_{\tJ} ( \beta, v, w )$. By the Stain factorization theorem applied to the composition map $\Pi_{\tJ} \circ \pi_{\beta, v, w}$, we deduce that the normalization of $\sQ_{\tJ} ( \beta, v, w )$ is bijective to $\sQ_{\tJ} ( \beta, v, w )$. Therefore, the weak normality of $\sQ_{\tJ} ( \beta, v, w )$ (Corollary \ref{wn-char0}) implies the result.
\end{proof}

We set $W^{\tJ} := W / w_0 W_{\tJ} w_0$, and identify it with the set of their minimal length representatives in $W$. Recall that $2 \rho_{\tJ}$ is the sum of the positive roots that belong to the unipotent radical of $\mathfrak p ( \tJ )$.

\begin{cor}\label{dim-general}
For each $\tJ \subset \tI$, $\beta \in Q^{\vee}_+$ such that $-w_0 \beta \in Q^{\vee}_{\tJ, +}$, $w \in W^{\tJ}$, and $v \in W$ such that $w_0 v \in W^\tJ$, the variety $\sQ_{\tJ} ( \beta, v, w )$ has dimension
\begin{equation}
\dim \, G / P ( \tJ ) - 2 \left< w_0 \beta, \rho_{\tJ} \right> + \ell ( v ) - \ell (w_0)  - \ell ( w )\label{par-dim}
\end{equation}
if it is nonempty. In case $v = w_0$, the variety $\sQ_{\tJ} ( \beta, v, w )$ is always nonempty. In addition, $\sQ_{\tJ} ( \beta, w_0, e )$ is the collection of DP-data as in Definition {\rm\ref{Zas}} with the data and compatibility conditions only for $\la \in P_{\tJ, +}$.
\end{cor}

\begin{rem}
If we replace the labels of $\sQ_{\tJ} ( \beta, v, w ) = \sQ'_{\tJ} ( vt_{\beta}, w )_{\C}$ with the labels of parabolic semi-infinite Bruhat order as in \cite[\S 10.3]{LS10}, then we can write the condition of $\sQ_{\tJ} ( \beta, v, w )$ to be nonempty by the order relation just as in Corollary \ref{Q'nonempty}.
\end{rem}

\begin{proof}[Proof of Corollary \ref{dim-general}]
By \cite{BFGM}, the collection of DP-data in Definition \ref{Zas} with the compatibility condition imposed only for $\la \in P_{\tJ, +}$ (without non-defining points) represents the (closure of the) space of maps $\P^1 \rightarrow G / P ( \tJ )$ such that the image of $[\P^1]$ in $\sB$ is $- w_0 \beta \in H_2 ( \sB, \Z )$. Thus, we conclude the last assertion.

The dimension of $\sQ_{\tJ} ( \beta, w_0, e )$ is given by $\dim \, G / P ( \tJ ) - 2 \left< w_0 \beta, \rho_{\tJ} \right>$ since the first Chern class of the tangent bundle of $G / P ( \tJ )$ is $2 \rho_{\tJ}$ (cf. \cite[Proposition 3.5]{FM99}). In particular, the dimension of $\sQ_{\tJ} ( \beta, w_0, e )$ is (\ref{par-dim}) in this case.

Since imposing a hyperplane section in the image of $G/P(\tJ)$ through the (generically defined) evaluation map at $0 \in \P^1$ (that is surjective as being $G$-equivariant)
$$\sQ_{\tJ} ( \beta, w_0, e ) \dashrightarrow G/P(\tJ)$$
lowers the dimension exactly by one, we deduce that $\sQ_{\tJ} ( \beta, w_0, w )$ has an irreducible component of expected dimension. From this and Corollary \ref{gen-norm}, we conclude that $\sQ_{\tJ} ( \beta, w_0, w )$ is nonempty and its dimension is given by (\ref{par-dim}).

We consider the dimension estimate in the general case.

Since $\sQ_{\tJ} ( \beta, s_i v, w ) \subsetneq \sQ_{\tJ} ( \beta, v, w )$ ($i \in \tI$) implies $s_i v >_{\si} v$ ($\Leftrightarrow s_i v < v$) when $w_0 s_i v \in W^\tJ$, we find that they are the images of $\sQ ( \beta', s_i v, w ) \subsetneq \sQ ( \beta', v, w )$ for some $\beta' \in \beta + \sum_{j \in \tJ} \Z_{\ge 0} \al_j^{\vee}$ (see Remark \ref{cvw}). In particular, $\sQ_{\tJ} ( \beta, v, w )$ has dimension at least (\ref{par-dim}) in general by induction from the case $v = w_0$.

Let $Z$ denote the space of genus zero stable maps to $\P^1 \times G/P ( \tJ )$ of bidegree $(1,\beta)$ with two marked points such that the first marked points and second marked points correspond to $0, \infty \in \P^1$, respectively (parabolic analogue of $\sGB^\flat_{2,\beta}$). Then, we have a resolution $\pi : Z \rightarrow \sQ _\tJ ( \beta, w_0, e )$ (that shares genuine maps $\P^1 \to G/P(\tJ)$ in their common open dense subset). Let $\eta : \sB = G / B \rightarrow G/P(\tJ)$ be the projection. By our choices of $v,w$, we find that
$$\mathrm{codim} _\sB \, \sB ( w ) = \mathrm{codim} _{G/P(\tJ)} \, \eta ( \sB ( w ) ) \hskip 3mm \text{and} \hskip 3mm \mathrm{codim} _\sB \, \sB^{\mathrm{op}} ( v ) = \mathrm{codim} _{G/P(\tJ)} \, \eta ( \sB^{\mathrm{op}} ( v ) ),$$
respectively. We have the evaluation map of two points $\mathtt{z} : Z \rightarrow ( G/P(\tJ))^2$. We set $Z_{v,w} := \mathtt{z}^{-1} ( \eta ( \sB ( w )) \times \eta ( \sB^{\mathrm{op}} ( v )))$. By \cite[Propostion 3.2 b)]{BCMP} (see also the proof of Corollary \ref{cBCMP}), we find that
\begin{align*}
\dim \, Z_{v,w} & = \dim \, Z - \mathrm{codim} _{G/P(\tJ)} \, \eta ( \sB ( w ) ) - \mathrm{codim} _{G/P(\tJ)} \, \eta ( \sB^{\mathrm{op}} ( v ) )\\
& = \dim \, Z - \mathrm{codim} _\sB \, \sB ( w ) - \mathrm{codim} _\sB \, \sB^{\mathrm{op}} ( v ) = (\ref{par-dim})
\end{align*}
and $Z_{v,w}$ is irreducible (if it is non-empty). Since $G / P ( \tJ )$ is homogeneous, we find that $G \cdot \pi ( Z_{w_0,w} ) \subset Z$ is Zariski dense. Hence, $Z_{w_0,w}$ contains an open dense subset whose points are genuine maps $\P^1 \to G/P (\tJ)$. In conjunction with the irreducibility of $Z_{w_0,w}$ and $\sQ _{\tJ} (\beta,w_0,w )$, we find $\pi^{-1} ( \sQ _{\tJ} (\beta,w_0,w ) ) = Z_{w_0,w}$. Similarly, we have $\pi^{-1} ( \sQ _{\tJ} (\beta,v,e ) ) = Z_{v,e}$. Therefore, we conclude that $Z_{v,w}$ surjects onto $\sQ _{\tJ} (\beta,v,w )$ by the restriction of $\pi$. In particular, $\sQ_{\tJ} ( \beta, v, w )$ has dimension at most (\ref{par-dim}) in general.

Therefore, our dimension estimate must be strict in general.
\end{proof}

\begin{cor}\label{ggnormal}
For each $\tJ \subset \tI$ and $w,v \in W_\af$, the scheme $\sQ'_{\tJ} ( v, w )_{\overline{\F}_p}$ is irreducible and normal for $p \gg 0$.
\end{cor}

\begin{proof}
Apply Proposition \ref{n-char0} to Corollary \ref{gen-norm} (cf. Lemma \ref{trans}).
\end{proof}

\begin{cor}\label{ggpnormal}
Let $\bK$ be an algebraically closed field of characteristic $0$ or $p \gg 0$. For each $\tJ \subset \tI$ and $w,v \in W_\af$, the scheme $\sQ'_{\tJ} ( v, w )_{\bK}$ is projectively normal with respect to a line bundle $\cO_{\sQ'_{\tJ} ( v, w )_{\bK}} ( \la )$ $(\la \in P_{\tJ, ++})$.
\end{cor}

\begin{proof}
In view of Lemma \ref{Qsurj-mult}, Lemma \ref{Q-flat}, and Theorem \ref{coh}, the multiplication of the section ring afforded by $\{ \cO_{\sQ'_{\tJ} ( v, w )_{\bK}} ( m \la ) \}_{m \ge 0}$ is surjective. Therefore, Corollary \ref{gen-norm} and Corollary \ref{ggnormal} implies the result.
\end{proof}

\appendix
\renewcommand{\thesection}{\Alph{section}}
\setcounter{section}{1}
\renewcommand{\theequation}{\Alph{section}.\arabic{equation}}
\setcounter{equation}{0}
\setcounter{thm}{0}

{\small
\begin{flushleft}
{\normalsize\textbf{Appendix A \hskip 2mm Some properties of $\bQ_G ( w )$ in positive characteristic}}
\end{flushleft}

We work in the setting of \S \ref{sec:mi}. Let $\bK$ be an algebraically closed field of characteristic $\neq 2$. The aim of this appendix is two fold: One is to show that our scheme $\bQ_{G, \tJ} ( w )_\bK$ is (projectively) normal. The other is to present an analogue of the Kempf vanishing theorem for $\bQ_G ( w )_\bK$. The both results are proved in \cite{KNS17} for the case $\bK = \C$. Our proof of the former is new and does not depend on \cite{KNS17} (when $w = e$), while our proof of the latter depends on the results and arguments in \cite{KNS17} in an essential way.

\begin{prop}\label{R-normal}
For each $w \in W_\af$ and $\tJ \subset \tI$, the ring $R_w ( \tJ )_{\bK}$ is normal.
\end{prop}

\begin{proof}
We first prove the case $w = e$ and $\tJ = \emptyset$. Let $\mathring{\bQ}_G$ denote the open $G [\![z]\!]$-orbit of $\bQ_G ( e )_\bK$ obtained by the $G$-translation of Lemma \ref{bQ-dense}. 

We have an inclusion
\begin{equation}
\bW ( \la )^{\vee}_\bK \subset \Gamma ( \mathring{\bQ}_G, \cO_{\bQ_G (e)_\bK} ( \la ))\label{nat-incl}
\end{equation}
since $R _{\bK}$ is integral (Corollary \ref{Re-int}). We also have an inclusion
$$\Gamma ( \mathring{\bQ}_G, \cO_{\bQ_G (e)_\bK} ( \la )) \hookrightarrow \Gamma ( \bO ( e )_\bK, \cO_{\bQ_G (e)_\bK} ( \la )) \hskip 5mm \la \in P_+.$$
Thanks to Lemma \ref{bQ-dense} and its proof, we deduce
$$\Gamma ( \bO ( e )_\bK, \cO_{\bQ_G (e)_\bK} ( \la )) \cong \bK [\bI / \left( \bI \cap H N (\!(z)\!) \right)] \otimes_{\bK} \bK_{-w_0\la},$$
that is cocyclic as a $U^{+}_{\bK}$-module. Since the $G$-action on $\Gamma ( \mathring{\bQ}_G, \cO_{\bQ_G (e)_\bK} ( \la ))$ is algebraic, we deduce that
$$( \bK [\bI / \left( \bI \cap H N (\!(z)\!) \right)] \otimes_{\bK} \bK_{-w_0\la} ) ^{\vee} \longrightarrow \!\!\!\!\! \rightarrow \Gamma ( \mathring{\bQ}_G, \cO_{\bQ_G (e)_\bK} ( \la ))^{\vee}$$
is a $\dot{U}_\bK^0$-integrable quotient. By Proposition \ref{proj-cover}, we conclude a surjection
$$\bW ( \la )_{\bK} \longrightarrow \!\!\!\!\! \rightarrow \Gamma ( \mathring{\bQ}_G, \cO_{\bQ_G (e)_\bK} ( \la ))^{\vee}.$$
Compared with (\ref{nat-incl}), we conclude the isomorphism
$$\Gamma ( \mathring{\bQ}_G, \cO_{\bQ_G (e)_\bK} ( \la )) \stackrel{\cong}{\longrightarrow} \Gamma ( \bQ_G ( e )_\bK, \cO_{\bQ_G (e)_\bK} ( \la )) \hskip 5mm \la \in P_+$$
since the space of sections supported on a dense open Zarsiki subset must be larger than (or equal to) the space of sections supported on the whole space for an integral scheme. In other words, $R _{\bK}$ is the maximal $\dot{U}_\bK^0$-integrable $U^+_{\bK}$-stable subring of
$$S_{\bK} := \bigoplus_{\la \in P_+} \Gamma ( \bO ( e )_\bK, \cO_{\bQ_G (e)_\bK} ( \la )) = \bigoplus_{\la \in P_+} \bK [\bI / \left( \bI \cap H N (\!(z)\!) \right)] \otimes_{\bK} \bK_{-w_0\la}.$$
The ring $S_{\bK}$ is integrally closed as it is a polynomial ring (of countably many variables). Hence, the integral closure $R^+$ of the ring $R _{\bK}$ is also a subring of $S_{\bK}$. In view of the $\dot{U}_\bK^0$-integrability, the ring $R_{\bK}$ admits an algebraic $G$-action. By the canonical nature of the normalization, we deduce that $R^+$ admits an algebraic $G$-action (note that we can approximate $R_{\bK}$ as a union of Noetherian rings with algebraic $G$-actions). However, $R_{\bK}$ is already the maximal $\dot{U}_\bK^0$-integrable $U^+_{\bK}$-stable subring of $S_{\bK}$ from the above. Hence, we have necessarily $R^+ = R_{\bK}$. Thus, the case of $w = e$ and $\tJ = \emptyset$ follows.

The case of arbitrary $w \in W_{\af}$ follows from the case of $w = e$ as in \cite[\S 4]{Kat18}. As the fraction field of $R_w ( \tJ )_{\bK}$ is a subfield of the fraction field of $( R_w )_{\bK}$ and $P_{\tJ, +} \subset P_+$ forms (the set of integral points of) a face, we conclude the general case by restriction.
\end{proof}

\begin{cor}\label{bQnormal}
The scheme $\bQ _{G, \tJ} ( w )_{\bK}$ is normal. \hfill $\Box$
\end{cor}

\begin{cor}\label{gsect}
For each $\la \in P_{\tJ,+}$, we have
$$H^0 ( \bQ_{G, \tJ} ( w )_{\bK}, \cO_{\bQ_{G, \tJ} ( w )_{\bK}} ( \la ) ) = \bW_{ww_0}  ( \la )_{\bK} ^{\vee}.$$
\end{cor}

\begin{proof}
As finding $H^0$ from $R_w  ( \tJ )_{\bK}$ can be seen as finding graded pieces of the normalization, the assertion follows from Proposition \ref{R-normal}.
\end{proof}

\begin{prop}[\cite{KNS17} Proposition 5.1]
Assume that $\mathsf{char} \, \bK \neq 2$. Let $w \in W$ and $\tJ \subset \tI$. Then, an $\bI$-equivariant line bundle on $\bQ _{G, \tJ} ( w )_{\bK}$ is a character twist of $\{\cO_{\bQ _{G, \tJ} ( w )_{\bK}} ( \la )\}_{\la \in P_{\tJ}}$.
\end{prop}

\begin{proof}
Taking Corollary \ref{pos-fp} and Corollary \ref{bQnormal} into account, the proof in \cite[Proposition 5.1]{KNS17} works in this setting.
\end{proof}

\begin{thm}
Let $w \in W$ and $\tJ \subset \tI$. For each $\la \in P_{\tJ}$, we have
$$H^{> 0} ( \bQ_{G, \tJ} ( w )_{\bK}, \cO_{\bQ_{G, \tJ} ( w )_{\bK}} ( \la ) ) = \{ 0 \} .$$
\end{thm}

\begin{proof}
Assume that $w = e$. We consider the subring defined by
$$T_{\tJ} := \bigoplus_{\la \in P_{\tJ, +}} T_{\tJ} (\la) \subset \bigoplus_{\la \in P_{\tJ, +}} \bW ( \la )_{\bK}^{\vee} = \bigoplus_{\la \in P_{\tJ, +}} \bW_{w_0} ( \la )_{\bK}^{\vee} = R_{\tJ},$$
where $T_{\tJ} (\la) \subset \bW ( \la )_{\bK}^{\vee}$ is the $P$-weight $( - \la )$-part of $\bW ( \la )_{\bK}^{\vee}$. They are generated by the duals of $\widetilde{P}_{i,m\delta} \bv_{\varpi_i} \in \bW ( \varpi_i )_{\bK}$ ($i \in \tI \setminus \tJ, m \in \Z_{>0}$; see \cite[(3.18)]{BN04} for the definition of $\widetilde{P}_{i,m\delta}$). In view of \cite[Proposition 3.18]{BN04}, the ring $T_{\tJ}$ is a polynomial ring. The actions of $\{\widetilde{P}_{i,m\delta}\}_{i \in ( \tI \setminus \tJ ), m > 0}$ define endomorphisms of the module $\bW ( \la )_{\bK}$ ($\la \in P_{\tJ,+}$), each of which is either an injection or zero (cf. the proof of Proposition \ref{proj-cover}). It follows that the action of $T_{\tJ}$ on $R_{\tJ}$ is torsion-free.

Thanks to \cite[Corollary 4.29]{KNS17}, the character comparison forces the torsion-free action of $T_{\tJ}$ on $R_{\tJ}$ to be free. Therefore, the proof of \cite[Theorem 4.30]{KNS17} works in a verbatim way. This proves the case $w = e$. The general case can be easily deduced from the $w = e$ case as in \cite[Corollary 4.31]{KNS17}.
\end{proof}

\setcounter{section}{2}
\setcounter{thm}{0}

\begin{flushleft}
{\normalsize\textbf{Appendix B \hskip 2mm An application of the Pieri-Chevalley formula}}
\end{flushleft}

We work in the setting of \S \ref{Q'Jvw}. The aim of this appendix is to present a method (Theorem \ref{SMT}) to describe the global sections of nef line bundles on $\sQ' ( v,w )_{\bK}$ for $v,w \in W_\af$ and an algebraically closed field $\bK$ of characteristic $\neq 2$.

For each $\mu \in P_{+}$, we have an extremal weight module $\bX ( \mu )_{\bK}$ and its global base $\bB ( \bX ( \mu ) )$ borrowed from Theorem \ref{b-compat}.

\begin{lem}\label{H-compat}
Let $\tJ \subset \tI$, and let $\mu \in P_{\tJ, +}$. A subset of $\bB (  \bX ( \mu ) )$ spans
\begin{align*}
H^{0} ( \bQ_{G} ( w )_{\bK}, \cO_{\bQ_{G} ( w )_{\bK}} ( \mu ) )^{\vee} = & \, H^{0} ( \bQ_{G, \tJ} ( w )_{\bK}, \cO_{\bQ_{G, \tJ} ( w )_{\bK}} ( \mu ) )^{\vee} \subset \bX ( \mu )_{\bK} \hskip 5mm \text{and}\\
\theta ( H^{0} ( \bQ_{G} ( vw_0 )_{\bK}, \cO_{\bQ_{G} ( vw_0 )_{\bK}} & \, ( -w_0\mu ) )^{\vee} )\\
= \theta ( & H^{0} ( \bQ_{G, \theta ( \tJ )} ( vw_0 )_{\bK}, \cO_{\bQ_{G, \theta ( \tJ )} ( vw_0 )_{\bK}} ( - w_0 \mu ) )^{\vee} ) \subset \bX ( \mu )_{\bK}
\end{align*}
for each $w,v \in W_\af$.
\end{lem}

\begin{proof}
Taking Corollary \ref{gsect} into account, the assertion follows from Theorem \ref{b-compat}.
\end{proof}

Corollary \ref{gsect}, combined with Theorem \ref{coh}, yields
\begin{equation}
H^{0} ( \sQ' ( v, w )_{\bK}, \cO_{\sQ' ( v, w )_{\bK}} ( \mu ) )^{\vee} \hookrightarrow H^{0} ( \bQ_{G} ( w )_{\bK}, \cO_{\bQ_{G} ( w )_{\bK}} ( \mu ) )^{\vee} \subset \bX ( \mu )_{\bK}\label{filt-piece}
\end{equation}
for each $w,v \in W_\af$.

\begin{lem}\label{init-fin}
Let $\mu \in P_{++}$. For each $b \in \bB (  \bX ( \mu ) )$, there exist unique elements $\kappa ( b ), \imath ( b ) \in W_\af$ with the following properties:
\begin{enumerate}
\item We have $b \in H^{0} ( \bQ_{G} ( w )_{\bK}, \cO_{\bQ_{G} ( w )_{\bK}} ( \mu ) )^{\vee}$ for $w \in W_\af$ if and only if $\kappa ( b ) \le_\si w$;
\item We have $b \in \theta ( H^{0} ( \bQ_{G} ( vw_0 )_{\bK}, \cO_{\bQ_{G} ( vw_0 )_{\bK}} ( -w_0 \mu ) )^{\vee} )$ for $v \in W_\af$ if and only if $\imath ( b ) \ge_\si v$.
\end{enumerate}
\end{lem}

\begin{proof}
In view of Corollary \ref{gsect}, the first assertion is a rephrasement of Lemma \ref{contain}. The second assertion is obtained from the first assertion in view of Lemma \ref{WX-comm} and (\ref{ord-opp}).
\end{proof}

\begin{cor}
The functions $\kappa$ and $\imath$ play the same role as the same named functions in {\rm\cite[(2.17)]{KNS17}} $($with opposite convention on the order $\le_\si)$.
\end{cor}

\begin{proof}
Compare Lemma \ref{init-fin} with \cite[Theorem 2.8]{KNS17}.
\end{proof}

\begin{cor}\label{Qga}
Let $\mu \in P_{++}$ and $v,w \in W_\af$. The space $H^{0} ( \sQ' ( v, w )_{\bK}, \cO_{\sQ' ( v, w )_{\bK}} ( \mu ) )^{\vee}$ is spanned by the subset of $b \in \bB ( \bX ( \mu ) )$ that satisfies $w \ge_\si \kappa ( b ) \ge_\si \imath ( b ) \ge_\si v$.
\end{cor}

\begin{proof}
This is a rephrasement of Theorem \ref{coh} for $\tJ = \emptyset$.
\end{proof}

\begin{thm}[\cite{KNS17} Theorem 5.8 and its proof]\label{PCKNS}
Let $w \in W_\af$ and $\la \in P_+$. There exists a unique collection of elements $a^u_w ( \la ) \in \Z [q^{-1}][H]$ $(u \in W_\af)$ such that
$$a^w_w ( \la ) = e^{- w w_0\la}, \hskip 6mm a^u_w ( \la ) = 0  \hskip 3mm\text{if} \hskip 3mm u \not\le_\si w,$$
and
$$\gch \, \Gamma ( \bQ_G ( w )_{\bK}, \cO_{\bQ_G ( w )_{\bK}} ( \la + \mu ) ) = \sum_{u \in W_\af} a^u_w ( \la ) \gch \, \Gamma ( \bQ_G ( u )_{\bK}, \cO_{\bQ_G ( u )_{\bK}} ( \mu ) )$$
for every $\mu \in P_{++}$. \hfill $\Box$
\end{thm}

The goal of this appendix is to prove the following:

\begin{thm}\label{SMT}
Let $w,v \in W_\af$ and $\la \in P_+$. Let $\{ a^v_w ( \la ) \}_{v \in W_\af}$ be the collection of elements in Theorem {\rm\ref{PCKNS}}. Then, we have
$$\gch \, \Gamma ( \sQ' ( v, w )_{\bK}, \cO_{\sQ' ( v, w )_{\bK}} ( \la ) ) = \sum_{u \ge_\si v} a^u_w ( \la ).$$
\end{thm}

\begin{proof}
The proof of the numerical part of Theorem \ref{PCKNS} is \cite[Theorem 3.5]{KNS17} (proved in \cite[\S 8.1]{KNS17}) and it counts the elements of $\bB ( \bX ( \la + \mu ) )$ that contributes $\Gamma ( \bQ_G ( w )_{\bK}, \cO_{\bQ_G ( w )_{\bK}} ( \la + \mu ) )^{\vee} = \bW _{ww_0} ( \la + \mu )_{\bK}$ in two ways. In particular, we can additionally impose the condition $\imath ( \bullet ) \ge_\si v$ for $\bB (  \bX ( \la + \mu ) )$ to deduce that
$$\gch \, \Gamma ( \sQ' ( v, w )_{\bK}, \cO_{\sQ' ( v, w )_{\bK}} ( \la + \mu ) ) = \sum_{u \le_\si v} a^u_w ( \la ) \gch \, \Gamma ( \sQ' ( u, w )_{\bK}, \cO_{\sQ' ( u, w )_{\bK}} ( \mu ) )$$
for every $\mu \in P_{++}$ in view of Corollary \ref{Qga}. The ($H \times \Gm$-equivariant) Euler characteristic of $\cO_{\sQ' ( v, w )_{\bK}} ( \mu )$ ($\mu \in P$) is a rational function on the characters of $H$, and we can specialize to $\mu = 0$. Now we apply Theorem \ref{coh} to deduce
$$\chi ( \sQ' ( v, w )_{\bK}, \cO_{\sQ' ( v, w )_{\bK}} ( \mu ) ) = \gch \, \Gamma ( \sQ' ( v, w )_{\bK}, \cO_{\sQ' ( v, w )_{\bK}} ( \mu ) )$$
for every $\mu \in P_+$. This implies the desired equality.
\end{proof}

\begin{rem}
{\bf 1)} In view of Corollary \ref{H-inh} and Remark \ref{cvw}, Theorem \ref{SMT} describes the space of global sections of $\cO_{\sQ'_{\tJ} ( v, w )_{\bK}} ( \la )$ for every $w,v \in W_\af$, $\tJ \subset \tI$, and $\la \in P_{\tJ,+}$. {\bf 2)} In conjunction with Theorem \ref{coh}, Theorem \ref{SMT} can be seen as an analogue of Lakshmibai-Littelmann \cite[Theorem 34]{LL03} for semi-infinite flag manifolds. {\bf 3)} Thanks to \cite[\S 3]{KNS17}, we have a combinatorial rule to express $a^u_w ( \la )$'s.
\end{rem}

\begin{center}
{\bf Acknowledgments}
\end{center}
The author would like to thank Prakash Belkale, Ievgen Makedonskyi, Leonardo Mihalcea, Katsuyuki Naoi, and Daisuke Sagaki for helpful discussions. This work was supported in part by JSPS KAKENHI Grant Number JP26287004 and JP19H01782.\\

{\footnotesize
\bibliography{kmref}
\bibliographystyle{hplain}}
\end{document}